\RequirePackage{fix-cm}
\documentclass[envcountsect,smallextended,a4paper]{svjour3}       
\smartqed
\usepackage{graphicx}
\usepackage[utf8]{inputenc}
\usepackage[T1]{fontenc}
\usepackage{lmodern}
\usepackage{amsmath}
\usepackage{amssymb}
\usepackage{marginnote}
\usepackage{xcolor}
\usepackage{fancyhdr}
\usepackage{esint}
\usepackage{tikz} 
\usepackage[arrow, matrix, curve]{xy}
\usepackage{hyperref}
\usetikzlibrary{fit,positioning} 
\usepackage{geometry}
\usepackage{pifont}
\usepackage{nccmath}
\usepackage{tocloft}
\usepackage{setspace}
\usepackage{graphicx}
\usepackage{calligra}
\usepackage[english]{babel}
\usepackage{amssymb}
\usepackage{mathtools}
\usepackage{blindtext}
\usepackage[justification=centering,labelfont=bf,belowskip=-10pt,aboveskip=3pt,font=small]{caption}
\usepackage[rose,novargreek,operator]{paper_diening} 
\usepackage{mathrsfs}  

\newcommand{\vertiii}[1]{{\vert\kern-0.25ex\vert\kern-0.25ex\vert #1 
		\vert\kern-0.25ex\vert\kern-0.25ex\vert}}
	
	\usepackage{color}
	\definecolor{Darkgreen}{rgb}{0,0.4,0}
	\usepackage{hyperref}
	\hypersetup{%
		pdfborder = {0 0 0},
		colorlinks,
		citecolor=blue,
		filecolor=black,
		linkcolor=blue,
		urlcolor=cyan!50!black!90
	}
\DeclareMathAlphabet\bscal{OMS}{cmsy}{b}{n}
\DeclareMathAlphabet\mathbfscr{OMS}{mdugm}{b}{n}

\spnewtheorem{defn}[equation]{Definition}{\bfseries}{\upshape}
\spnewtheorem{prop}[equation]{Proposition}{\bfseries}{\upshape}
\spnewtheorem{thm}[equation]{Theorem}{\bfseries}{\upshape}
\spnewtheorem{cor}[equation]{Corollary}{\bfseries}{\upshape}
\spnewtheorem{rmk}[equation]{Remark}{\bfseries}{\upshape}
\spnewtheorem{lem}[equation]{Lemma}{\bfseries}{\upshape}
\spnewtheorem{expl}[equation]{Example}{\bfseries}{\upshape}

\numberwithin{equation}{section} 
\pdfoutput=1
\begin{document}
	
		\title{Pseudo-monotone operator theory for unsteady problems in variable exponent spaces
	}
	
	\author{A. Kaltenbach}

	\institute{A. Kaltenbach \at
		Institute of Applied Mathematics, Albert--Ludwigs--University Freiburg, Ernst--Zermelo--Straße 1, 79104 Freiburg,\\
		\email{alex.kaltenbach@mathematik.uni-freiburg.de}           
	}
	
	\date{Received: date / Accepted: date}

	\maketitle
	
	\begin{abstract}
		We 
		prove by means of advanced pseudo-monotonicity methods an abstract existence result
		for parabolic partial differential equations with  $\log$--Hölder continuous variable exponent nonlinearity governed by the symmetric part of a gradient only.  To this end, we introduce the notions Bochner pseudo-monotonicity
		and Bochner coercivity, which are appropriate extensions of the concepts of pseudo-monotonicity and coercivity to unsteady problems in variable exponent spaces. In this context, we apply the so-called Hirano--Landes approach, which enables us to give general and easily verifiable 
		conditions for these new notions. 
		Moreover, we prove essential parabolic embedding  and compactness results involving only the symmetric part of the gradient.
		
		\keywords{Variable exponent spaces\and Poincar\'e
			inequality\and Pseudo-monotone operator\and Symmetric gradient \and Nonlinear parabolic problem \and Existence
			of weak solutions}
		\subclass{46E35, 35K55, 47H05} 
	\end{abstract}
	
	\section{Introduction}
	Let $\Omega\subseteq \setR^d$, $d\ge 2$, be a bounded 
	Lipschitz domain, $I:=\left(0,T\right)$, $T<\infty$, a 
	bounded time interval, $Q_T:=I\times \Omega$ a cylinder and $\Gamma_T:=I\times\pa\Omega$. 
	We consider as a model problem the system
	\begin{align}
	\begin{split}
	\begin{alignedat}{2}
	\pa_t \bsu-\divo(\bfS(\cdot,\cdot,\bfvarepsilon(\bsu)))
	+\bfb(\cdot,\cdot,\bsu)&
	=\bsf-\divo(\bsF)&&\quad\text{ in }Q_T,\\
	\bsu&=0&&\quad\text{ on }\Gamma_T,\\
	\bsu(0)&=\bfu_0&&\quad\text{ in }\Omega.
	\end{alignedat}
	\end{split}\label{eq:model}
	\end{align}
	Here, $\bsf:Q_T\to \setR^d$ is a given vector field, 
	$\bsF:Q_T\to \mathbb{M}^{d\times d}_{\sym}$ a given tensor field\footnote{Here, $\mathbb{M}^{d\times d}_{\sym}:=\{\mathbf{A}\in \mathbb{R}^{d\times d}\mid  \mathbf{A}^\top=\mathbf{A}\}$ is equipped~with~the~Frobenius~scalar~product~${\bfA:\bfB:=\sum_{i,j=1}^{d}{\textrm{A}_{ij}\textrm{B}_{ij}}}$\\[-2pt] 
		and the Frobenius norm $\vert\bfA\vert:=(\bfA:\bfA)^{\frac{1}{2}}$. 
		By $\bfa\cdot\bfb:=\sum_{i=1}^{d}{\textrm{a}_{i}\textrm{b}_{i}}$ we denote the Euclidean scalar product in $\setR^d$\\[-1pt]  and 
		by $\vert \bfa\vert:=(\bfa\cdot\bfa)^{\frac{1}{2}}$  the Euclidean norm.}, 
	$\bfu_0:\Omega\to \setR^d$ an initial condition, and 
	$\bfvarepsilon(\bfu):=\frac{1}{2}(\nb \bfu+\nb \bfu^\top)$ 
	denotes the symmetric part of the gradient. Moreover, let $p:Q_T\to\left[1,+\infty\right)$
	be a globally $\log$--Hölder continuous exponent with
	\begin{align*}
	1< p^-:=\inf_{(t,x)^\top\in Q_T}{p(t,x)}\leq p^+:=\sup_{(t,x)^\top\in Q_T}{p(t,x)}<\infty.
	\end{align*}
	The  tensor-valued mapping $\bfS:Q_T\times \mathbb{M}^{d\times d}_{\sym}\to 
	\mathbb{M}^{d\times d}_{\sym}$ is supposed to possess $(p(\cdot,\cdot),\delta)$--\textit{structure}, i.e., for some exponent $p:Q_T\to\left[1,+\infty\right)$ and some $\delta\ge 0$ the following properties are satisfied:
	\begin{description}[{\textbf{(S.3)}}]
		\item[\textbf{(S.1)}]\hypertarget{S.1}{} $\bfS:Q_T\times \mathbb{M}^{d\times d}_{\sym}\to 
		\mathbb{M}^{d\times d}_{\sym}$ is a Carath\'eodory mapping\footnote{$\bfS(\cdot,\cdot,\bfA):Q_T\to \mathbb{M}^{d\times d}_{\sym}$ is Lebesgue measurable for every $\bfA\in \mathbb{M}^{d\times d}_{\sym}$ and $\bfS(t,x,\cdot):\mathbb{M}^{d\times d}_{\sym}\to \mathbb{M}^{d\times d}_{\sym}$ is continuous for almost every $(t,x)^\top\in Q_T$.}.
		\item[\textbf{(S.2)}]\hypertarget{S.2}{} $\vert\bfS(t,x,\bfA)\vert \leq \alpha(\delta+\vert \bfA\vert)^{p(t,x)-2}\vert\bfA\vert+\beta(t,x)$ 
		for every $\bfA\in \mathbb{M}^{d\times d}_{\sym}$, a.e. $(t,x)^\top\in Q_T$ \newline
		$(\alpha\ge 1,\;\beta\in L^{p'(\cdot,\cdot)}(Q_T,\setR_{\ge 0}))$.
		\item[\textbf{(S.3)}]\hypertarget{S.3}{} $\bfS(t,x,\bfA):\bfA\ge 
		c_0(\delta+\vert \bfA\vert)^{p(t,x)-2}\vert\bfA\vert^2-c_1(t,x)$ 
		for every $\bfA\in \mathbb{M}^{d\times d}_{\sym}$, a.e. $(t,x)^\top\in Q_T$ \newline
		$(c_0>0,\;c_1\in L^1(Q_T,\setR_{\ge 0}))$.
		\item[\textbf{(S.4)}]\hypertarget{S.4}{} $(\bfS(t,x,\bfA)-\bfS(t,x,\bfB)):
		(\bfA-\bfB)\ge 0$ for every
		$\bfA,\bfB\in \mathbb{M}^{d\times d}_{\sym}$, a.e. $(t,x)^\top\in Q_T$.
	\end{description}
	The  vector-valued mapping $\bfb:Q_T\times \setR^d\to \setR^d$ is supposed
	to satisfy:
	\begin{description}[{\textbf{(B.3)}}]
		\item[\textbf{(B.1)}]\hypertarget{B.1}{} $\bfb:Q_T\times \setR^d\to \setR^d$ is a 
		Carath\'eodory mapping.
		\item[\textbf{(B.2)}]\hypertarget{B.2}{} $\vert\bfb(t,x,\bfa)\vert \leq 
		\gamma(1+\vert \bfa\vert)^{r(t,x)-1}+\eta(t,x)$ for every
		$\bfa\in \setR^d$, a.e. $(t,x)^\top\in Q_T$, where \linebreak${r:=\max\{2,p_*\}-\vep\in \mathcal{P}^{\log}(Q_T)}$ for some $\vep \in (0,(p^-)_*-1]$ and $(\cdot)_*\in W^{1,\infty}(1,\infty)$ is for all $s\in (1,\infty)$ given via $s_*:=s\frac{d+2}{d}$ if $s<d$ and $s_*:=s+2$ if $s\ge d$ (cf.~Definition~\ref{5.7})\newline
		$(\gamma\ge 1,\;\eta\in L^{r'(\cdot,\cdot)}(Q_T,\setR_{\ge 0}))$. 
		\item[\textbf{(B.3)}]\hypertarget{B.3}{} $\bfb(t,x,\bfa)\cdot\bfa\ge c_2\vert \bfa\vert^2-c_3(t,x)$ 
		for every $\bfa\in \setR^d$,~a.e.~${(t,x)^\top\in Q_T}$~${(c_2\ge 0,c_3\in L^1(Q_T,\setR_{\ge 0}))}$.
	\end{description}
	A system like \eqref{eq:model}  has e.g. applications in nonlinear elastic mechanics \cite{Z86} or in 
	 image restoration
	\cite{LLP10,HHLT13,CLR06}. Similar equations also appear in Zhikov's model for the termistor problem~in~\cite{Z08b}, or in the investigation of variational integrals with non-standard growth by Acerbi and Mingione in  \cite{AM00,AM01,AM02}~or~by~Marcellini~in~\cite{Mar91}.
	However, the author's underlying incentive for the consideration of the system \eqref{eq:model} is that it represents
	 a simplification
	of the not yet completely solved unsteady $p(\cdot,\cdot)$--Stokes and
	 $p(\cdot,\cdot)$--Navier--Stokes equations, as it does not
	contain an incompressibility constraint, i.e., the condition $\textup{div}(\bsu)=0$ in $Q_T$ is missing in \eqref{eq:model}, which would further complicate the hereinafter analysis, but the nonlinear
	elliptic operator, i.e., the mapping  ${\bfS:Q_T\times \mathbb{M}_{\textup{sym}}^{d\times d}\to  \mathbb{M}_{\textup{sym}}^{d\times d}}$,  is governed
	by the symmetric gradient only, rather than the full gradient. In fact, 
	a long term goal of this research is to
	prove the existence of weak solutions of the unsteady $p(\cdot,\cdot)$--Stokes and 
	$p(\cdot,\cdot)$--Navier--Stokes equations. 	
	
	The system \eqref{eq:model}  in its present form, up to more restrictive assumptions on the lower order nonlinear operator, i.e., on the mapping  $\bfb:Q_T\times \setR^d\to \setR^d$, has already~been~considered~in~\cite{KR20}. There it has been elaborated that the appropriate functional framework for the treatment of the problem \eqref{eq:model} is given through so-called variable exponent Bochner--Lebesgue spaces and variable exponent Bochner--Sobolev spaces with a~symmetric~gradient~structure,~i.e., the spaces\footnote{All spaces are thoroughly defined in Section \ref{sec:3} and Section \ref{sec:4}.}
	\begin{align}
	\bscal{X}^{q,p}_{\bfvarepsilon}(Q_T)
	&:=\Big\{\bsu\in L^{q(\cdot,\cdot)}(Q_T)^d\;\big|\;
	\bfvarepsilon(\bsu)\in L^{p(\cdot,\cdot)}(Q_T,\mathbb{M}^{d\times d}_{\sym}),
	\;\bsu(t)\in X^{q,p}_{\bfvarepsilon}(t)\text{ for a.e. }t\in I\Big\},\notag\\
	\bscal{W}^{q,p}_{\bfvarepsilon}(Q_T)
	&:=\bigg\{\bsu\in \bscal{X}^{q,p}_{\bfvarepsilon}(Q_T)\;\Big|\;\frac{\textbf{d}\bsu}{\textbf{dt}} \in \bscal{X}^{q,p}_{\bfvarepsilon}(Q_T)^*\bigg\},\label{eq:ff}
	\end{align}
	where $X^{q,p}_{\bfvarepsilon}(t)$ is for every $t\in I$ defined as the closure of $C_0^\infty(\Omega)^d$ with respect to the norm $\|\cdot\|_{L^{q(t,\cdot)}(\Omega)^d}+\|\bfvarepsilon(\cdot)\|_{L^{p(t,\cdot)}(\Omega)^{d\times d}}$ (cf.~Definition~\ref{3.1}) and
	$\frac{\bfd}{\bfd\bft}$ denotes the \textit{generalized time derivative} (cf.~Definition~\ref{4.1}), which is highly related to the distributional time derivative, living~in~$\mathcal{D}'(Q_T)^d$, i.e., the topological dual space of the locally convex Hausdorff vector space $C_0^\infty(Q_T)^d$ (cf.~\cite[Kap.~2,~4.]{GGZ74}). Using these spaces, the maximal monotonicity of the generalized time derivative $\frac{\bfd}{\bfd\bft}$ on the  space $\{\bsu\in\bscal{W}^{q,p}_{\bfvarepsilon}(Q_T)\mid \bsu_c(0)=\mathbf{0}\textup{ in }Y\}$~(cf.~\cite[Proposition 7.3]{KR20}),~where~${Y:=L^2(\Omega)^d}$, and a version of the main theorem on pseudo-monotone perturbations of
	maximal monotone mappings (cf.~\cite[Cor. 22]{Bre72}), \cite{KR20} proved the following~abstract~existence~result~(cf.~\cite[Thm.~7.11]{KR20}).\newpage
	\begin{thm}\label{1.2}
		Let $\Omega\subseteq \setR^d$, $d\ge 2$, be a bounded Lipschitz domain, $I:=\left(0,T\right)$,  $T<\infty$,  $Q_T:=I\times \Omega$ and $q,p\in \mathcal{P}^{\log}(Q_T)$  with $q^->1$ and $p^-\ge \frac{2d}{d+2}$. Moreover, let the operator $\bscal{A}:\bscal{W}^{q,p}_{\bfvarepsilon}(Q_T)\subseteq \bscal{X}^{q,p}_{\bfvarepsilon}(Q_T)\to \bscal{X}^{q,p}_{\bfvarepsilon}(Q_T)^*$ be coercive, $\frac{\textbf{d}}{\textbf{dt}}$-pseudo-monotone, i.e., from
		\begin{align}
			\begin{split}
			\bsu_n\overset{n\to \infty}{\weakto }\bsu\quad\text{ in }\bscal{W}^{q,p}_{\bfvarepsilon}(Q_T),\\
			\limsup_{n\to \infty}{\langle \bscal{A}\bsu_n,\bsu_n-\bsu\rangle_{\bscal{X}^{q,p}_{\bfvarepsilon}(Q_T)}}\leq 0,
		\end{split}\label{Lpm}
		\end{align}
		it follows that $\langle \bscal{A}\bsu,\bsu-\bsv\rangle_{\bscal{X}^{q,p}_{\bfvarepsilon}(Q_T)}\leq \liminf_{n\to \infty}{\langle \bscal{A}\bsu_n,\bsu_n-\bsv\rangle_{\bscal{X}^{q,p}_{\bfvarepsilon}(Q_T)}}$ for every $\bsv\in \bscal{X}^{q,p}_{\bfvarepsilon}(Q_T)$, and assume that there exists a bounded function $\psi:\setR_{\ge 0}\times \setR_{\ge 0}\to \setR_{\ge 0}$ and a constant $\theta\in \left[0,1\right)$, such that for every $\bsu\in \bscal{W}^{q,p}_{\bfvarepsilon}(Q_T)$ there holds
		\begin{align}
			\|\bscal{A}\bsu\|_{\bscal{X}^{q,p}_{\bfvarepsilon}(Q_T)^*}\leq \psi(\|\bsu\|_{\bscal{X}^{q,p}_{\bfvarepsilon}(Q_T)},\|\bsu_c(0)\|_Y)+\theta\bigg\|\frac{\textbf{d}\bsu}{\textbf{dt}}\bigg\|_{\bscal{X}^{q,p}_{\bfvarepsilon}(Q_T)}.\label{eq:bddcon}
		\end{align}
		 Then, for arbitrary $\bfu_0\in Y$ and $\bsu^*\in \bscal{X}^{q,p}_{\bfvarepsilon}(Q_T)^*$ there exists a solution $\bsu\in \bscal{W}^{q,p}_{\bfvarepsilon}(Q_T)$ of 
		\begin{alignat*}{2}
		\frac{\bfd\bsu}{\bfd\bft}+\bscal{A}\bsu&=\bsu^*&&\quad\text{ in }\bscal{X}^{q,p}_{\bfvarepsilon}(Q_T)^*,\\
		\bsu_c(0)&=\bfu_0&&\quad\text{ in }Y.
		\end{alignat*}
	\end{thm}

	In fact, in \cite[Proposition 8.1 \& Proposition 8.9]{KR20} it has been proved for $q:=p^-\ge \frac{2d}{d+2}$~that ${\bscal{A}\!:\! \bscal{W}^{q,p}_{\bfvarepsilon}(Q_T)\!\subseteq\! \bscal{X}^{q,p}_{\bfvarepsilon}(Q_T)\!\to\! \bscal{X}^{q,p}_{\bfvarepsilon}(Q_T)^*}$, defined via $\langle \bscal{A}\bsu,\bsv\rangle_{\bscal{X}^{q,p}_{\bfvarepsilon}(Q_T)}\!:=\!\int_{Q_T}{\!\bfS(\cdot,\cdot,\bfvarepsilon(\bsu)):\bfvarepsilon(\bsv)\,dtdx}\linebreak+\int_{Q_T}{\bfb(\cdot,\cdot,\bsu)\cdot\bsv\,dtdx}$ for all $\bsu\in \bscal{W}^{q,p}_{\bfvarepsilon}(Q_T)$ and $\bsv\in  \bscal{X}^{q,p}_{\bfvarepsilon}(Q_T)$, is  coercive, $\frac{\bfd}{\bfd\bft}$-pseudo-monotone and satisfies the boundedness condition \eqref{eq:bddcon}, provided that the mapping $\bfb:Q_T\times \setR^d\to \setR^d$ does not satisfy (\hyperlink{B.2}{B.2}) and (\hyperlink{B.3}{B.3}), but has the following properties instead:
	\begin{description}[{(B.3*) }]
		\item[\textbf{(B.2*)}]\hypertarget{B.2*}{} $\vert\bfb(t,x,\bfa)\vert \leq 
		\gamma(1+\vert \bfa\vert)^s+\eta(t,x)$ for every
		$\bfa\in \setR^d$, a.e. $(t,x)^\top\in Q_T$ \newline$(\gamma\ge 0,s\in \left(0,(p^-)_*/(p^-)'\right),\eta\in L^{(p^-)'}(Q_T,\setR_{\ge 0}))$.
		\item[\textbf{(B.3*)}]\hypertarget{B.3*}{} $\bfb(t,x,\bfa)\cdot\bfa\ge -c_3(t,x)$ 
		for every $\bfa\in \setR^d$, a.e. $(t,x)^\top\in Q_T$ $(c_3\in L^1(Q_T,\setR_{\ge 0}))$.
	\end{description}
	 Therefore, Theorem~\ref{1.2} is applicable on \eqref{eq:model}, if we replace  (\hyperlink{B.2}{B.2}) and (\hyperlink{B.3}{B.3}) by (\hyperlink{B.2*}{B.2*}) and (\hyperlink{B.3*}{B.3*}). 
	
	Apparently, (\hyperlink{B.3*}{B.3*}) is more restrictive than (\hyperlink{B.3}{B.3}). Also note that (\hyperlink{B.2*}{B.2*}) implies (\hyperlink{B.2}{B.2}).~In~fact, suppose that $\bfb:Q_T\times \setR^d\to \setR^d$  satisfies (\hyperlink{B.2*}{B.2*}). 
	 Then, there holds
	${s\in \left(0,(p^-)_*/(p^-)'\right)}$, i.e., $s\leq (p^-)_*/(p^-)'-\varepsilon$ for some $\varepsilon\in (0,(p^-)_*-1]$. Thus, if we define $r:=\max\{2,p_*\}-\varepsilon$, then \linebreak
	$s\leq p_*/(p_*)'-\varepsilon=p_*-1-\vep\leq r-1$, where we used that $(p_*)'\leq (p^-)'$, due to $p_*\ge p^-$, i.e., ${\bfb:Q_T\times \setR^d\to \setR^d}$  satisfies (\hyperlink{B.2}{B.2}).
	However, if we have (\hyperlink{B.2}{B.2}) and (\hyperlink{B.3}{B.3}) in \eqref{eq:model},~or~even~${p^-<\frac{2d}{d+2}}$, then we can neither say that the operator ${\bscal{A}: \bscal{W}^{q,p}_{\bfvarepsilon}(Q_T)\subseteq \bscal{X}^{q,p}_{\bfvarepsilon}(Q_T)\to \bscal{X}^{q,p}_{\bfvarepsilon}(Q_T)^*}$
	is coercive, nor that it  satisfies the boundedness condition \eqref{eq:bddcon}, i.e., Theorem~\ref{1.2} is reaching its limits. This issue indicates an imbalance between the demanded continuity and coercivity conditions in Theorem~\ref{1.2}.
	 More precisely, while the required $\frac{\textbf{d}}{\textbf{dt}}$-pseudo-monotonicity is quite general, coercivity is a restrictive assumption, which e.g. for ${\bfb:Q_T\times \setR^d\to \setR^d}$ with (\hyperlink{B.3*}{B.3*})   is not fulfilled.
	
	In \cite{KR19,K19} the same issue has already been considered in the context of Bochner--Lebesgue and Bochner--Sobolev spaces. The approach of these articles consists in the introduction of alternative notions of pseudo-monotonicity and coercivity, which, in contrast to $\frac{\mathbf{d}}{\mathbf{dt}}$-pseudo-monotonicity and coercivity, both incorporate information from the time derivative, and therefore are more in balance. The idea is to weaken the pseudo-monotonicity assumption to a bearable extend, in order to make a coercivity condition accessible, which takes the additional information from the time derivative into account. This approach led to the notions Bochner pseudo-monotonicity, Bochner condition (M) and Bochner coercivity. We will follow the approaches of  \cite{KR19,K19}. 
	To be more precise, the main purpose of this paper is to show that the ideas from \cite{KR19,K19} can be generalized to the functional framework of  \eqref{eq:model}, i.e., to \eqref{eq:ff}. In this context, we prove fundamental parabolic embedding and compactness results involving  the spaces $\bscal{X}^{2,p}_{\bfvarepsilon}(Q_T)$ and $\bscal{Y}^{\infty}(Q_T):=L^\infty(I,Y)$, which are 
	indispensable, in order to guarantee that ${\bscal{A}: \bscal{W}^{q,p}_{\bfvarepsilon}(Q_T)\subseteq \bscal{X}^{q,p}_{\bfvarepsilon}(Q_T)\to \bscal{X}^{q,p}_{\bfvarepsilon}(Q_T)^*}$ satisfies these new concepts, if only (\hyperlink{B.2}{B.2}) and (\hyperlink{B.3}{B.3}) are satisfied in \eqref{eq:model}.\vspace*{-1cm}
	
	\newpage
	\textbf{Plan of the paper:} In Section~\ref{sec:2} we recall some basic definitions and results concerning evolution equations and variable exponent spaces. In Section~\ref{sec:3} we recall the definition of the space $\bscal{X}^{q,p}_{\bfvarepsilon}(Q_T)$ and some of its basic properties, such as completeness, reflexivity and a characterization of its dual space. In Section~\ref{sec:4} we recapitulate the definition of the generalized time derivative $\frac{\bfd}{\bfd\bft}$ and of the space $\bscal{W}^{q,p}_{\bfvarepsilon}(Q_T)$. In Section~\ref{sec:5} we show the invalidity of a Poincar\'e type inequality for $\bscal{X}^{q,p}_{\bfvarepsilon}(Q_T)$ and prove 
	in this context an appropriate analogue of  Poincar\'e's~inequality~for~$\bscal{X}^{q,p}_{\bfvarepsilon}(Q_T)$. Simultaneously, we derive from
	this inequality several embedding results for the space $\bscal{X}^{q,p}_{\bfvarepsilon}(Q_T)$.
	In Section~\ref{sec:6} we introduce extensions of the notions Bochner
	pseudo-monotonicity, Bochner condition (M) and Bochner coercivity to  the framework of $\bscal{X}^{q,p}_{\bfvarepsilon}(Q_T)$. In Section~\ref{sec:7} we generalize the so-called Hirano--Landes approach to the framework $\bscal{X}^{q,p}_{\bfvarepsilon}(Q_T)$, to derive 
	sufficient conditions on  families  of operators that imply these 
	concepts. In Section~\ref{sec:8} we prove an abstract existence result for bounded, Bochner coercive operators, which satisfy the Bochner condition (M). In Section~\ref{sec:9} we apply this abstract existence result on~the~problem~\eqref{eq:model}.
	
	\section{Preliminaries}
	\label{sec:2}
	
	\subsection{Basic Notation}
	For a Banach space $X$ with norm $\|\cdot\|_X$ we denote by $X^*$ its dual space equipped with the norm $\|\cdot\|_{X^*}$. The duality pairing is denoted by
	$\langle \cdot,\cdot\rangle_X$. All occuring Banach spaces are assumed to be real. By $D(A)$ we denote the domain of definition of an operator $A:D(A)\subseteq X\to Y$, and by $R(A):=\{Ax\mid x\in D(A)\}$ its range.
	
	\subsection{Evolution equations}
	Let $X$ be a reflexive Banach space, 
	$Y$ a Hilbert space
	and $j:X\to Y$ an embedding, i.e., $j $ is linear, bounded and injective, such that $R(j)$ is dense in $Y$.
	Then, the triple $(X,Y,j)$ is said to be an \textbf{evolution triple}.  Denote by $R_Y:Y\to Y^*$ the Riesz isomorphism with 
	respect to $(\cdot,\cdot)_Y$. As $j$ is a dense embedding 
	the adjoint operator $j^*:Y^*\to X^*$ and therefore~${e:=j^*R_Yj:X\to X^*}$
	are embeddings~as~well. We call $e$ the 
	\textbf{canonical embedding} of $(X,Y,j)$.
	
	\begin{defn}\label{2.1}
	Let $(X,Y,j)$ be an evolution triple, $I:=\left(0,T\right)$, $T<\infty$, and~${p\in \left(1,\infty\right)}$.  A function $\bsu\in L^p(I,X)$ possesses a \textbf{generalized derivative with respect to $e$}, if there exists ${\bsu^*\in (L^p(I,X))^*}$, such that for every $x\in X$ and $\varphi\in C^\infty_0(I)$ it holds
	\begin{align*}
		-\int_I{(\bsu(t),jx)_Y\varphi^\prime(t)\,dt}=\langle \bsu^*,x\varphi\rangle_{L^p(I,X)}.
	\end{align*} 
	In this case, we define $\frac{d_e\bsu}{dt}:=\bsu^*$.
	By $W_e^{1,p,p'}(I,X,X^*):=\big\{\bsu\in L^p(I,X)\mid\exists \frac{d_e\bsu}{dt}\in (L^p(I,X))^*\big\}$
	we denote the  \textbf{Bochner--Sobolev space with respect to $e$}.
	\end{defn}

	\begin{prop}\label{2.2}
		Let $(X,Y,j)$ be an evolution triple, $I:=\left(0,T\right)$, $T<\infty$, and~${p\in \left(1,\infty\right)}$. Then, it holds:
		\begin{description}[{(ii)}]
			\item[{(i)}] Given $\bsu\in W_e^{1,p,p'}(I,X,X^*)$ the function 
			$\bsj\bsu\in L^p(I,Y)$, given via 
			$(\bsj\bsu)(t):=j(\bsu(t))$ in $Y$ for 
			almost every $t\in I$, possesses a unique representation 
			$\bsj_c\bsu\in C^0(\overline{I},Y)$ and the resulting mapping 
			$\bsj_c:W_e^{1,p,p'}(I,X,X^*)\to C^0(\overline{I},Y)$ is an embedding.
			\item[{(ii)}] For every  $\bsu,\bsv\in W_e^{1,p,p'}(I,X,X^*)$ 
			and $t,t'\in \overline{I}$ with $t'\leq t$ there holds
			\begin{align*}
			\int_{t'}^{t}{\left\langle
				\frac{d_e\bsu}{dt}(s),\bsv(s)\right\rangle_Xds}
			=\left[((\bsj_c\bsu)(s), (\bsj_c
			\bsv)(s))_Y\right]^{s=t}_{s=t'}-\int_{t'}^{t}{\left\langle
				\frac{d_e\bsv}{dt}(s),\bsu(s)\right\rangle_Xds}.
			\end{align*}
		\end{description}  
	\end{prop}
	
	\begin{proof}
		See \cite[Chapter III.1, Proposition 1.2]{Sho97}.\hfill$\qed$\vspace*{-0.5cm}
	\end{proof}
	\newpage

	For an evolution triple $(X,Y,j)$, $I:=\left(0,T\right)$, $T<\infty$, and $p\in \left(1,\infty\right)$, we define the induced embedding $\boldsymbol{j}:L^p(I,X)\to L^p(I,Y)$ for every $\boldsymbol{u}\in L^p(I,X)$ via $(\boldsymbol{j}\boldsymbol{u})(t):=j(\boldsymbol{u}(t))$ in $Y$ for almost every $t\in I$ (cf.~\cite[Bem.~2.40~(ii)]{Ru04}). Moreover, we define
	the intersection space
	\begin{align*}
		L^p(I,X)\cap_{\bsj}L^\infty(I,Y):=\big\{\bsu\in L^p(I,X)\mid\bsj\bsu\in L^\infty(I,Y)\big\},
	\end{align*}
	which forms a Banach space, if equipped with the canonical norm (cf.~\cite{KR19,K19})
	\begin{align*}
	\|\cdot\|_{L^p(I,X)\cap_{\bsj}L^\infty(I,Y)}:=\|\cdot\|_{L^p(I,X)}+\|\bsj(\cdot)\|_{L^\infty(I,Y)}.
	\end{align*}
	In the context of evolutionary problems the following notions of continuity and coercivity proved themselves the to be fruitful.

	\begin{defn}[Bochner pseudo-monotonicity and Bochner condition (M)]\label{2.3}\newline
		Let $(X,Y,j)$ be an evolution triple, $I:=\left(0,T\right)$, $T<\infty$, and $p\in \left(1,\infty\right)$.  Then, an
		 operator ${\bscal{A}:L^p(I,X)\cap_{\bsj}L^\infty(I,Y)
		\to(L^p(I,X))^*}$ is said to
		\begin{description}[{(ii)}]
			\item[{(i)}] be \textbf{Bochner pseudo-monotone}, if 
			for a sequence $(\bsu_n)_{n\in\setN}\subseteq 
			L^p(I,X)\cap_{\bsj} L^\infty(I,Y)$ from
			\begin{alignat}{2}
			\bsu_n&\overset{n\to\infty}{\weakto}\bsu
			&&\quad\text{ in }L^p(I,X),\label{eq:2.4}\\
			\bsj\bsu_n&\;\;\overset{\ast}{\rightharpoondown}\;\;\bsj\bsu
			&&\quad\text{ in }L^\infty(I,Y),\label{eq:2.5}\\
			(\bsj\bsu_n)(t)&\overset{n\to\infty}{\weakto}(\bsj\bsu)(t)&&\quad\text{ in }Y\quad\text{ for a.e. }t\in I,\label{eq:2.6}
			\end{alignat}
			and 
			\begin{align}
				\limsup_{n\to\infty}{\langle \bscal{A}\bsu_n,
					\bsu_n-\bsu\rangle_{L^p(I,X)}}\leq 0,\label{eq:2.7}
			\end{align}
			it follows that $\langle\bscal{A}\bsu,
			\bsu-\bsv\rangle_{L^p(I,X)}
			\leq	\liminf_{n\to\infty}{\langle \bscal{A}\bsu_n,
				\bsu_n-\bsv\rangle_{L^p(I,X)}}$ for every
			$\bsv\in L^p(I,X)$.
			\item[{(ii)}] satisfy the \textbf{Bochner condition (M)}, if 
			for a sequence $(\bsu_n)_{n\in\setN}\subseteq 
		L^p(I,X)\cap_{\bsj}L^\infty(I,Y)$ from \eqref{eq:2.4}--\eqref{eq:2.6} and
			\begin{align}
			\bscal{A}\bsu_n\overset{n\to\infty}{\weakto}
			\bsu^*\quad\text{ in }(L^p(I,X))^*,\label{eq:2.8}\\
			\limsup_{n\to\infty}{\langle \bscal{A}\bsu_n,
				\bsu_n\rangle_{L^p(I,X)}}\leq \langle \bsu^*,
			\bsu\rangle_{L^p(I,X)},\label{eq:2.9}
			\end{align}
			it follows that $\bscal{A}\bsu=\bsu^*$ in $(L^p(I,X))^*$.
		\end{description}
	\end{defn}

	\begin{defn}[Bochner coercivity]\label{2.4}
			Let $(X,Y,j)$ be an evolution triple, $I:=\left(0,T\right)$, ${T<\infty}$, and $p\in \left(1,\infty\right)$. 
		Then, an operator $\bscal{A}:L^p(I,X)\cap_{\bsj}L^\infty(I,Y)
		\to(L^p(I,X))^*$ is said to be
		\begin{description}[{(ii)}]
			\item[(i)] \textbf{Bochner coercive with respect to $\bsu^*\in (L^p(I,X))^*$ and $u_0\in Y$}, 
			if~there~exists~a~constant $M:=M(\bsu^*,u_0,\bscal{A})>0$, 
			such that for every $\bsu\in L^p(I,X)\cap_{\bsj}L^\infty(I,Y)$ 
			from 
			\begin{align*}
			\frac{1}{2}\|(\bsj\bsu)(t)\|_Y^2+\langle\bscal{A}\bsu-
			\bsu^*,\bsu\chi_{\left[0,t\right]}\rangle_{L^p(I,X)}\leq \frac{1}{2}\|u_0\|_Y^2\quad
			\text{ for a.e. }t\in I,
			\end{align*}
			it follows that $\|\bsu\|_{L^p(I,X)\cap_{\bsj}L^\infty(I,Y)}\leq M$. 
			\item[(ii)] \textbf{Bochner coercive}, if it is Bochner coercive with 
			respect to  all $\bsu^*\in(L^p(I,X))^*$ and $u_0\in Y$.
		\end{description}
	\end{defn}

	\begin{thm}\label{2.9}
		Let $(X,Y,j)$ be an evolution triple, $I:=\left(0,T\right)$, ${T<\infty}$, and ${p\in \left(1,\infty\right)}$.  Moreover, let  $\bscal{A}:L^p(I,X)\cap L^\infty(I,Y)\to (L^p(I,X))^*$ be bounded, Bochner coercive with respect to $\bsu^*\in (L^p(I,X))^*$ and $u_0\in Y$, and satisfying the Bochner condition (M). Then, there exists a solution $\bsu\in W_e^{1,p,p'}(I,X,X^*)$ of the evolution equation
		\begin{align*}
		\begin{alignedat}{2}
		\frac{d_e\bsu}{dt}+\bscal{A}\bsu&=\bsu^*&&\quad\text{ in }(L^p(I,X))^*,\\
		(\bsj_c\bsu)(0)&=u_0&&\quad\text{ in }Y.
		\end{alignedat}
		\end{align*}
		Here, the initial condition has to be understood in the sense of the unique continuous representation $\bsj_c\bsu\in C^0(\overline{I},Y)$ (cf.~Proposition~\ref{2.2}~(i)).
	\end{thm}

	\begin{proof}
		Follows from \cite[Theorem 6.1]{K19}, where the assertion is proved, using the even more general notions $C^0$--Bochner coercivity and $C^0$--Bochner condition (M).\hfill$\qed$\vspace*{-1cm}
	\end{proof}
	
	\subsection{Variable exponent spaces an their basic properties}

In this section we summarize all essential information
concerning variable exponent Lebesgue and Sobolev spaces, which
will find use in the hereinafter analysis. These spaces has
already been studied by Hudzik \cite{Hud80}, Musielak
\cite{Mus83}, Kov\'{a}\v{c}ik, R\'{a}kosn\'{\i}k \cite{KR91},
R\r{u}\v{z}i\v{c}ka \cite{Ru00} and many others. For an extensive presentation, including all results of this section, we want to refer to \cite{KR91,DHHR11,CUF13}.

Let $G\subseteq \setR^n$, $n\in \setN$, be a measurable set and $p:G\to\left[1,+\infty\right)$ a measurable function.
Then, $p$ will be called a \textbf{variable exponent} and we define  $\mathcal{P}(G)$ to be the set of all variable exponents.
Furthermore, we define the limit exponents
$p^-:=\essinf_{x\in G}{p(x)}$ and $p^+:=\esssup_{x\in G}{p(x)}$ as well as the set of \textbf{bounded variable exponents} $\mathcal{P}^\infty(G):=\{p\in \mathcal{P}(G)\fdg p^+<\infty\}$.

For $p\in \mathcal{P}^\infty(G)$, 
the \textbf{variable exponent Lebesgue space} $L^{p(\cdot)}(G)$ consists of all 
measurable functions $f:G\to \setR$ for which the \textbf{modular} $\rho_{p(\cdot)}(f):=\int_{G}{\vert f(x)\vert^{p(x)}\,dx}$ is finite. The so-called Luxemburg norm $\|f\|_{L^{p(\cdot)}(G)}:=
\inf\{\lambda>0\fdg \rho_{p(\cdot)}(\lambda^{-1}f)\leq 1\}$ 
turns $L^{p(\cdot)}(G)$~into~a~Banach~space. 

For $p\in \mathcal{P}^\infty(G)$ with $p^->1$, $p'\in
\mathcal{P}^\infty(G)$ is defined via
$p'(x):=\frac{p(x)}{p(x)-1}$ for almost~every~${x\in G}$. We
make frequent use of H\"older's inequality
\begin{align}
\|gf\|_{L^1(G)} \leq 2\, \|g\|_{p'(\cdot)}\|f\|_{p(\cdot)}, \label{hoelder}
\end{align}
valid for every $g\in
L^{p'(\cdot)}(G)$ and $f \in L^{p(\cdot)}(G)$. In particular, \eqref{hoelder} ensures the well-definedness of the product $(\cdot,\cdot)_{L^{p(\cdot)}(G)}:L^{p'(\cdot)}(G)\times L^{p(\cdot)}(G)\to \setR$, for every  $g\in L^{p'(\cdot)}(G)$ and $f\in L^{p(\cdot)}(G)$ given via\vspace*{-2mm} 
\begin{align*}
(g,f)_{L^{p(\cdot)}(G)}\!:=\!\int_{G}{g(x)f(x)\,dx}.
\end{align*}\\[-10pt]
The products $(\cdot,\cdot)_{L^{p(\cdot)}(G)^n}$ and $(\cdot,\cdot)_{L^{p(\cdot)}(G)^{n\times n}}$ are defined accordingly. 

For a bounded set $G\subseteq \setR^n$, $n\in \setN$, and  $q,p\in \mathcal{P}^\infty(G)$, satisfying $q\leq p$~almost~everywhere~in~$G$, every function $f\in L^{p(\cdot)}(G)$ also satisfies $f\in L^{q(\cdot)}(G)$ with
\begin{align}
	\|f\|_{L^{q(\cdot)}(G)}\leq 2(1+\vert G\vert)\|f\|_{L^{p(\cdot)}(G)}.\label{embed}
\end{align}

	For an open set $G\subseteq \setR^n$, $n\in \setN$, 
	and $q,p\in \mathcal{P}^\infty(G)$, we define the \textbf{variable exponent Sobolev spaces}
	\begin{align*}
	\widehat{X}^{q(\cdot),p(\cdot)}_{\nabla}(G)&:=\{f\in L^{q(\cdot)}(G)\fdg \nabla f\in L^{p(\cdot)}(G)^n\},\\
	\widehat{X}^{q(\cdot),p(\cdot)}_{\bfvarepsilon}(G)&:=\{\bff\in L^{q(\cdot)}(G)^n\fdg \bfvarepsilon(\bff)\in L^{p(\cdot)}(G,\mathbb{M}^{n\times n}_{\sym})\},
	\end{align*} 
	where $\nabla f$ and $\bfvarepsilon(\bff)$ have to be understood in a distributional sense. $\widehat{X}^{q(\cdot),p(\cdot)}_{\nabla}(G)$ and $\widehat{X}^{q(\cdot),p(\cdot)}_{\bfvarepsilon}(G)$  form Banach spaces (cf.~\cite[Proposition~2.13]{KR20}), if equipped with the  norms\vspace*{-1mm}
	\begin{align*}
	\|\cdot\|_{\widehat{X}^{q(\cdot),p(\cdot)}_{\nabla}(G)}&:=\|\cdot\|_{L^{q(\cdot)}(G)}
	+\|\nabla\cdot\|_{L^{p(\cdot)}(G)^n},\\
	\|\cdot\|_{\widehat{X}^{q(\cdot),p(\cdot)}_{\bfvarepsilon}(G)}&:=\|\cdot\|_{L^{q(\cdot)}(G)^n}
	+\|\bfvarepsilon(\cdot)\|_{L^{p(\cdot)}(G)^{n\times n}},
    \end{align*}
	respectively.
	We define $X^{q(\cdot),p(\cdot)}_{\nabla}(G)$ and $X^{q(\cdot),p(\cdot)}_{\bfvarepsilon}(G)$  to be  the closures of $C_0^\infty(G)$~and~$C_0^\infty(G)^n$  with respect to
	 the norms $\|\cdot\|_{\widehat{X}^{q(\cdot),p(\cdot)}_{\nabla}(G)}$ and $\|\cdot\|_{\widehat{X}^{q(\cdot),p(\cdot)}_{\bfvarepsilon}(G)}$, respectively.

We say that a bounded exponent $p\in \mathcal P^\infty (G)$ is locally
\textbf{$\log$--Hölder continuous}, if there is a constant $c_1>0$, such that
for all $x,y\in G$
\begin{align*}
\vert p(x)-p(y)\vert \leq \frac{c_1}{\log(e+1/\vert x-y\vert)}.
\end{align*}
We say that $p \in \mathcal P^\infty (G)$ satisfies the \textbf{$\log$--Hölder decay condition}, if there exist 
constant $c_2>0$ and $p_\infty\in \setR$, such that for all $x\in G$
\begin{align*}
\vert p(x)-p_\infty\vert \leq\frac{c_2}{\log(e+1/\vert x\vert)}.
\end{align*}\vspace*{-0.5cm}\newpage
We say that $p$ is \textbf{globally $\log$--Hölder continuous} on $G$, if it is locally 
$\log$--Hölder continuous and satisfies the $\log$--Hölder decay condition. Moreover,
we denote by $\mathcal{P}^{\log}(G)$ the set of  all
globally $\log$--Hölder continuous
variable exponents on $G$.

One pleasant property of globally $\log$--Hölder continuous exponents is their 
ability to admit extensions to the whole space $\setR^n$, $n\in\setN$, having similar characteristics.

\begin{prop}\label{2.13}
	Let $G\subsetneq \setR^n$, $n\in\setN$, be a domain. If $p\in \mathcal{P}^{\log}(G)$,  then it admits an extension
	 $\overline{p}\in \mathcal{P}^{\log}(\setR^n)$, i.e., 
	$\overline{p}=p$ in $G$, satisfying $\overline{p}^-=p^-$ and $\overline{p}^+=p^+$.
\end{prop}

A similar extension result holds for continuous exponents on closed sets, see e.g. \cite[Prop.~3.1]{Erc97}.

\begin{prop}\label{2.13.1}
	Let $G\subsetneq \setR^n$, $n\in\setN$, be closed. If  $p\in \mathcal{P}^{\infty}(G)\cap C^0(G)$, then it admits an extension $\overline{p}\in \mathcal{P}^{\infty}(\setR^n)\cap C^0(\setR^n)$, i.e., 
	$\overline{p}=p$ in $G$, satisfying $\overline{p}^-=p^-$ and $\overline{p}^+=p^+$.
\end{prop}

	\section{Variable exponent Bochner--Lebesgue spaces}
	\label{sec:3}
	
	In this section we recall and discuss properties of variable exponent Bochner--Lebesgue spaces with a symmetric gradient structure, the appropriate substitute of usual Bochner--Lebesgue spaces for the treatment of unsteady problems in variable exponent spaces governed by the symmetric part~of~the~gradient, such as the model problem \eqref{eq:model}. Variable exponent Bochner--Lebesgue spaces incorporating only the symmetric part of the gradient have already been introduced and thoroughly~examined~in~\cite{KR20}. 
	
	Throughout the entire section, if nothing else is stated, let $\Omega\subseteq\setR^d$, $d\ge 2$, 
	be a bounded domain, $I:=\left(0,T\right)$, $T<\infty$,  $Q_T:=I\times \Omega$,
	and $q,p\in \mathcal{P}^\infty(Q_T)$.
	
	\begin{defn}[Time slice spaces]\label{3.1} We define for almost every $t\in I$, the 
		\textbf{time slice spaces}
		\begin{align*}
		X^{q,p}_{\bfvarepsilon}(t):= X^{q(t,\cdot),p(t,\cdot)}_{\bfvarepsilon}(\Omega).
		\end{align*}
		Furthermore, we define the \textbf{limiting time slice spaces} 
		\begin{align*}
			X^{q,p}_+:=X^{q^+,p^+}_{\bfvarepsilon}(\Omega),\qquad X^{q,p}_-:= X^{q^-,p^-}_{\bfvarepsilon}(\Omega).
		\end{align*}
	\end{defn}	
	
	\begin{defn}\label{3.3} We define the \textbf{variable exponent Bochner--Lebesgue space}
		\begin{align*}
		\bscal{X}_{\bfvarepsilon}^{q,p}(Q_T)
		:=\Big\{\bsu\in L^{q(\cdot,\cdot)}(Q_T)^d
		\;\big|\; \bfvarepsilon(\bsu)\in L^{p(\cdot,\cdot)}(Q_T,\mathbb{M}_{\sym}^{d\times d}),\;\bsu(t)\in X^{q,p}_{\bfvarepsilon}(t)\text{ for a.e. }t\in I\Big\}.
		\end{align*}
		Furthermore, we define the \textbf{limiting Bochner--Lebesgue spaces} 
		\begin{align*}
				\bscal{X}_+^{q,p}(Q_T):=L^{\max\{q^+,p^+\}}(I,X^{q,p}_+),\qquad \bscal{X}_-^{q,p}(Q_T):=L^{\max\{q^-,p^-\}}(I,X^{q,p}_-).
		\end{align*}
	\end{defn}
	
	\begin{prop}\label{3.4}
		The space  $\bscal{X}_{\bfvarepsilon}^{q,p}(Q_T)$ forms a separable Banach space,~if~equipped~with~the~norm
		\begin{align*}
		\|\cdot\|_{\bscal{X}_{\bfvarepsilon}^{q,p}(Q_T)}
		:=\|\cdot\|_{L^{q(\cdot,\cdot)}(Q_T)^d}
		+\|\bfvarepsilon(\cdot)\|_{L^{p(\cdot,\cdot)}(Q_T)^{d\times d}}.
		\end{align*}
		If $q^-,p^->1$, then $\bscal{X}_{\bfvarepsilon}^{q,p}(Q_T)$  is additionally reflexive.
	\end{prop}
	
	\begin{proof}
		See~\cite[Proposition 4.6 \& Proposition 4.12]{KR20}.\hfill$\qed$
	\end{proof}
	
	\begin{cor}\label{3.5}
		Let
		$(\bsu_n)_{n\in\setN}
		\subseteq \bscal{X}^{q,p}_{\bfvarepsilon}(Q_T)$ 
		be a sequence,  such that $\bsu_n\to\bsu\text{ in }\bscal{X}^{q,p}_{\bfvarepsilon}(Q_T)$ ${(n\to\infty)}$.
		Then, there exists a subsequence $(\bsu_n)_{n\in\Lambda}\subseteq \bscal{X}^{q,p}_{\bfvarepsilon}(Q_T)$, 
		with $\Lambda\subseteq \setN$, such that 
		${\bsu_n(t)\to \bsu(t)}$ in $X^{q,p}_{\bfvarepsilon}(t)$ $(\Lambda\ni n\to\infty)$
		for almost every $t\in I$.
	\end{cor}
	
	\begin{proof}
		See~\cite[Corollary 4.11]{KR20}.\hfill$\qed$
	\end{proof}

	Under additional assumptions on the regularity of the domain and the continuity of the variable exponents, the space $C^\infty_0(Q_T)^d\!:=\!\{\boldsymbol{\phi}\in C^\infty(Q_T)^d\mid \textup{supp}(\boldsymbol{\phi})\!\subset\subset\! Q_T\}$~lies~densely~in~$\bscal{X}_{\bfvarepsilon}^{q,p}(Q_T)$.

		\begin{prop}\label{3.4.1}
		If $\Omega\subseteq \mathbb{R}^d$, $d\ge 2$, is a bounded Lipschitz domain and $q,p\in \mathcal{P}^{\log}(Q_T)$ with $q^-,p^->1$, then $C^\infty_0(Q_T)^d$ is dense in $\bscal{X}_{\bfvarepsilon}^{q,p}(Q_T)$.
	\end{prop}

	\begin{proof}
		See~\cite[Proposition 5.18]{KR20}.\hfill$\qed$\vspace*{-0.6cm}
	\end{proof}
	\newpage

	We have the following characterization of duality in $\bscal{X}^{q,p}_{\bfvarepsilon}(Q_T)$.
	
	\begin{prop}[Characterization of 
		$\bscal{X}^{q,p}_{\bfvarepsilon}(Q_T)^*$]\label{3.7} Let $q,p\in\mathcal{P}^\infty(Q_T)$ with $q^-,p^->1$. Then, the operator
		${\bscal{J}_{\bfvarepsilon}:
		L^{q'(\cdot,\cdot)}(Q_T)^d\times 
		L^{p'(\cdot,\cdot)}(Q_T,\mathbb{M}_{\sym}^{d\times d})\to \bscal{X}^{q,p}_{\bfvarepsilon}(Q_T)^*}$,
		given via
		\begin{align*}
		\langle	\bscal{J}_{\bfvarepsilon}(\bsf,
		\bsF),\bsu\rangle
		_{\bscal{X}^{q,p}_{\bfvarepsilon}(Q_T)}&:=( \bsf,\bsu)_{L^{q(\cdot,\cdot)}(Q_T)^d}
		+(\bsF,
		\bfvarepsilon(\bsu))_{L^{p(\cdot,\cdot)}(Q_T)^{d\times d}}
		\end{align*}
		for every $\bsf\in L^{q'(\cdot,\cdot)}(Q_T)^d$, 
		$\bsF\in L^{p'(\cdot,\cdot)}(Q_T,\mathbb{M}_{\sym}^{d\times d})$ 
		and $\bsu\in\bscal{X}^{q,p}_{\bfvarepsilon}(Q_T)$, 
		is well-defined, linear, Lipschitz continuous with constant $2$ and surjective.
	\end{prop}
	
	\begin{proof}
		See~\cite[Proposition 4.13]{KR20}.\hfill$\qed$
	\end{proof}
	
	In the same manner, we can characterize the dual space of $X^{q(\cdot),p(\cdot)}_{\bfvarepsilon}(\Omega)$.
	
	\begin{cor}[Characterization of 
		$X^{q(\cdot),p(\cdot)}_{\bfvarepsilon}(\Omega)^*$]\label{3.8}
		Let $q,p\in \mathcal{P}^\infty(\Omega)$ with $q^-,p^->1$ and  $X_{\bfvarepsilon}^{q,p}:=X^{q(\cdot),p(\cdot)}_{\bfvarepsilon}(\Omega)$. Then,~the~operator  $\mathcal{J}_{\bfvarepsilon}:L^{q'(\cdot)}(\Omega)^d\times L^{p'(\cdot)}(\Omega,\mathbb{M}^{d\times d}_{\sym})\to (X_{\bfvarepsilon}^{q,p})^*$,~given~via 
		\begin{align*}
		\langle\mathcal{J}_{\bfvarepsilon}(\bff,\bfF),\bfu\rangle_{X_{\bfvarepsilon}^{q,p}}:=(\bff,\bfu)_{L^{q(\cdot)}(\Omega)^d}+( \bfF,\bfvarepsilon(\bfu))_{L^{p(\cdot)}(\Omega)^{d\times d}},
		\end{align*}
		for every ${\bff\in L^{q'(\cdot)}(\Omega)^d}$, $\bfF\in L^{p'(\cdot)}(\Omega,\mathbb{M}^{d\times d}_{\sym})$ and $\bfu\in X_{\bfvarepsilon}^{q,p}$,
		is well-defined, linear, Lipschitz continuous with constant $2$ and  surjective. In addition, $X_{\bfvarepsilon}^{q,p}$ is reflexive.
	\end{cor}
	
	\begin{proof}
		See~\cite[Corollary 4.17]{KR20}.\hfill$\qed$
	\end{proof}
	
	Proposition~\ref{3.7} and Corollary~\ref{3.8} enable us to give a  well-posed definition
	of a time evaluation for functionals in $\bscal{X}^{q,p}_{\bfvarepsilon}(Q_T)^*$.
	\begin{rmk}
		[Time slices of functionals $\bsu^*\in \bscal{X}^{q,p}_{\bfvarepsilon}(Q_T)^*$]
		\label{3.9} Let $q,p\!\in\! \mathcal{P}^\infty(Q_T)$~with~${q^-,p^-\!>\!1}$.
		\begin{description}[{(ii)}]
			\item[(i)] On the basis of Corollary~\ref{3.8}, there holds for every $\bsf\in
			L^{q'(\cdot,\cdot)}(Q_T)^d$, $\bsF\in
			L^{p'(\cdot,\cdot)}(Q_T,\mathbb{M}_{\sym}^{d\times
				d})$ and
			$\bsu\in
			\bscal{X}^{q,p}_{\bfvarepsilon}(Q_T)$ that
			$(t\mapsto\langle
			\mathcal{J}_{\bfvarepsilon}(\bsf(t),\bsF(t)),\bsu(t)\rangle_{X^{q,p}_{\bfvarepsilon}(t)})\in L^1(I)$. More precisely, we have
			\begin{align}
			\begin{split}
			\langle \bscal{J}_{\bfvarepsilon}(\bsf,\bsF),\bsu\rangle_{\bscal{X}^{q,p}_{\bfvarepsilon}(Q_T)}
			=\int_I{\langle\mathcal{J}_{\bfvarepsilon}(\bsf(t),\bsF(t)),\bsu(t)\rangle_{X^{q,p}_{\bfvarepsilon}(t)}\,dt}.
			\end{split}\label{eq:3.9.1}
			\end{align}
			\item[(ii)]  According  to Proposition~\ref{3.7}, for every functional $\bsu^*\in\bscal{X}^{q,p}_{\bfvarepsilon}(Q_T)^*$,  there exist functions $\bsf\in L^{q'(\cdot,\cdot)}(Q_T)^d$ and $\bsF\in L^{p'(\cdot,\cdot)}(Q_T,\mathbb{M}_{\sym}^{d\times d})$, such that $\bsu^*=\bscal{J}_{\bfvarepsilon}(\bsf,\bsF)$ in $\bscal{X}^{q,p}_{\bfvarepsilon}(Q_T)^*$. Therefore, we define
			\begin{align}
			\bsu^*(t):=\mathcal{J}_{\bfvarepsilon}(\bsf(t),\bsF(t))\quad\textup{ in }X^{q,p}_{\bfvarepsilon}(t)^*\quad\text{ for a.e. }t\in I.\label{eq:2.22.1}
			\end{align}
			This definition is independent of the choice of $\bsf\in L^{q'(\cdot,\cdot)}(Q_T)^d$ and $\bsF\in L^{p'(\cdot,\cdot)}(Q_T,\mathbb{M}_{\sym}^{d\times d})$ with $\bsu^*=\bscal{J}_{\bfvarepsilon}(\bsf,\bsF)$ in $\bscal{X}^{q,p}_{\bfvarepsilon}(Q_T)^*$ (cf.~\cite[Remark 4.18]{KR20}).
		\end{description}
	\end{rmk}

	\section{Generalized time derivative and formula of integration by parts }
	\label{sec:4}	
	Since the introduction of variable exponent Bochner--Lebesgue
	spaces was driven by the inability of usual Bochner--Lebesgue
	spaces to properly encode unsteady problems with variable
	exponent structure, such as the model problem
	\eqref{eq:model}, the notion of generalized time derivative from Definition~\ref{2.1},
	which lives in Bochner--Lebesgue spaces,  is likewise inadequate for the treatment of problems of this nature. Therefore, we next introduce 
	variable exponent Bochner--Sobolev spaces with a symmetric gradient structure, the
	appropriate substitute of customary Bochner--Sobolev spaces (cf.~Definition~\ref{2.1}) for treatment of unsteady problems in variable exponents spaces governed by the symmetric part of the gradient. 
	
	In what follows, let $\Omega\subseteq \setR^d$, $d\ge 2$, be a bounded Lipschitz domain, $I:=\left(0,T\right)$, $T<\infty$, $Q_T:=I\times \Omega$ and $q,p\in \mathcal{P}^{\log}(Q_T)$ with $q^-,p^->1$, such that ${X^{q,p}_-\embedding Y: =L^2(\Omega)^d}$, e.g., if ${p^-\ge \frac{2d}{d+2}}$ or simply $q^-\ge 2$.\newpage
	
	\begin{defn}\label{4.1}  A function $\bsu\in\bscal{X}^{q,p}_{\bfvarepsilon}(Q_T)$ possesses a \textbf{generalized time derivative in} $\bscal{X}^{q,p}_{\bfvarepsilon}(Q_T)^*$, if there exists a functional $\bsu^*\in \bscal{X}^{q,p}_{\bfvarepsilon}(Q_T)^*$,  such that for every $\bfphi\in C^\infty_0(Q_T)^d$ it holds
		\begin{align}
		-\int_I{(\bsu(t),\pa_t\bfphi(t))_Y\,dt}=\langle\bsu^*,\bfphi\rangle_{\bscal{X}^{q,p}_{\bfvarepsilon}(Q_T)}.\label{eq:4.1}
		\end{align}
		In this case, we define $\frac{\textbf{d}\bsu}{\textbf{dt}}:=\bsu^*$ in $\bscal{X}^{q,p}_{\bfvarepsilon}(Q_T)^*$.
	\end{defn}
	
	\begin{rmk}\label{4.2} The generalized time derivative in the sense of Definition~\ref{4.1} is unique, and thus well-defined (cf.~\cite[Lemma 6.3]{KR20}).
	\end{rmk}

	\begin{rmk}\label{4.2.1}
		Denote by $\mathcal{D}'(Q_T)^d$ the space of vector-valued distributions, i.e., the topological dual space of the locally convex Hausdorff vector space $C_0^\infty(Q_T)^d$ (cf.~\cite[Kap. II, 4.]{GGZ74}). Then, the generalized time derivative $\frac{\textbf{d}}{\textbf{dt}}$ from Definition~\ref{4.1} is related to the distributional time derivative $\partial_t$ in the following sense.
		If $\bsu\in\bscal{X}^{q,p}_{\boldsymbol{\varepsilon}}(Q_T)$ has a generalized time derivative $\frac{\bfd\bsu}{\bfd\bft}\in \bscal{X}^{q,p}_{\bfvarepsilon}(Q_T)^*$, then  Proposition~\ref{3.7} provides functions $\bsf\in L^{q'(\cdot,\cdot)}(Q_T)^d$ and $\bsF\in L^{p'(\cdot,\cdot)}(Q_T,\mathbb{M}^{d\times d}_{\sym})$, such that\linebreak${\frac{\bfd\bsu}{\bfd\bft}=\bscal{J}_{\bfvarepsilon}(\bsf,\bsF)}$ in $\bscal{X}^{q,p}_{\bfvarepsilon}(Q_T)^*$. Thus, we easily derive from \eqref{eq:4.1} that $\partial_t\bsu=\boldsymbol{f}-\textup{div}(\boldsymbol{F})$~in~$\mathcal{D}'(Q_T)^d$.
	\end{rmk}
	
	\begin{defn}\label{4.3}
		We define the \textbf{variable exponent Bochner--Sobolev space}
		\begin{align*}
		\bscal{W}^{q,p}_{\bfvarepsilon}(Q_T):=
		\bigg\{\bsu\in \bscal{X}^{q,p}_{\bfvarepsilon}(Q_T)\;\Big|\; \exists\frac{\bfd\bsu}{\mathbf{dt}}\in 	\bscal{X}^{q,p}_{\bfvarepsilon}(Q_T)^*\bigg\}.
		\end{align*}
	\end{defn}
	
	\begin{prop}\label{4.4}
		The space $\bscal{W}_{\bfvarepsilon}^{q,p}(Q_T)$ forms a  separable, reflexive Banach space, if equipped with the norm
		\begin{align*}
		\|\cdot\|_{\bscal{W}_{\bfvarepsilon}^{q,p}(Q_T)}:=\|\cdot\|_{\bscal{X}_{\bfvarepsilon}^{q,p}(Q_T)}+\left\|\frac{\bfd}{\mathbf{dt}}\cdot\right\|_{\bscal{X}_{\bfvarepsilon}^{q,p}(Q_T)^*}.
		\end{align*}
	\end{prop}
	
	\begin{proof}
		See~\cite[Proposition 6.5]{KR20}.\hfill$\qed$
	\end{proof}
	
	We have the following equivalent characterization of $\bscal{W}^{q,p}_{\bfvarepsilon}(Q_T)$.
	\begin{prop}\label{4.5}
		Let $\bsu\in \bscal{X}^{q,p}_{\bfvarepsilon}(Q_T)$ be a function and $\bsu^*\in \bscal{X}^{q,p}_{\bfvarepsilon}(Q_T)^*$ a functional. Then, the following statements are equivalent:
		\begin{description}[{(ii)}]
			\item[(i)]  
			$\bsu\in \bscal{W}^{q,p}_{\bfvarepsilon}(Q_T)$ with $\frac{\bfd\bsu}{\mathbf{dt}}=\bsu^*$ in $\bscal{X}^{q,p}_{\bfvarepsilon}(Q_T)^*$.
			\item[(ii)] For every $\bfv\in X^{q,p}_+$ and
			$\varphi\in C^\infty_0(I)$ there holds
			\begin{align*}
			-\int_I{(\bsu(t),\bfv)_Y\varphi^\prime(t)\,dt}=\int_I{\langle \bsu^*(t),\bfv\rangle_{X^{q,p}_{\bfvarepsilon}(t)}\varphi(t)\,dt}.
			\end{align*}
		\end{description}
	\end{prop}

	\begin{proof}
		See \cite[Proposition 6.8]{KR20}.\hfill$\qed$
	\end{proof}

	The notion of generalized time derivative from Definition~\ref{4.1} admits an appropriate analogue of Proposition~\ref{2.2} 
	in the framework of $\bscal{W}_{\bfvarepsilon}^{q,p}(Q_T)$.
	
	\begin{prop}\label{4.6}
		Let $\Omega\subseteq\setR^d$, $d\ge 2$, be a bounded Lipschitz domain and $q,p\in \mathcal{P}^{\log}(Q_T)$ with $q^-,p^->1$, such that $X^{q,p}_-\hookrightarrow Y$.
		Moreover, let $\bscal{Y}^0(Q_T):=C^0(\overline{I},Y)$. Then, it holds:
		\begin{description}[{(ii)}]
			\item[(i)] Each function $\bsu\in \bscal{W}_{\bfvarepsilon}^{q,p}(Q_T)$ (defined almost everywhere) possesses a unique representation $\bsu_{c}\in \bscal{Y}^0(Q_T)$, and the resulting mapping $(\cdot)_{c}:\bscal{W}_{\bfvarepsilon}^{q,p}(Q_T)\to
			\bscal{Y}^0(Q_T)$ is an embedding.
			\item[(ii)] For every $\bsu,\bsv\in \bscal{W}_{\bfvarepsilon}^{q,p}(Q_T)$ and
			$t,t'\in \overline{I}$ with $t'\leq t$ there holds
			\begin{align*}
			\int_{t'}^{t}{\left\langle
				\frac{\bfd\bsu}{\mathbf{dt}}(s),\bsv(s)\right\rangle_{X^{q,p}_{\bfvarepsilon}(s)}\!\!ds}
			=\left[(\bsu_c(s), 
			\bsv_c(s))_Y\right]^{s=t}_{s=t'}-\int_{t'}^{t}{\left\langle
				\frac{\bfd\bsv}{\mathbf{dt}}(s),\bsu(s)\right\rangle_{X^{q,p}_{\bfvarepsilon}(s)}\!\!ds}.
			\end{align*}
		\end{description}  
	\end{prop}

	\begin{proof}
		See~\cite[Proposition 6.24]{KR20}.\hfill$\qed$
	\end{proof}
	
	\newpage
	\section{Embedding theorems}
	\label{sec:5}
	
	This section is concerned with the proof of fundamental embedding and compactness theorems for variable exponent Bochner--Lebesgue spaces with a symmetric gradient structure. 
	In order to get an appreciation, where the difficulties of a proof of such embedding theorems are hidden in, let us resurvey  Poincar\'e's inequality in the context of variable~exponent~spaces.
	
	\begin{prop}[Poincar\'e's inequality]\label{5.1}
		Let $\Omega\subseteq \setR^d$, $d\ge 2$, be a bounded domain and ${p\in \mathcal{P}^{\infty}(\Omega)\cap C^0(\overline{\Omega})}$. Then, there exists a constant $c>0$, such that for every ${\bfu\in X^{p(\cdot),p(\cdot)}_{\nabla}(\Omega)^d}$
		\begin{align*}
		\|\bfu\|_{L^{p(\cdot)}(\Omega)^d}\leq c\|\nb\bfu\|_{L^{p(\cdot)}(\Omega)^{d\times d}}.
		\end{align*}
	\end{prop}

	\begin{proof}
		See \cite[Theorem 3.10]{KR91} or \cite[Theorem 8.2.18.]{DHHR11}.\hfill$\qed$
	\end{proof}
	
	Proposition \ref{5.1} can be interpreted as a minimal embedding result, as it also states that every function $\bfu\in X^{1,p(\cdot)}_{\nabla}(\Omega)^d$  also satisfies ${\bfu\in X^{p(\cdot),p(\cdot)}_{\nabla}(\Omega)^d}$,  i.e., we~additionally~have~${\bfu\in L^{p(\cdot)}(\Omega)^d}$. This  straightforwardly follows from Proposition \ref{5.1} and the density of $C_0^\infty(\Omega)^d$ in $X^{1,p(\cdot)}_{\nabla}(\Omega)^d$.
	In the case of a constant exponent $p\in\left(1,\infty\right)$, Proposition~\ref{5.1} yields
	for~${\bsu\in L^1(I,X^{1,p}_{\nabla}(\Omega)^d)}$ with $\nb\bsu\in L^{p}(Q_T)^{d\times d}$ that $\bsu\in L^p(Q_T)^d$, i.e., $\bsu\in L^p(I,W^{1,p}_0(\Omega)^d)$.
	Moreover, Proposition~\ref{5.1} guarantees that the gradient norm $\|\nb\cdot\|_{L^{p(\cdot)}(\Omega)^{d\times d}}$, and if additionally $p\in \mathcal{P}^{\log}(\Omega)$~with~${p^->1}$, by virtue of Korn's inequality (cf.~\cite[Thm.~5.5]{DR03}), also $\|\bfvarepsilon(\cdot)\|_{L^{p(\cdot)}(\Omega)^{d\times d}}$,~defines~an~equivalent~norm\\[-2pt] on $X^{p(\cdot),p(\cdot)}_{\nabla}(\Omega)^d$. Our purpose is to gain an analogue of Proposition~\ref{5.1} for variable exponent Bochner--Lebesgue spaces. In this context, we introduce for $p\in \mathcal{P}^{\infty}(Q_T)$ the special space
	\begin{align}
	\bscal{X}^{\circ,p}_{\bfvarepsilon}(Q_T):=\bscal{X}^{p^-,p}_{\bfvarepsilon}(Q_T).\label{eq:W0}
	\end{align}
	If in addition $p^->1$, then 
	by means of Poincar\'e's and Korn's inequality with respect to the constant exponent ${p^-\in \left(1,\infty\right)}$ and \eqref{embed}, we obtain for every $\bsu\in\bscal{X}^{\circ,p}_{\bfvarepsilon}(Q_T)$ the inequality
	\begin{align}
	\|\bsu\|_{L^{p^-}(Q_T)^d}\leq c_{p^-}\|\bfvarepsilon(\bsu)\|_{L^{p^-}(Q_T)^{d\times d}}\leq c_{p^-}2(1+\vert Q_T\vert)\|\bfvarepsilon(\bsu)\|_{L^{p(\cdot,\cdot)}(Q_T)^{d\times d}},\label{eq:W1}
	\end{align}
	i.e., $\|\bfvarepsilon(\cdot)\|_{L^{p(\cdot,\cdot)}(Q_T)^{d\times d}}$ defines an equivalent norm on $\bscal{X}^{\circ,p}_{\bfvarepsilon}(Q_T)$. We emphasize that~inequality \eqref{eq:W1} already resembles an analogue of Proposition~\ref{5.1} for variable exponent Bochner--Lebesgue spaces, though it is much weaker, as we actually desire an inequality of type \eqref{eq:W1} in which the constant exponent $p^-\in \left(1,\infty\right)$ is replaced by the variable exponent $p\in \mathcal{P}^{\log}(Q_T)$, i.e., the embedding $\bscal{X}^{\circ,p}_{\bfvarepsilon}(Q_T)\embedding L^{p(\cdot,\cdot)}(Q_T)^d$. Regrettably, we cannot hope for such an inequality, even if the variable exponent is smooth and not depending on time. 
	
	\begin{rmk}[Invalidity of Poincar\'e's inequality on 	$\bscal{X}^{\circ,p}_{\bfvarepsilon}(Q_T)$]\label{5.2}
		Let $\Omega\subseteq \setR^d$, $d\ge 2$, be an arbitrary bounded  domain and $I:=\left(0,T\right)$, $T<\infty$, and $Q_T:=I\times \Omega$. Moreover, let
		\begin{align*}
		\bscal{F}_{\textup{Poin}}:=\big\{(\bfu,\nb \bfu)^\top\mid \bfu\in C_0^\infty(\Omega)^d\big\}\subseteq C_0^\infty(\Omega)^d\times C_0^\infty(\Omega)^{d\times d}.
		\end{align*}
		Let $\bfu\in C_0^\infty(\Omega)^d$ with $\bfu\equiv \bfa$ in $G$, where $G\subset\subset \Omega$ is a domain and $\bfa\in \setR^d\setminus\{\mathbf{0}\}$. Then, we have
		$(\bfu,\nb\bfu)^\top\in \bscal{F}_{\textup{Poin}}$ with $\nb\bfu\not\equiv\mathbf{0}$ and  $\textup{int}(\textup{supp}(\bfu))\setminus\textup{supp}(\nb\bfu)\neq \emptyset$. In consequence, according to \cite[Proposition 4.4]{KR20}, there exists  an exponent $p\in C^\infty(\setR^d)$ with $p^->1$, which does not admit a constant $c>0$, such that for every $\bfphi\in C_0^\infty(Q_T)^d$ there holds
		\begin{align*}
		\|\bfphi\|_{L^{p(\cdot)}(Q_T)^d}\leq c\|\nb\bfphi\|_{L^{p(\cdot)}(Q_T)^{d\times d}}.
		\end{align*}
		In addition, \cite[Proposition 4.4]{KR20} provides for every $\varphi\in L^{p^-}(I)\setminus L^{p^+}(I)$ that $\varphi\bfu\in L^{p^-}(Q_T)^d$ and $\varphi\nabla\bfu\in L^{p(\cdot)}(Q_T)^{d\times d}$, but $\varphi\bfu\notin L^{p(\cdot)}(Q_T)^d$, i.e., we have $\varphi\bfu\in \bscal{X}^{\circ,p}_{\bfvarepsilon}(Q_T)\setminus L^{p(\cdot)}(Q_T)^d$. Therefore, we have  $\bscal{X}^{\circ,p}_{\bfvarepsilon}(Q_T)\not\subseteq L^{p(\cdot)}(Q_T)^d$.
		
		The situation is illustrated in Figure~\ref{poincare} for $\Omega=B_{2.5}^2(0)$, $G=B_{0.6}^2(0)$, $\bfu=\bfa\eta\in C_0^\infty(\Omega)^2$, where  $\bfa=\mathbf{e}_1=(1,0)^\top\in \setR^2$ and $\eta=\chi_{B_1^2(0)}\ast \omega_\vep\in C_0^\infty(\Omega)$ for $\vep=0.4$. Here, $ \omega_\vep\in C_0^\infty(B_\vep^2(0))$ denotes the scaled standard mollifier, which is for every $x\in \setR^2$ given via
		$\omega_\vep(x):=\frac{1}{\vep^2}\omega(\frac{x}{\vep})$, where $\omega(x):=\exp\big(\frac{-1}{1-\vert x\vert^2}\big)$, if $x\in B_1^2(0)$, and $\omega(x):=0$, if $x\in B_1^2(0)^c$.\vspace*{-1cm}
	\end{rmk}
	\newpage
	\begin{figure}[ht]
		\centering
		\vspace*{-2mm}
		\hspace*{-5mm}\includegraphics[width=15.2cm]{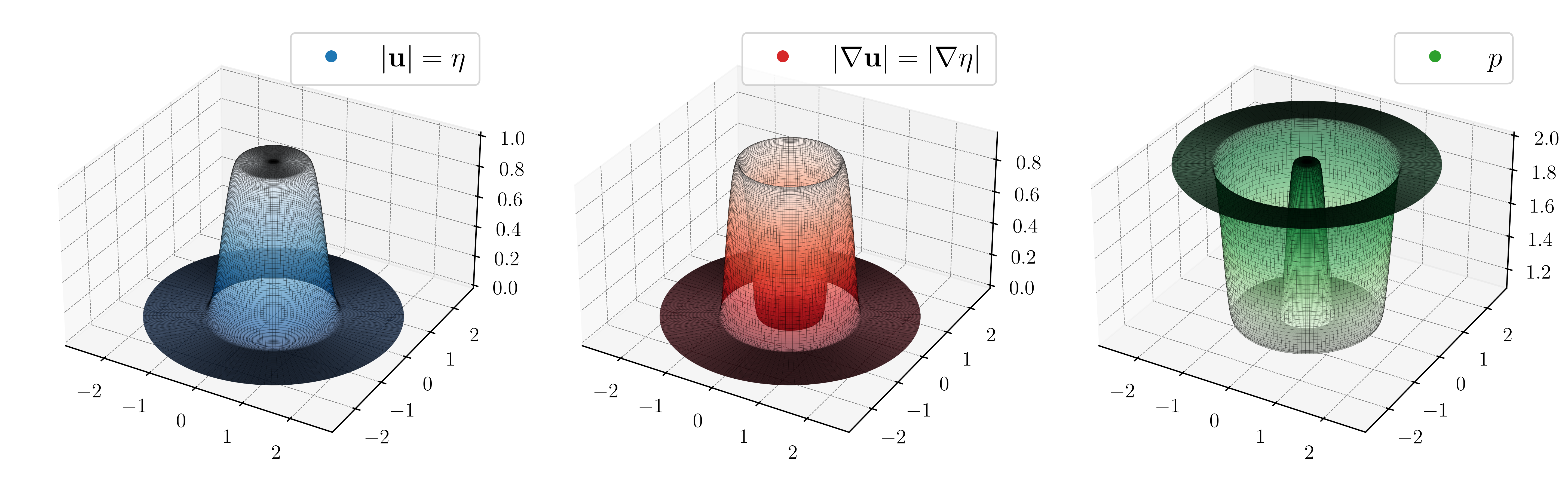}
		\caption{Plots of $\vert \bfu\vert =\eta\in C_0^\infty(\Omega)$ (blue/left), $\vert \nabla\bfu\vert =\vert\nabla\eta\vert\in C_0^\infty(\Omega)$ (red/middle) and $p\in C^\infty(\overline{\Omega})$ (green/right) constructed according to \cite[Proposition 4.4]{KR20} for $d=2$, ${p^-=1.1}$~and~${p^+=2}$.}
		\label{poincare}
	\end{figure}

	Fortunately, in unsteady problems, even with variable exponent structure, the appearance of a time derivative usually provides some additional regularity in time. In fact, one can expect a solution to lie additionally in the space
	\begin{align}
	\bscal{Y}^0(Q_T):=C^0(\overline{I},Y)\qquad\text{ or at least }\qquad\bscal{Y}^\infty(Q_T):=L^\infty(I,Y),\label{eq:c0y}
	\end{align}
	where $Y:=L^2(\Omega)^d$.
	Taking into account this additional regularity beforehand, we will manage to gain an embedding result, which provides the desired extension of Proposition~\ref{5.1} to the framework of variable exponent Bochner--Lebesgue spaces. An essential tool for the incorporation of the additional  $\bscal{Y}^\infty(Q_T)$--information is the following  Gagliardo--Nirenberg~interpolation~inequality.
	
	\begin{lem}[Gagliardo--Nirenberg interpolation inequality]\label{5.3}
		Let $\Omega\subseteq \setR^d$, $d\ge 2$, be a bounded Lipschitz domain and $s,r \in\left[1,\infty\right]$. Then, for every $\bfu\in L^s(\Omega)^d$ with $\nb \bfu\in L^r(\Omega)^{d\times d}$ it holds $\bfu\in L^q(\Omega)^d$~with
		\begin{align}
		\| \bfu\|_{L^q(\Omega)^d}\leq c_1\|\nb \bfu\|_{L^r(\Omega)^{d\times d}}^\theta\|\bfu\|_{L^s(\Omega)^d}^{1-\theta}+c_2\|\bfu\|_{L^s(\Omega)^d},\label{GN}
		\end{align}
		where $c_1=c_1(s,r,\Omega),c_2=c_2(s,r,\Omega)>0$ and $\frac{1}{q}=\theta(\frac{1}{r}-\frac{1}{d})+(1-\theta)\frac{1}{s}$
		for every $\theta\in \left[0,1\right]$, unless $1<r<\infty$ and $1-\frac{d}{r}\in \mathbb{N}\cup\{0\}$, in which \eqref{GN} only holds for $\theta<1$.
	\end{lem}
	
	\begin{proof}
			See \cite{Gag59}, \cite{Nir59} or \cite{Nir66}.\hfill$\qed$
	\end{proof}
	
 	Because the Gagliardo--Nirenberg interpolation inequality only takes the full gradient into account, but we have solely control over the symmetric part of the gradient, we need to  switch locally from the full gradient to the symmetric gradient via Korn's second inequality.
	
	\begin{lem}[Korn's second inequality]\label{5.4}
		Let $\Omega\subseteq\setR^d$, $d\ge 2$, be a bounded Lipschitz domain and $s\in \left(1,\infty\right)$. Then, there exists a constant $c>0$, such that for every $\bfu\in L^s(\Omega)^d$ with $\boldsymbol{\varepsilon}(\mathbf{u})\in L^s(\Omega)^{d\times d}$  it holds $\bfu\in W^{1,s}(\Omega)^d$ with
		\begin{align*}
		\|\bfu\|_{W^{1,s}(\Omega)^d}\leq c\big(\|\bfu\|_{L^s(\Omega)^d}+\|\boldsymbol{\varepsilon}(\mathbf{u})\|_{L^s(\Omega)^{d\times d}}\big).
		\end{align*}
	\end{lem}
	
	\begin{proof}
		See \cite[Lem. 3]{Fu94}, \cite[Thm. 1.10]{MNRR96}, \cite[Thm. 5.17]{DRS10} or \cite[Thm. 14.3.23]{DHHR11}.\hfill$\qed$
	\end{proof}
	
	By means of Lemma~\ref{5.3} and Lemma~\ref{5.4} we are able to prove the main result of this section.
	
	\begin{prop}[Poincar\'e's inequality for $\bscal{X}^{2,p}_{\bfvarepsilon}(Q_T)\cap\bscal{Y}^\infty(Q_T)$]\label{5.5}
		Let $\Omega\subseteq \setR^d$, $d\ge 2$, be a bounded domain, $I:=\left(0,T\right)$, $T<\infty$, $Q_T:=I\times \Omega$ and $p\in C^0(\overline{Q_T})$ with $p^->1$. Then,  there exists a constant $c>0$, such that for every $\bsu\in \bscal{X}^{2,p}_{\bfvarepsilon}(Q_T)\cap\bscal{Y}^\infty(Q_T)$~it~holds~$\bsu\in L^{p(\cdot,\cdot)}(Q_T)^d$~with
		\begin{align}
		\|\bsu\|_{L^{p(\cdot,\cdot)}(Q_T)^d}\leq c\big(\|\bfvarepsilon(\bsu)\|_{L^{p(\cdot,\cdot)}(Q_T)^{d\times d}}+\|\bsu\|_{\bscal{Y}^\infty(Q_T)}\big),\label{eq:2.64.a}
		\end{align}
		i.e., there holds the embedding  $\bscal{X}^{2,p}_{\bfvarepsilon}(Q_T)\cap\bscal{Y}^\infty(Q_T)\embedding L^{p(\cdot,\cdot)}(Q_T)^d$.
	\end{prop}
	
	More precisely, Proposition~\ref{5.5} is basically  a consequence of the following stronger lemma.
	
	\begin{lem}\label{5.6}
		Let $\Omega\subseteq \setR^d$, $d\ge 2$, be a bounded domain, $I:=\left(0,T\right)$, $T<\infty$, ${Q_T:=I\times\Omega}$ and $p\in C^0(\overline{Q_T})$ with $p^->1$. Then, exist  constants $c:=c(p,\Omega,I),\gamma:=\gamma(p,\Omega,I)>0$ (both not depending on $t\in I$), such that for every $t\in I$ and $\bfu\in X^{2,p}_{\boldsymbol{\varepsilon}}(t)$ it holds $\bfu\in L^{p(t,\cdot)}(\Omega)^d$ with
		\begin{align}
		\rho_{p(t,\cdot)}(\bfu)\leq c\big[1+\rho_{p(t,\cdot)}(\bfvarepsilon(\bfu))+\|\bfu\|_Y^{\gamma}\big].\label{eq:2.64.b}
		\end{align}
	\end{lem}
	
	\begin{proof}
		We split the proof into two main steps.
		
		\textbf{Step 1:} 
		Let $\overline{p}\in C^0(\setR^{d+1})$ be an extension of $p\in C^0(\overline{Q_T})$ with $p^-\leq \overline{p}\leq p^+$ in~$\setR^{d+1}$ (cf.~Proposition~\ref{2.13.1}). We construct a finite open covering of $\overline{Q_T}$~by~cylinders~${(Q_{ij})_{i,j=1,...,m}}$,~${m\!\in\! \setN}$, where $Q_{ij}=I_i\times B_j$ for open intervals $I_i$, $i=1,...,m$, and open balls $B_j$, $j=1,...,m$, such that the local exponents 
		$p^+_{ij}:=\sup_{(s,y)^\top\in Q_{ij}}{\overline{p}(s,y)}$ and $p^-_{ij}:=\inf_{(s,y)^\top\in Q_{ij}}{\overline{p}(s,y)}$ satisfy
		\begin{align}
		p^+_{ij}<p^-_{ij}\Big(1+\frac{2}{d}\Big)\quad\text{ for all }i,j=1,...,m.\label{eq:5.6.1}
		\end{align}
		Since  $\overline{p}\in C^0(\setR^{d+1})$, i.e., $\overline{p}:\overline{Q_T}\to \left(1,\infty\right)$ is uniformly continuous, there exists a~radius~${\rho>0}$, such that for every
		$z:=(t,x)^\top\in \overline{Q_T}$ the local exponents $p_z^+:=\sup_{(s,y)^\top\in Q_z^\rho}{\overline{p}(s,y)}$ and ${p_z^-:=\inf_{(s,y)^\top\in Q_z^\rho}{\overline{p}(s,y)}}$, where $Q_z^\rho:=B^1_{\rho}(t)\times B^d_{\rho}(x)$, satisfy $p_z^+<p_z^-(1+\frac{2}{d})$.~On~the~other~hand, as $\overline{I}$ and $\overline{\Omega}$ are compact, there exist finitely many $t_i\in \overline{I}$, $i=1,...,m$,  $m\in \setN$, and $x_j\in \overline{\Omega}$, $j=1,...,m$, such that the invervals $I_i:=B_\rho^1(t_i)$, $i=1,...,m$, form an open covering of $\overline{I}$ and  the balls $B_j=B_{\rho}^d(x_i)$, $j=1,...,m$, form an open covering of $\overline{\Omega}$. Putting all together, the cylinders $Q_{ij}:=I_i\times B_j$, $i,j=1,...,m$, form an open covering of $\overline{Q_T}$, such that the local exponents $p^+_{ij}:=\sup_{(s,y)^\top\in Q_{ij}}{\overline{p}(s,y)}$ and $p^-_{ij}:=\inf_{(s,y)^\top\in Q_{ij}}{\overline{p}(s,y)}$  satisfy \eqref{eq:5.6.1}.
		
		\textbf{Step 2:} We fix an arbitrary $t\in I$. Then, there exists some index $i=1,...,m$,~such~that~$t\in I_i$. Due to the density of $C_0^\infty(\Omega)^d$ in $X^{2,p}_{\boldsymbol{\varepsilon}}(t)$, it suffices to consider $\bfu\in C_0^\infty(\Omega)^d$. Let us fix an arbitrary ball $B_j$  from step~\textbf{1} for some $j=1,...,m$. There are two possibilities. First, we consider the case $p_{ij}^+\leq 2$. Then, using that $a^{\overline{p}(t,x)}\leq (1+a)^2\leq 2+2a^2$ for all $a\ge 0$ and $x\in B_j$, and $\|\bfu\|_{L^2(B_j)^d}\leq\|\bfu\|_Y$,~we~obtain 
		\begin{align}
		\rho_{\overline{p}(t,\cdot)}(\bfu\chi_{B_j})\leq 2\vert B_j\vert +2\|\bfu\|_{L^2(B_j)^d}^2\leq 2\vert \Omega\vert +2\|\bfu\|_Y^2.\label{eq:5.main.2}
		\end{align}
		Next, assume that $p_{ij}^+> 2$. Then, exploiting that $p_{ij}^->p_{ij}^+\frac{d}{d+2}>\frac{2d}{d+2}$ (cf.~\eqref{eq:5.6.1}), i.e., $\frac{d-p_{ij}^-}{dp_{ij}^-}<\frac{1}{2}$,\\[-2pt] there holds
		\begin{align}
		0<\theta_{ij}:=\frac{\frac{1}{2}-\frac{1}{p_{ij}^+}}{\frac{1}{2}-\frac{d-p_{ij}^-}{dp_{ij}^-}}=\frac{p_{ij}^-}{p_{ij}^+}\frac{d(p_{ij}^+-2)}{d(p_{ij}^--2)+2p_{ij}^-}<\frac{p_{ij}^-}{p_{ij}^+}\frac{d(p_{ij}^-+\frac{2p_{ij}^-}{d}-2)}{d(p_{ij}^--2)+2p_{ij}^-}=\frac{p_{ij}^-}{p_{ij}^+}\leq 1.\label{eq:5.main.3}
		\end{align}
		By  Gagliardo--Nirenberg's interpolation inequality (cf.~Lemma~\ref{5.3}), because ${\theta_{ij}< 1}$~(cf.~\eqref{eq:5.main.3}),\\[-1pt] we obtain  constants $c_{ij}=c_{ij}(p^-_{ij},\Omega), \theta_{ij}:=\theta_{ij}(p^-_{ij},\Omega)>0$,~such~that
	\begin{align}
		\rho_{p_{ij}^+}(\bfu\chi_{B_j})\leq c_{ij}\|\nb\bfu\|^{p_{ij}^+\theta_{ij}}_{L^{p_{ij}^-}(B_j)^{d\times d}}\|\bfu\|^{p_{ij}^+(1-\theta_{ij})}_{L^2(B_j)^d}+c_{ij}\|\bfu\|^{p_{ij}^+}_{L^2(B_j)^d}.\label{eq:5.main.4}
	\end{align}
		Since $\pa B_j\in C^\infty$, Korn's second inequality (cf.~Lemma~\ref{5.4}) yields a  constant $c_{ij}=c_{ij}(p^-_{ij},\Omega)>0$, such that 
	\begin{align}
		\|\nb\bfu\|_{L^{p_{ij}^-}(B_j)^{d\times d}}\leq c_{ij}\Big[\|\bfu\|_{L^{p_{ij}^-}(B_j)^d}+\|\bfvarepsilon(\bfu)\|_{L^{p_{ij}^-}(B_j)^{d\times d}}\Big].\label{eq:5.main.5}
	\end{align}
		\newpage
		
		\hspace*{-5mm}By inserting \eqref{eq:5.main.5} in \eqref{eq:5.main.4}, also using anew that $\|\bfu\|_{L^2(B_j)^d}\leq\|\bfu\|_Y$, we  gain
		\begin{align}
		\begin{split}
		\rho_{p_{ij}^+}(\bfu\chi_{B_j})\leq c_{ij}\Big[\|\bfu\|_{L^{p_{ij}^-}(B_j)^d}+\|\bfvarepsilon(\bfu)\|_{L^{p_{ij}^-}(B_j)^{d\times d}}\Big]^{p_{ij}^+\theta_{ij}}\|\bfu\|_Y^{p_{ij}^+(1-\theta_{ij})}+c_{ij}\|\bfu\|_Y^{p_{ij}^+}.
		\end{split}\label{eq:5.main.6}
		\end{align}
		Since by construction $p_{ij}^+\theta_{ij}<p_{ij}^-$ (cf.~\eqref{eq:5.main.3}), we can apply the $\vep$--Young inequality with respect to the exponent $\rho_{ij}:=\frac{p_{ij}^-}{p_{ij}^+\theta_{ij}}>1$ with constant $c_{ij}(\vep):=\frac{(\rho_{ij}\vep)^{1-\rho_{ij}'}}{\rho_{ij}'}>0$~for~all~$\vep>0$~in~\eqref{eq:5.main.6}. In doing so, using  ${(a+b)^{p^-_{ij}}\!\leq\! 2^{p^+}(a^{p^-_{ij}}+b^{p^-_{ij}})}$ and $a^{p^-_{ij}}\leq 2^{p^+}(1+a^{p^+_{ij}})$ for all $a,b\ge 0$, we deduce 
		\begin{align}
		\rho_{p_{ij}^+}(\bfu\chi_{B_j})&\leq c_{ij}\vep\Big[\|\bfu\|_{L^{p_{ij}^-}(B_j)^d}+\|\bfvarepsilon(\bfu)\|_{L^{p_{ij}^-}(B_j)^{d\times d}}\Big]^{p_{ij}^-}+c_{ij}(\vep)\|\bfu\|_Y^{p_{ij}^+(1-\theta_{ij})\rho_{ij}'}+c_{ij}\|\bfu\|_Y^{p_{ij}^+}\notag\\[-5pt]&\leq c_{ij}\vep 2^{p^+}\big[\rho_{p_{ij}^-}(\bfu\chi_{B_j})+\rho_{p_{ij}^-}(\bfvarepsilon(\bfu)\chi_{B_j})\big]+c_{ij}(\vep)\|\bfu\|_Y^{p_{ij}^+(1-\theta_{ij})\rho_{ij}'}+c_{ij}\|\bfu\|_Y^{p_{ij}^+}
		\label{eq:5.main.7}\\&\leq c_{ij}\vep 2^{2p^+}\big[\vert B_j\vert+\rho_{p_{ij}^+}(\bfu\chi_{B_j})+\rho_{p_{ij}^-}(\bfvarepsilon(\bfu)\chi_{B_j})\big]+ c_{ij}(\vep)\|\bfu\|_Y^{p_{ij}^+(1-\theta_{ij})\rho_{ij}'}+c_{ij}\|\bfu\|_Y^{p_{ij}^+}.\notag
		\end{align}
		We set $c_0:=\max_{i,j=1,...,m}{c_{ij}}$ and $c(\vep):=\max_{i,j=1,...,m}{c_{ij}(\vep)}$. Then, choosing $\vep:=\frac{2^{-2p^+-1}}{c_0}>0$ and absorbing $c_{ij}\vep 2^{2p^+}\rho_{p_{ij}^+}(\bfu\chi_{B_j})\!\leq\! \frac{1}{2}\rho_{p_{ij}^+}(\bfu\chi_{B_j})$ in the left-hand side in \eqref{eq:5.main.7}, we infer~from~\eqref{eq:5.main.7}
		\begin{align}
		\rho_{p_{ij}^+}(\bfu\chi_{B_j})\leq 
		\vert \Omega\vert +\rho_{p_{ij}^-}(\bfvarepsilon(\bfu)\chi_{B_j})+2c(\vep)\|\bsu\|_Y^{p_{ij}^+(1-\theta_{ij})\rho_{ij}'}+2c_0\|\bsu\|_Y^{p_{ij}^+}.\label{eq:5.main.8}
		\end{align}
		If we set $\gamma\!:=\!\max_{i,j=1,...,m}{p_{ij}^+(1-\theta_{ij})\rho_{ij}'+p_{ij}^+}$ and use $\alpha^{p_{ij}^+(1-\theta_{ij})\rho_{ij}'}+\alpha^{p_{ij}^+}\!\leq\!  2^{\gamma+1}(1+\alpha^\gamma)$~for~${\alpha\!\ge \! 0}$, $\rho_{p(t,\cdot)}(\bfu\chi_{B_j\cap \Omega})\!=\!\rho_{\overline{p}(t,\cdot)}(\bfu\chi_{B_j})\!\leq\! 2^{p^+}(\vert\Omega\vert+\rho_{p_{ij}^+}(\bfu\chi_{B_j}))$,  $\rho_{p_{ij}^-}(\bfvarepsilon(\bfu)\chi_{B_j})
		\!\leq\! 2^{p^+}(\vert\Omega\vert +\rho_{\overline{p}(t,\cdot)}(\bfvarepsilon(\bfu)\chi_{B_j}))$ and 
		$\rho_{\overline{p}(t,\cdot)}(\bfvarepsilon(\bfu)\chi_{B_j})\leq \rho_{p(t,\cdot)}(\bfvarepsilon(\bfu))$ in \eqref{eq:5.main.8}, then we further deduce
		\begin{align}
			\begin{split}
		\rho_{p(t,\cdot)}(\bfu\chi_{B_j\cap \Omega})&\leq 2^{p^+}\big(\vert\Omega\vert+\rho_{p_{ij}^+}(\bfu\chi_{B_j})\big)\\[-3pt]&
		\leq 2^{p^+}\big(2\vert\Omega\vert+\rho_{p_{ij}^-}(\bfvarepsilon(\bfu)\chi_{B_j})+2c(\vep)\|\bfu\|_Y^{p_{ij}^+(1-\theta_{ij})\rho_{ij}'}+2c_0\|\bfu\|_Y^{p_{ij}^+}\big)
		\\&
		\leq 2^{p^+}\big(2\vert\Omega\vert +2^{p^+}\big(\vert\Omega\vert +\rho_{p(t,\cdot)}(\bfvarepsilon(\bfu))\big)+(c(\vep)+c_0)2^{\gamma+2}(1+\|\bfu\|_Y^{\gamma})\big).	\end{split}\label{eq:5.main.9}
		\end{align}
		Eventually, if we sum up the inequalities \eqref{eq:5.main.2} and \eqref{eq:5.main.9} with respect to $j=1,...,m$, then we conclude for every $t\in I$ the desired inequality \eqref{eq:2.64.b}.\hfill$\qed$
	\end{proof}

	\begin{proof} (of Proposition~\ref{5.5}) 
		Let $\bsu\in \bscal{X}^{2,p}_{\bfvarepsilon}(Q_T)\cap\bscal{Y}^\infty(Q_T)$, i.e., we have $\bsu(t)\in X^{2,p(t,\cdot)}_{\bfvarepsilon}(\Omega)$ for almost every $t\in I$.  Therefore, Lemma \ref{5.6} provides constants $c,\gamma>0$ (not depending on $t\in I$), such that for almost every $t\in I$ there holds $\bsu(t)\in L^{p(t,\cdot)}(\Omega)^d$ with
		\begin{align}
			\rho_{p(t,\cdot)}(\bsu(t))\leq c\big[1+\rho_{p(t,\cdot)}(\bfvarepsilon(\bsu)(t))+\|\bsu(t)\|_Y^\gamma\big].\label{eq:5.5.1}
		\end{align}
		By the Fubini--Tonelli theorem (cf.~\cite[§2., Satz 2.1,~(a)]{Els05}) the function~${(t\mapsto \rho_{p(t,\cdot)}(\bsu(t)))\!:\!I\!\to\! \mathbb{R}_{\ge 0}}$ is Lebesgue measurable. Moreover,
		due to  \eqref{eq:5.5.1} it has an $L^1(I)$--integrable majorant. Therefore, we have $(t\mapsto \rho_{p(t,\cdot)}(\bsu(t)))\in L^1(I)$ with
		\begin{align}
					\int_I{\int_{\Omega}{\vert \boldsymbol{u}(t,x)\vert^{p(t,x)}\,dx}\,dt}=\int_I{\rho_{p(t,\cdot)}(\bsu(t))\,dt}\leq c\big[1+\rho_{p(\cdot,\cdot)}(\bfvarepsilon(\bsu))+T\|\bsu\|_{\bscal{Y}^\infty(Q_T)}^\gamma\big].\label{eq:5.5.2}
		\end{align}
		The Fubini--Tonelli theorem (cf.~\cite[§2., Satz 2.1,~(c)]{Els05}) applied on \eqref{eq:5.5.2} yields ${\vert \boldsymbol{u}\vert^{p(\cdot,\cdot)}\in L^1(Q_T)}$, i.e., $\bsu\!\in\! L^{p(\cdot,\cdot)}(Q_T)^d$ with $\rho_{p(\cdot,\cdot)}(\bsu)\!\leq\! c\big[1+\rho_{p(\cdot,\cdot)}(\bfvarepsilon(\bsu))+T\|\bsu\|_{\bscal{Y}^\infty(Q_T)}^\gamma\big]$. In other words,~there~holds\\[-1pt] the embedding $\bscal{X}^{\circ,p}_{\bfvarepsilon}(Q_T)\cap\bscal{Y}^\infty(Q_T)\embedding L^{p(\cdot,\cdot)}(Q_T)^d$. In particular, there exists a constant $c>0$, such that the inequality \eqref{eq:2.64.a} holds, for which we also apply~the~inequality~\eqref{eq:W1}.\hfill$\qed$
	\end{proof}
	
	Proposition~\ref{5.5} can also be interpreted as a variable exponent version of the well-known parabolic interpolation inequality, which likewise exploits additional $\bscal{Y}^\infty(Q_T)$--information to improve parabolic embedding results via interpolation arguments.\vspace*{-0.5cm}
	\newpage
	
	\begin{defn}\label{5.7}
		For $d\!\in\!\setN$ and $p\!\in\! \left[1,\infty\right)$, the \textbf{parabolic interpolation exponent}~is~defined~by
		\begin{align}
		p_*:=\begin{rcases}\begin{dcases}
		p\frac{d+2}{d}&\textup{ if }p<d\\
		p+2&\textup{ if }p\ge d
		\end{dcases}\end{rcases}\in\left(1,\infty\right).\label{eq:pie}
		\end{align}
	\end{defn}
	
	\begin{prop}[Parabolic interpolation inequality]\label{5.8}
		Let $\Omega\subseteq \setR^d$, $d\ge 2$, be a bounded domain, $I:=\left(0,T\right)$, $T<\infty$, $Q_T:=I\times \Omega$ and $p\in \left(1,\infty\right)$ constant. 
		If we have $p\neq d$, then for every $s\in \left[1,p_*\right]$ there exist constants $c=c(d,p,s,\Omega),\gamma=\gamma(d,p,s)>0$, such that for  every ${\bsu\in L^p(I,W^{1,p}_0(\Omega)^d)\cap L^\infty(I,L^2(\Omega)^d)}$ it holds $\bsu\in L^s(Q_T)^d$ with 
		\begin{align}
		\|\bsu\|_{L^{s}(Q_T)^d}^s\leq c \|\bsu\|_{L^p(I,W^{1,p}_0(\Omega)^d)}^p\|\bsu\|_{L^\infty(I,L^2(\Omega)^d)}^\gamma.\label{eq:5.8.a}
		\end{align}
		If we have $p= d$, then for every $s\in\left[1,p^*\right)$ there exist constants $c=c(d,p,\Omega),\gamma=\gamma(d,p,s)>0$, such that \eqref{eq:5.8.a} holds, unless $p=d=2$,
		in which \eqref{eq:5.8.a} holds for every $s\in\left[1,p^*\right]$. In addition, if $\partial\Omega\in C^1$, then all statements hold true with $W^{1,p}_0(\Omega)^d$ replaced by $W^{1,p}(\Omega)^d$.
	\end{prop}
	
	\begin{proof}
		First, let $\bfu\in W^{1,p}_0(\Omega)^d\cap L^2(\Omega)^d$ be arbitrary. We distinguish between four cases:
		
		\textbf{(p<d)} We have $p_*=p\frac{d+2}{d}$. By interpolation with $\frac{1}{p_*}=\frac{1-\theta}{p^*}+\frac{\theta}{2}$, where~${\theta=\frac{2}{d+2}}$~and~${p^*=\frac{dp}{d-p}}$, also using the embedding $W^{1,p}_0(\Omega)^d\embedding L^{p^*}(\Omega)^d$, we obtain 
		\begin{align}
		\|\bfu\|_{L^{p_*}(\Omega)^d}^{p_*}\leq \|\bfu\|_{L^{p^*}(\Omega)^d}^{p_*\frac{d}{d+2}}\|\bfu\|_{L^2(\Omega)^d}^{p_*\frac{2}{d+2}}\leq c_p\|\bfu\|_{W^{1,p}_0(\Omega)^d}^p\|\bfu\|_{L^2(\Omega)^d}^{\frac{2p}{d}}.\label{eq:5.8.1}
		\end{align}
		
		\textbf{(p=d)} We have $p_*=p+2$. Let $s\in \left(p,p+2\right)$. By interpolation with $\frac{1}{s}=\frac{1-\theta}{q}+\frac{\theta}{2}$, where $\theta=1-\frac{p}{s}$ and $q=\frac{2p}{2-s+p}$, also using the embedding $W^{1,p}_0(\Omega)^d\embedding L^q(\Omega)^d$, we obtain
		\begin{align}
		\|\bfu\|_{L^s(\Omega)^d}^s\leq \|\bfu\|_{L^q(\Omega)^d}^p\|\bfu\|_{L^2(\Omega)^d}^{s-p}\leq c_p\|\bfu\|_{W^{1,p}_0(\Omega)^d}^p\|\bfu\|_{L^2(\Omega)^d}^{s-p}.\label{eq:5.8.2}
		\end{align}
		
		\textbf{(p=d=2)} We have $p_*=4$. Due to \cite[Proposition III.2.35 \& Remark III.2.17]{BF13}, there exists a constant $c_2>0$, such that
		\begin{align}
			\|\bfu\|_{L^4(\Omega)^d}^4\leq c_2\|\bfu\|_{W^{1,2}_0(\Omega)^d}^2\|\bfu\|_{L^2(\Omega)^d}^2.\label{eq:5.8.2.1}\vspace*{-1mm} 
		\end{align}
		
		\textbf{(p>d)} We have $p_*=p+2$. By the embedding $W^{1,p}_0(\Omega)^d\embedding L^\infty(\Omega)^d$~we~immediately~obtain
		\begin{align}
		\|\bfu\|_{L^{p_*}(\Omega)^d}^{p_*}\leq\|\bfu\|_{L^\infty(\Omega)}^p\|\bfu\|_{L^2(\Omega)^d}^2\leq c_p\|\bfu\|_{W^{1,p}_0(\Omega)^d}^p\|\bfu\|_{L^2(\Omega)^d}^2.\label{eq:5.8.3}
		\end{align}
		Thus, the assertion for $\bsu\in L^p(I,W^{1,p}_0(\Omega)^d)\cap L^\infty(I,L^2(\Omega)^d)$ easily follows from \eqref{eq:5.8.1}--\eqref{eq:5.8.3}, respectively, by integration  with respect to $t\in I$. \hfill$\qed$
	\end{proof}
	
	Needless to say, we should not be satisfied with Proposition~\ref{5.5}, as it provides no higher integrability than the given variable exponent $p\in \mathcal{P}^{\log}(Q_T)$, whereas Proposition~\ref{5.8} improves the integrability at least up to the constant interpolation exponent $p_*\ge \frac{d+2}{d}p> p$. Therefore, our next objective is to prove an analogue of Proposition~\ref{5.8} involving variable exponents. To this end, let us first introduce a variable exponent version of the parabolic interpolation exponent.
	
	\begin{defn}\label{5.9}
		Let $\Omega\subseteq \setR^d$, $d\ge 2$, be a bounded domain, $I:=\left(0,T\right)$,  $T<\infty$, $Q_T:=I\times \Omega$ and $p\in \mathcal{P}^{\infty}(Q_T)$. Let the mapping $(\cdot)_*:\left(1,\infty\right)\to \setR$ be defined through \eqref{eq:pie}.
		Then, we denote by $p_*:=(\cdot)_*\circ p\in \mathcal{P}^{\infty}(Q_T)$ the \textbf{variable parabolic interpolation exponent}.
	\end{defn}
	
	\begin{prop}[Interpolation inequality for $\bscal{X}^{2,p}_{\bfvarepsilon}(Q_T)\cap \bscal{Y}^\infty(Q_T)$]\label{5.10}
		Let $\Omega\subseteq \setR^d$, $d\ge 2$, be a bounded domain, $I:=\left(0,T\right)$, $T<\infty$, $Q_T:=I\times \Omega$ and $p\in C^0(\overline{Q_T})$ with $p^->1$. Then, for every $\vep\in (0,(p^-)_*-1]$ there exists a constant $c_\vep=c(\vep,d,p,\Omega)>0$, such that for every $\bsu\in \bscal{X}^{2,p}_{\bfvarepsilon}(Q_T)\cap \bscal{Y}^\infty(Q_T)$ it holds $\bsu\in L^{p_*(\cdot,\cdot)-\vep}(Q_T)^d$ with 
		\begin{align*}
		\|\bsu\|_{L^{p_*(\cdot,\cdot)-\vep}(Q_T)^d}\leq c_\vep\big( \|\bfvarepsilon(\bsu)\|_{L^{p(\cdot,\cdot)}(Q_T)^{d\times d}}+\|\bsu\|_{\bscal{Y}^\infty(Q_T)}\big),
		\end{align*}
		i.e., there holds the embedding  $\bscal{X}^{2,p}_{\bfvarepsilon}(Q_T)\cap \bscal{Y}^\infty(Q_T)\embedding L^{p_*(\cdot,\cdot)-\vep}(Q_T)^d$ for all $\vep\in (0,(p^-)_*-1]$.
	\end{prop}
	
	Similar to Proposition~\ref{5.5}, Proposition~\ref{5.10} follows from the following stronger lemma.
	
	\begin{lem}\label{5.11}
		Let $\Omega\subseteq \setR^d$, $d\ge 2$, be a bounded domain, $I:=\left(0,T\right)$,  $T<\infty$,~${Q_T:=I\times\Omega}$,  \linebreak and $p\in C^0(\overline{Q_T})$ with $p^->1$. Then, for every $\vep\in (0,(p^-)_*-1]$ there exist constants \linebreak$c_{\vep}=c(\vep,d,p,\Omega),\gamma_{\vep}= \gamma(\vep,d,p,\Omega)>0$ (both not depending on $t\in I$),~such~that~for~every~${t\in I}$ and $\bfu\in X^{2,p}_{\bfvarepsilon}(t)$ it holds $\bfu\in L^{p_*(t,\cdot)-\vep}(\Omega)^d$ with 
		\begin{align}
		\rho_{p_*(t,\cdot)-\vep}(\bfu)\leq c_{\vep}\big[1+ \rho_{p(t,\cdot)}(\bfvarepsilon(\bfu))+\|\bfu\|_{Y}^{\gamma_{\vep}}\big]\big(1+\|\bfu\|_{Y}^{\gamma_{\vep}}\big).\label{eq:5.11.0}
		\end{align}
	\end{lem}
	
	\begin{proof}
		We fix an arbitrary $\vep\in  (0,(p^-)_*-1]$ and split the proof into two steps:
		
		\textbf{Step 1:} Let $\overline{p}\in C^0(\overline{Q_T})$ be an extension of $p\in C^0(\overline{Q_T})$ with $p^-\leq \overline{p}\leq p^+$ in~$\setR^{d+1}$ (cf.~Proposition~\ref{2.13.1}). Since $\overline{p}:\overline{Q_T}\to \left(1,\infty\right)$ and $(\cdot)_*:\left[1,\infty\right)\to\left(1,\infty\right)$ are uniformly continuous, the same argumentation as in step \textbf{1} of the proof of Lemma \ref{5.6} shows that there exist finitely many $t_i\in \overline{I}$, $i=1,...,m$, $m\in \setN$, and $x_j\in \overline{\Omega}$, $j=1,...,m$, such that the intervals $I_i:=B_\rho^1(t_i)$, $i=1,...,m$, form an open covering of $\overline{I}$, the balls $B_j=B_{\rho}^d(x_j)$, $j=1,...,m$, form an open~covering~of~$\overline{\Omega}$ and the local exponents $p_{ij}^+:=\sup_{(s,y)^\top\in Q_{ij}}{\overline{p}(s,y)}$ and $p_{ij}^-:=\inf_{(s,y)^\top\in Q_{ij}}{\overline{p}(s,y)}$, where ${Q_{ij}:=I_i\times B_j}$,  satisfy ${(p_{ij}^+)_*-\frac{\vep}{2}< (p_{ij}^-)_*}$ for all $i,j=1,...,m$..
		
		\textbf{Step 2:} We fix an arbitrary $t\in I$.	Then, there exists some index $i=1,...,m$, such~that~${t\in I_i}$. It suffices anew to consider $\bfu\in C_0^\infty(\Omega)^d$.\; Let us fix an arbitrary ball~$B_j$~for~some~${j=1,...,m}$.\\[-2pt] Hence, Lemma~\ref{5.4}, \eqref{eq:5.8.1}--\eqref{eq:5.8.3}~with~${s=(p_{ij}^-)_*-\frac{\vep}{2}}$~and~$W_0^{1,p_{ij}^-}(B_j)^d$ replaced by $W^{1,p_{ij}^-}(B_j)^d$,\\[-1pt] as $\partial B_j\in C^\infty$,  Lemma~\ref{5.6} and the estimate $\rho_{p_{ij}^-}(\bfvarepsilon(\bfu)\chi_{B_j})\leq 2^{p_{ij}^+}(\vert\Omega\vert+\rho_{p(t,\cdot)}(\bfvarepsilon(\bfu)))$, provide constants $c_{ij}^{\vep}=c_{ij}(\vep,p_{ij}^-,B_j),\gamma_{ij}^{\vep}=\gamma_{ij}(\vep,p_{ij}^-,B_j)\!>\!0$ (not depending on $t\in I$),~such~that
		\begin{align}
		\begin{split}
		\rho_{(p_{ij}^-)_*-\frac{\vep}{2}}(\bfu\chi_{B_j})&\leq c_{ij}^{\vep} \big[\rho_{p_{ij}^-}(\bfu\chi_{B_j})+\rho_{p_{ij}^-}(\nb\bfu\chi_{B_j})\big]\|\bfu\|_{L^2(B_j)^d}^{\gamma_{ij}^{\vep}}
		\\&\leq c_{ij}^{\vep} \big[\rho_{p_{ij}^-}(\bfu\chi_{B_j})+\rho_{p_{ij}^-}(\bfvarepsilon(\bfu)\chi_{B_j})\big]\|\bfu\|_{L^2(B_j)^d}^{\gamma_{ij}^{\vep}}\\&\leq c_{ij}^{\vep} \big[1+\rho_{p(t,\cdot)}(\bfu)+\rho_{p(t,\cdot)}(\bfvarepsilon(\bfu))\big]\|\bfu\|_Y^{\gamma_{ij}^{\vep}}\\&
		\leq c_{ij}^{\vep} \big[1+\rho_{p(t,\cdot)}(\bfvarepsilon(\bfu))+\|\bfu\|_Y^{\gamma}\big]\|\bfu\|_Y^{\gamma_{ij}^{\vep}}.
		\end{split}\label{eq:5.11.2}
		\end{align}
		We set $\gamma_{\vep}:=\max_{i,j=1,...,m}{\gamma_{ij}^{\vep}+\gamma}>0$ and $c_\vep:=\max_{i,j=1,...,m}{c_{ij}^{\vep}}>0$.
		Because $(\cdot)_*:\left[1,\infty\right)\to \setR$ is non-decreasing, one can readily see that ${(\overline{p}_*)^+_{ij}:=\sup_{(s,y)^\top\in Q_{ij}}{\overline{p}_*(s,y)}\leq(p^+_{ij})_*}$.
		Therefore, also using that $\rho_{\overline{p}_*(t,\cdot)-\vep}(\bfu\chi_{B_j})\leq 2^{(p^+)_*}(\vert\Omega \vert+\rho_{(p^-_{ij})_*-\frac{\vep}{2}}(\bfu\chi_{B_j}))$, as $\overline{p}_*(t,x)-\vep\leq (p^+_{ij})_*-\vep<(p^-_{ij})_*-\frac{\vep}{2}$ for every $x\in B_j$ (cf.~step~\textbf{1}), and $a^{\gamma}+a^{\gamma_{ij}^{\vep}}\leq 2^{\gamma_{\vep}+1}(1+a^{\gamma_{\vep}})$ for all $a\ge 0$,  we obtain 
		\begin{align}
		\begin{split}
		\rho_{p_*(t,\cdot)-\vep}(\bfu\chi_{B_j\cap \Omega})=\rho_{\overline{p}_*(t,\cdot)-\vep}(\bfu\chi_{B_j})\leq c_{\vep}\big[1+ \rho_{p(t,\cdot)}(\bfvarepsilon(\bfu))+\|\bfu\|_Y^{\gamma_{\vep}}\big]\big(1+\|\bfu\|_Y^{\gamma_{\vep}}\big).
		\end{split}
		\label{eq:5.11.3}
		\end{align}
		Eventually, summing \eqref{eq:5.11.3} with respect to $j=1,...,m$, we conclude \eqref{eq:5.11.0}.\hfill$\qed$
	\end{proof}

	\begin{proof} (of Proposition~\ref{5.10}) 
		We proceed as in the proof of Proposition~\ref{5.5}, but now we apply Lemma~\ref{5.6}.\hfill$\qed$
	\end{proof}
	\newpage 
	
	Let us next prove an  unsteady analogue of the following steady compactness result.
	
	\begin{prop}\label{5.12}
		Let $\Omega\subseteq \setR^d$, $d\ge 2$, be a bounded domain and $p\in C^0(\overline{\Omega})$ with ${1\!<\!p^-\!\leq \!p^+\!<\!d}$. Moreover, denote by $p^*:=\frac{dp}{d-p}\in C^0(\overline{\Omega})$ the \textbf{Sobolev conjugate exponent}.~Then,~there~holds\\[-3pt] for every ${\vep\in \left(0,d'\right)}$ the compact embedding $W_0^{1,p(\cdot)}(\Omega)^d\embedding\embedding L^{p^*(\cdot)-\vep}(\Omega)^d$.
		In addition, if $p\in \mathcal{P}^{\log}(\Omega)$ with ${1<\!p^-\!\leq \!p^+\!<\!d}$, then ${W_0^{1,p(\cdot)}(\Omega)^d\embedding\embedding L^{p^*(\cdot)}(\Omega)^d}$.
	\end{prop}
	
	\begin{proof}
		See \cite[Theorem 3.3 \& Theorem 3.8]{KR91} or \cite[Theorem 8.3.1 \&  Corollary 8.4.4]{DHHR11}.\hfill$\qed$
	\end{proof}
	
	Motivated by a comparison of Proposition~\ref{5.10} and Proposition~\ref{5.12}, we next prove a variable exponent Bochner--Lebesgue version of  Proposition~\ref{5.12}. This result will be based on the following compactness principle due to Landes and Mustonen~\cite{LM87}.
	
	\begin{prop}[Landes' and Mustonen's compactness principle]\label{5.13}
		Let $\Omega\subseteq \setR^d$, $d\ge 2$, be a bounded domain, $I:=\left(0,T\right)$, $T<\infty$, $Q_T:=I\times \Omega$ and $p\in \left(1,\infty\right)$. Then, for a bounded sequence $(\bsu_n)_{n\in \setN}\subseteq L^p(I,W^{1,p}_0(\Omega)^d)\cap L^\infty(I,L^1(\Omega)^d)$ from 
		\begin{align}
		\begin{split}
		\begin{alignedat}{2}
		\bsu_n&\overset{n\to\infty}{\weakto }\bsu&&\quad\textup{ in } L^p(I,W^{1,p}_0(\Omega)^d),\\
		\bsu_n(t)&\overset{n\to\infty}{\weakto }\bsu(t)&&\quad\textup{ in }L^1(\Omega)^d\quad\textup{for a.e. }t\in I,
		\end{alignedat}
		\end{split}\label{eq:5.13.a}
		\end{align}
		it follows that $\bsu_n\to \bsu$ in $L^p(I,L^p(\Omega)^d)$ $(n\to \infty)$.
	\end{prop}
	
	\begin{proof}
		See \cite[Proposition 1]{LM87}.\hfill$\qed$
	\end{proof}	

	Combining the parabolic interpolation inequality (cf.~Proposition~\ref{5.8}) and Landes' and Mustonen's compactness principle, we gain the following interpolated compactness principle.
	
	\begin{cor}\label{5.13.1}
		Let $\Omega\subseteq \mathbb{R}^d$, $d\ge 2$, be a bounded domain, $I:=\left(0,T\right)$, $T<\infty$, $Q_T:=I\times \Omega$ and $p\in \left(1,\infty\right)$. Then, for a bounded sequence $(\boldsymbol{u}_n)_{n\in \setN}\subseteq L^p(I,W^{1,p}_0(\Omega)^d)\cap L^\infty(I,L^2(\Omega)^d)$ from \eqref{eq:5.13.a}, 
		it follows that $\boldsymbol{u}_n\to \boldsymbol{u}$ in $L^s(Q_T)^d$ $(n\to\infty)$ for all $s\in \left[1,p_*\right)$. 
	\end{cor}
	
	\begin{prop}[Compactness principle for $\bscal{X}^{2,p}_{\bfvarepsilon}(Q_T)\cap\bscal{Y}^\infty(Q_T)$]\label{5.14}
		Let $\Omega\subseteq \setR^d$, $d\ge 2$, be a bounded  domain, $I:=\left(0,T\right)$, $T<\infty$, $Q_T:=I\times \Omega$ and $p\in C^0(\overline{Q_T})$~with~${p^->1}$. Moreover, let  $(\bsu_n)_{n\in \setN}\subseteq \bscal{X}^{2,p}_{\bfvarepsilon}(Q_T)\cap\bscal{Y}^\infty(Q_T)$ be a sequence, such that
		\begin{alignat*}{2}
		\bsu_n&\overset{n\to\infty}{\weakto}\bsu\quad&&\quad\textup{ in }\bscal{X}^{2,p}_{\bfvarepsilon}(Q_T),\\
		\bsu_n&\;\;\overset{\ast}{\rightharpoondown}\;\;\bsu&&\quad\textup{ in }\bscal{Y}^\infty(Q_T)\quad(n\to\infty),\\
		\bsu_n(t)&\overset{n\to\infty}{\weakto}\bsu(t)&&\quad\textup{ in }Y\quad\textup{for a.e. }t\in I.
		\end{alignat*}
		Then, there holds $\bsu_n\to\bsu$ in $L^{\max\{2,p_*(\cdot,\cdot)\}-\vep}(Q_T)^d$ $(n\to\infty)$ for every $\vep\in (0,(p^-)_*-1]$.
	\end{prop}
	
	\begin{proof}
	Due to the embeddings $ \bscal{X}^{2,p}_{\bfvarepsilon}(Q_T)\embedding L^{p^-}(I,W^{1,p^-}_0(\Omega)^d)$, $\bscal{Y}^\infty(Q_T)\embedding L^\infty(I,L^1(\Omega)^d)$ and $Y\embedding L^1(\Omega)^d$, Proposition~\ref{5.13} yields a subsequence $(\bsu_n)_{n\in \Lambda}$, with $\Lambda\subseteq\setN$, such that $\bsu_n\to\bsu$ in $\setR^d$ $(\Lambda \ni n\to \infty)$ almost everywhere in $Q_T$.
	Owning to Proposition~\ref{5.10}, the sequence $(\bsu_n)_{n\in \Lambda}\subseteq L^{\max\{2,p_*(\cdot,\cdot)\}-\vep}(Q_T)^d$ is $L^{\max\{2,p_*(\cdot,\cdot)\}-\vep}(Q_T)$--uniformly integrable for every $\vep\in (0,(p^-)_*-1]$. Therefore, Vitali's convergence theorem for variable exponent Lebesgue spaces (cf.~\cite[Thm.~3.8]{Gut09})    and the standard convergence principle 
	(cf.~\cite[Kap.~I, Lem.~5.4]{GGZ74}) yield the assertion.\hfill$\qed$
	\end{proof}
	
	\begin{rmk}\label{general}
		Lemma \ref{5.6} and Lemma \ref{5.11} also hold, if for every $t\in I$ only $\mathbf{u}\in X^{p,p}_{\boldsymbol{\varepsilon}}(t)\cap Y$. Note that Lemma~\ref{5.6} implies the inclusion  $X^{2,p}_{\boldsymbol{\varepsilon}}(t)\subseteq X^{p,p}_{\boldsymbol{\varepsilon}}(t)\cap Y$ for every $t\in I$.~However, if only ${p \in  C^0(\overline{Q_T})}$, we cannot say that $X^{p,p}_{\boldsymbol{\varepsilon}}(t)\cap Y\subseteq X^{2,p(\cdot)}_{\boldsymbol{\varepsilon}}(t)$ for every $t\in I$, because it is not clear, whether
		$C_0^\infty(\Omega)^d$ is dense in $X^{p,p}_{\boldsymbol{\varepsilon}}(t)\cap Y$ for every $t\in I$, without imposing further assumptions on the regularity of the domain or the continuity of the variable exponents. In fact, using \cite[Corollary 5.15]{KR20}, it is readily seen that $C_0^\infty(\Omega)^d$ is dense in $X^{p,p}_{\boldsymbol{\varepsilon}}(t)\cap Y$ for every $t\in I$, if $\Omega\subseteq \mathbb{R}^d$, $d\ge 2$, is a bounded Lipschitz domain and $q,p\in \mathcal{P}^{\log}(Q_T)$ with $q^-,p^->1$. In addition, if ${p\in \mathcal{P}^{\log}(Q_T)}$ with $p^->1$, then also~${\bscal{X}^{\circ,p}(Q_T)\cap \bscal{Y}^\infty(Q_T)\!\hookrightarrow\! L^{\max\{2,p_*(\cdot,\cdot)\}-\vep}(Q_T)^d}$~for~every~${\vep\!\in\! \left(0,(p^-)_*-1\right]}$, which extends Propositon~\ref{5.10}, as  we cannot say  ${\bscal{X}^{\circ,p}(Q_T)\cap \bscal{Y}^\infty(Q_T)=\bscal{X}^{2,p}(Q_T)\cap \bscal{Y}^\infty(Q_T)}$ beforehand. For proofs of all these extensions, we refer to the thesis \cite{K21}.\vspace*{-1cm}
	\end{rmk}
	\newpage

	\section{Bochner pseudo-monotonicity, Bochner condition (M) and Bochner coercivity}
	\label{sec:6}
	
	In this section we extend the notions  Bochner pseudo-monotonicity, Bochner condition (M) and Bochner coercivity to the framework of variable exponent Bochner--Lebesgue spaces. 
	Throughout the entire section, if nothing else is stated, let $\Omega\subseteq \setR^d$, $d\ge 2$, be a bounded domain, $I:=\left(0,T\right)$, $T<\infty$, $Q_T:=I\times \Omega$ and $q,p\in \mathcal{P}^{\infty}(Q_T)$ with $q^-,p^->1$.
	
	\begin{defn}\label{main.1}
		An operator
		$\bscal{A}:\bscal{X}^{q,p}_{\bfvarepsilon}(Q_T)\cap
		\bscal{Y}^\infty(Q_T)\to\bscal{X}^{q,p}_{\bfvarepsilon}(Q_T)^*$
		is said to 
		\begin{description}[{(ii)}]
			\item[(i)] be \textbf{Bochner pseudo-monotone}, if for a sequence
			$(\bsu_n)_{n\in\setN}\subseteq
			\bscal{X}^{q,p}_{\bfvarepsilon}(Q_T)\cap
			\bscal{Y}^\infty(Q_T)$ from
			\begin{alignat}{2}
			\bsu_n\overset{n\to\infty}{\weakto}
			&\bsu\quad &&\text{ in }\bscal{X}^{q,p}_{\bfvarepsilon}(Q_T)\label{eq:6.1.a},
			\\\bsu_n\;\;\overset{\ast}{\rightharpoondown}\;\;
			&\bsu&&\text{ in }
			\bscal{Y}^\infty(Q_T) \quad (n\to\infty),
			\label{eq:6.1.b}
			\\
			\bsu_n(t)\overset{n\to\infty}{\weakto}
			&\bsu(t)\quad &&\text{ in }Y\quad\text{for a.e. }t\in I,\label{eq:6.1.c}
			\end{alignat}
			and
			\begin{align}
			\limsup_{n\to\infty}{\langle \bscal{A}\bsu_n,\bsu_n-\bsu\rangle_{\bscal{X}^{q,p}_{\bfvarepsilon}(Q_T)}}\leq 0,\label{eq:6.1.d}
			\end{align}
			it follows that $
			{\langle \bscal{A}\bsu,\bsu-\bsv\rangle_{\bscal{X}^{q,p}_{\bfvarepsilon}(Q_T)}\!\leq\!	\liminf_{n\to\infty}{\langle \bscal{A}\bsu_n,\bsu_n-\bsv\rangle_{\bscal{X}^{q,p}_{\bfvarepsilon}(Q_T)}}}$~for~all~${\bsv\in\bscal{X}^{q,p}_{\bfvarepsilon}(Q_T)}$.
			\item[(ii)] satisfy the \textbf{Bochner condition (M)}, if for a sequence
			$(\bsu_n)_{n\in\setN}\subseteq
			\bscal{X}^{q,p}_{\bfvarepsilon}(Q_T)\cap
			\bscal{Y}^\infty(Q_T)$ from \eqref{eq:6.1.a}--\eqref{eq:6.1.c} and
			\begin{align}
			\bscal{A}\bsu_n\overset{n\to\infty}{\weakto}\bsu^*\quad\text{ in }\bscal{X}^{q,p}_{\bfvarepsilon}(Q_T)^*,\label{eq:6.1.e}\\
			\limsup_{n\to\infty}{\langle \bscal{A}\bsu_n,\bsu_n\rangle_{\bscal{X}^{q,p}_{\bfvarepsilon}(Q_T)}}\leq \langle \bsu^*,\bsu\rangle_{\bscal{X}^{q,p}_{\bfvarepsilon}(Q_T)},\label{eq:6.1.f}
			\end{align}
			it follows that $\bscal{A}\bsu=\bsu^*$ in $\bscal{X}^{q,p}_{\bfvarepsilon}(Q_T)^*$.
		\end{description}
	\end{defn}

	We will see in the proof of Theorem~\ref{main} that \eqref{eq:6.1.a}--\eqref{eq:6.1.d} are natural properties of a sequence $(\bsu_n)_{n\in \setN}\subseteq \bscal{X}^{q,p}_{\bfvarepsilon}(Q_T)\cap
	\bscal{Y}^\infty(Q_T)$ coming from an appropriate Galerkin approximation. In fact, \eqref{eq:6.1.a} usually is a consequence of the coercivity of $\bscal{A}$, \eqref{eq:6.1.b} stems from the time derivative, while \eqref{eq:6.1.c} and \eqref{eq:6.1.d} follow directly from the Galerkin approximation. The following proposition clarifies the relations between these new notions, also in comparison to the standard notion.
	
	\begin{prop}\label{6.2}
		The following statements hold true:
		\begin{description}[{(ii)}]
			\item[(i)] If $\bscal{A}:\bscal{X}^{q,p}_{\bfvarepsilon}(Q_T)\to\bscal{X}^{q,p}_{\bfvarepsilon}(Q_T)^*$ is pseudo-monotone (or satisfies the condition (M)), then it is Bochner pseudo-monotone (or satisfies the Bochner condition (M)).
			\item[(ii)] If $\bscal{A}:\bscal{X}^{q,p}_{\bfvarepsilon}(Q_T)\cap
			\bscal{Y}^0(Q_T)\to\bscal{X}^{q,p}_{\bfvarepsilon}(Q_T)^*$ is Bochner pseudo-monotone, then it satisfies the Bochner condition (M).
		\end{description}
	\end{prop}

	\begin{proof}
		\textbf{(i)} follows right from the definitions. For \textbf{(ii)} we proceed as in \cite[Proposition 3.3 (i)]{K19}.\hfill$\qed$
	\end{proof}
	
	Next, we introduce a relaxed notion of coercivity, which, in contrast to the usual notion of coercivity (cf.~\cite[Kap. I, Def. 1.1 (iv)]{Ru04}), takes the additional information provided by the generalized time derivative $\frac{\bfd}{\bfd\bft}$ into account.

	\begin{defn}\label{main.3}
		An operator
		$\bscal{A}:\bscal{X}^{q,p}_{\bfvarepsilon}(Q_T)\cap
		\bscal{Y}^\infty(Q_T)\to\bscal{X}^{q,p}_{\bfvarepsilon}(Q_T)^*$
		is said to be
		\begin{description}[{(ii)}]
			\item[(i)]  \textbf{Bochner coercive with respect to
				$\bsu^*\in\bscal{X}^{q,p}_{\bfvarepsilon}(Q_T)^*$ and $\bfu_0\in Y$}, if there
			exists a constant \mbox{$M:=M(\bsu^*,\bfu_0,\bscal{A})>0$}, such
			that for every
			$\bsu\in\bscal{X}^{q,p}_{\bfvarepsilon}(Q_T)\cap \bscal{Y}^\infty(Q_T)$
			from
			\begin{align}
			\frac{1}{2}\|\bsu(t)\|_Y^2
			+\langle\bscal{A}\bsu
				-\bsu^*,\bsu\chi_{\left[0,t\right]}\rangle_{\bscal{X}^{q,p}_{\bfvarepsilon}(Q_T)}
			\leq \frac{1}{2}\|\bfu_0\|_Y^2\quad\text{ for a.e. }t\in I,\label{eq:5.11}
			\end{align}
			it follows that $\|\bsu\|_{\bscal{X}^{q,p}_{\bfvarepsilon}(Q_T)\cap
				\bscal{Y}^\infty(Q_T)}\leq M$.
			\item[(ii)] \textbf{Bochner coercive}, if it is Bochner coercive with
			respect to  all
			$\bsu^*\in\bscal{X}^{q,p}_{\bfvarepsilon}(Q_T)^*$ and $\bfu_0\in Y$.
		\end{description}
	\end{defn}
	\newpage

	Note that Bochner coercivity, similar to semi-coercivity
	(cf.~\cite{Rou05}) in conjunction with Gr\"onwall's inequality, takes
	into account the information of both the operator and the time
	derivative. In fact, Bochner coercivity is a more general property. Bochner coercivity is phrased in the spirit of a
	local coercivity\footnote{$A:D(A)\subseteq X\to X^*$ is said to be
		locally coercive (cf.~\cite[§32.4.]{Zei90B}) with respect to $f\in X^*$, if
		$D(A)$ is unbounded and there exists a constant $M>0$, such that for every
		$x\in X$ from $\langle Ax,x\rangle_X\leq \langle f,x\rangle_X$ it
		follows that $\|x\|_X\leq M$.} type
	condition of
	${\frac{\textbf{d}}{\textbf{dt}}+\bscal{A}:\bscal{W}^{q,p}_{\bfvarepsilon}(Q_T)\subseteq \bscal{X}^{q,p}_{\bfvarepsilon}(Q_T)\to
	\bscal{X}^{q,p}_{\bfvarepsilon}(Q_T)^*}$. To be more precise, suppose that the assumptions of Proposition \ref{4.6} are satisfied and 
that the operator 
	$\bscal{A}:\bscal{X}^{q,p}_{\bfvarepsilon}(Q_T)\cap\bscal{Y}^\infty(Q_T)\to
	\bscal{X}^{q,p}_{\bfvarepsilon}(Q_T)^*$ is Bochner coercive with~respect~to~${\bsu^*\in\bscal{X}^{q,p}_{\bfvarepsilon}(Q_T)^*}$ and $\bfu_0\in Y$, then
	for $\bsu\in \bscal{W}^{q,p}_{\bfvarepsilon}(Q_T)$ from
	${\|\bsu_c(0)\|_Y\leq
		\|\bfu_0\|_Y}$, i.e., 
	${\langle
	\frac{\textbf{d}\bsu}{\textbf{dt}},\bsu\rangle_{\bscal{X}^{q,p}_{\bfvarepsilon}(Q_T)}\ge
	-\frac{1}{2}\|\bfu_0\|_Y^2}$, and
	\begin{align}
	\bigg\langle \frac{\textbf{d}\bsu}{\textbf{dt}}+\bscal{A}\bsu,\bsu\chi_{\left[0,t\right]}\bigg\rangle_{\bscal{X}^{q,p}_{\bfvarepsilon}(Q_T)}\leq \langle \bsu^*,\bsu\chi_{\left[0,t\right]}\rangle_{\bscal{X}^{q,p}_{\bfvarepsilon}(Q_T)}\quad\text{ for a.e. }t\in I,\label{eq:5.13}
	\end{align}
	it follows that $\|\bsu\|_{\bscal{X}^{q,p}_{\bfvarepsilon}(Q_T)\cap\bscal{Y}^\infty(Q_T)}\leq M $, because \eqref{eq:5.13} is just \eqref{eq:5.11} (cf.~Proposition \ref{4.6}~(ii)). In other words, if the image of $\bsu\in\bscal{W}^{q,p}_{\bfvarepsilon}(Q_T)$ with respect to $\frac{\textbf{d}}{\textbf{dt}}$ and $\bscal{A}$ is bounded by the data $\bfu_0$, $\bsu^*$ in this weak sense, then $\bsu$ is contained in a fixed ball in $\bscal{X}^{q,p}_{\bfvarepsilon}(Q_T)\cap\bscal{Y}^\infty(Q_T)$. We chose \eqref{eq:5.11} instead of \eqref{eq:5.13} in Definition~\ref{main.3}, since $\bsu\in \bscal{X}^{q,p}_{\bfvarepsilon}(Q_T)\cap\bscal{Y}^\infty(Q_T)$~is~not~admissible~in~\eqref{eq:5.13}.
	
	\section{Hirano--Landes approach}
	\label{sec:7}
	This section is concerned with the so-called \textbf{Hirano--Landes approach}, which - in the broadest sense - can be traced back to contributions of Landes \cite{Lan86} and Hirano \cite{Hir1,Hir2}, see also \cite{Pap97,Rou05,BR17,BR20,KR19,BKR20} for more recent developments, and states, roughly speaking, that if an operator ${\bscal{A}:\bscal{X}^{q,p}_{\bfvarepsilon}(Q_T)\cap\bscal{Y}^\infty(Q_T)\!\to\!\bscal{X}^{q,p}_{\bfvarepsilon}(Q_T)^*}$ is \textbf{induced} by a  family ${A(t):X^{q,p}_{\bfvarepsilon}(t)\to X^{q,p}_{\bfvarepsilon}(t)^*}$, $t\in I$, i.e., for $\bsu\in \bscal{X}^{q,p}_{\bfvarepsilon}(Q_T)\cap\bscal{Y}^\infty(Q_T)$ and $\bsv\in \bscal{X}^{q,p}_{\bfvarepsilon}(Q_T)$ there holds
	\begin{align}
	\langle \bscal{A}\bsu,\bsv\rangle_{\bscal{X}^{q,p}_{\bfvarepsilon}(Q_T)}:=\int_I{\langle A(t)(\bsu(t)),\bsv(t)\rangle_{X^{q,p}_{\bfvarepsilon}(t)}\,dt},\label{eq:induced}
	\end{align}
	or equivalently $(\bscal{A}\bsu)(t):=A(t)(\bsu(t))$ in $X^{q,p}_{\bfvarepsilon}(t)^*$ for almost every $t\in I$ by virtue of Remark~\ref{3.9},
	then the question, whether $\bscal{A}$ is Bochner pseudo-monotone, satisfies the Bochner condition (M) or is Bochner coercive, can be traced back to the properties of the inducing~family~of~operators $A(t):X^{q,p}_{\bfvarepsilon}(t)\to X^{q,p}_{\bfvarepsilon}(t)^*$, $t\in I$, in each single time slice $t\in I$, i.e., to those properties which are usually already well-known from the corresponding steady problem. We will give sufficient conditions on such operator families, such that the induced operator satisfies these new concepts.
	
	Throughout the entire section, let $\Omega\subseteq \setR^d$, $d\ge 2$, be a bounded Lipschitz domain, $I:=\left(0,T\right)$, $T<\infty$, $Q_T:=I\times \Omega$ and $q,p\in \mathcal{P}^{\log}(Q_T)$ with $q^-,p^->1$. 
	
	Let us first examine which assumptions on a family of operators  ${A(t):X^{q,p}_{\bfvarepsilon}(t)\to X^{q,p}_{\bfvarepsilon}(t)^*}$,~${t\!\in\! I}$, are sufficient, in order to guarantee the well-definedness and boundedness of the corresponding induced operator $ \bscal{A}:\bscal{X}^{q,p}_{\bfvarepsilon}(Q_T)\cap\bscal{Y}^\infty(Q_T)\to\bscal{X}^{q,p}_{\bfvarepsilon}(Q_T)^*$, given trough \eqref{eq:induced}.
	
	\begin{prop}\label{7.1}
		Let $A(t):X^{q,p}_{\bfvarepsilon}(t)\to X^{q,p}_{\bfvarepsilon}(t)^*$, $t\in I$, be a family with the following
		properties:
		\begin{description}[{\textbf{(C.3)}}]
			\item[\textbf{(C.1)}] \hypertarget{C.1}
			$A(t):X^{q,p}_{\bfvarepsilon}(t)\to X^{q,p}_{\bfvarepsilon}(t)^*$ is demi-continuous for
			almost every $t\in I$.
			\item[\textbf{(C.2)}] \hypertarget{C.2}
			$(t\mapsto \langle A(t)\bfu,\bfv\rangle_{X^{q,p}_{\bfvarepsilon}(t)}):I\to \setR$ is Lebesgue measurable for every $\bfu,\bfv\in X^{q,p}_+$.
			\item[\textbf{(C.3)}] \hypertarget{C.3} There exist a function $\alpha\in L^1(I,\setR_{\ge 0})$ and
			a non-decreasing function
			$\mathscr{B}:\setR_{\ge 0}\to \setR_{\ge 0}$, such that for almost every $t\in I$ and every $\bfu,\bfv\in X^{q,p}_{\bfvarepsilon}(t)$ it holds
			\begin{align*}
			\vert\langle A(t)\bfu,\bfv\rangle_{X^{q,p}_{\bfvarepsilon}(t)}\vert\leq\mathscr{B}(\|\bfu\|_Y)\big(\alpha(t)+\rho_{q(t,\cdot)}(\bfu)+\rho_{p(t,\cdot)}(\bfvarepsilon(\bfu))+\rho_{q(t,\cdot)}(\bfv)+\rho_{p(t,\cdot)}(\bfvarepsilon(\bfv))\big).
			\end{align*}
		\end{description}
		Then, $\bscal{A}:\bscal{X}^{q,p}_{\bfvarepsilon}(Q_T)\cap\bscal{Y}^\infty(Q_T)\to\bscal{X}^{q,p}_{\bfvarepsilon}(Q_T)^*$, given through \eqref{eq:induced}, is well-defined and bounded.
	\end{prop}
	
	\newpage
	\begin{proof}
		
		\textbf{1. Well-definedness:}
		Let $\bsu_1\in
		\bscal{X}^{q,p}_{\bfvarepsilon}(Q_T)\cap\bscal{Y}^\infty(Q_T)$ and $\bsu_2\in \bscal{X}^{q,p}_{\bfvarepsilon}(Q_T)$.~Since~the~set~of simple functions $\bscal{S}(I,X^{q,p}_+)$ (cf.~\cite[Kap. II, Def. 1.1]{Ru04}) is dense in $\bscal{X}^{q,p}_+(Q_T)$ (cf.~\cite[Kap.~II, Lem.~1.25~(i)]{Ru04}) and $C^\infty_0(Q_T)^d$ is dense in $\bscal{X}^{q,p}_{\bfvarepsilon}(Q_T)$ (cf.~Prop.~\ref{3.4.1}), $\bscal{S}(I,X^{q,p}_+)$~is~dense~in~$\bscal{X}^{q,p}_{\bfvarepsilon}(Q_T)$.\\[-1pt] Thus, there exist 
		$(\boldsymbol{s}^m_n)_{n\in\setN}\subseteq \bscal{S}(I,X^{q,p}_+)$, $m=1,2$, i.e.,  $\boldsymbol{s}^m_n(\cdot)=\sum_{i=1}^{k^m_n}{\textbf{s}^m_{n,i}\chi_{E^m_{n,i}}}(\cdot)$,~${m=1,2}$,~where\\[-1pt] $\textbf{s}^m_{n,i}\in X^{q,p}_+$, $k^m_n\in\setN$ and    $E^m_{n,i}\subseteq I$ is  measurable with $\bigcup_{i=1}^{k_n^m}{E_{n,i}^m}=I$ and~${E^m_{n,i}\cap E^m_{n,j}=\emptyset}$~for~${i\neq j}$, such that $\boldsymbol{s}^m_n\to\bsu_m$ in $\bscal{X}^{q,p}_{\bfvarepsilon}(Q_T)$ ${(n\to \infty)}$ for $m=1,2$. Hence, Corollary~\ref{3.5} yields subsequences $(\boldsymbol{s}^m_n)_{n\in\Lambda}$, $m=1,2$, with a cofinal\footnote{Let $(A,\leq )$ be a partially ordered set. Then,  $B\subseteq A$ is cofinal, if for every $a\in A$ there exists~a~${b\in B}$~with~${a\leq b}$.\vspace*{-1cm}} subset $\Lambda\subseteq \setN$, such that for $m=1,2$ it holds ${\boldsymbol{s}^m_n(t)\to \bsu_m(t)}$ in $X^{q,p}_{\bfvarepsilon}(t)$ ${(\Lambda\ni n\to \infty)}$ for almost every $t\in I$. 	Due to (\hyperlink{C.1}{C.1}), we infer from this that $\langle A(t)\boldsymbol{s}_n^1(t),\boldsymbol{s}_n^2(t)\rangle_{X^{q,p}_{\bfvarepsilon}(t)}\to \langle A(t)\bsu_1(t),\bsu_2(t)\rangle_{X^{q,p}_{\bfvarepsilon}(t)}$ $(\Lambda\ni n\to \infty)$ for almost every $t\in I$. Thus, since ${(t\mapsto\langle A(t)\boldsymbol{s}_n^1(t),\boldsymbol{s}_n^2(t)\rangle_{X^{q,p}_{\bfvarepsilon}(t)}):I\to \setR}$, $n\in \setN$, are Lebesgue~measurable,~because
		\begin{align*}
			\langle A(t)\boldsymbol{s}_n^1(t),\boldsymbol{s}_n^2(t)\rangle_{X^{q,p}_{\bfvarepsilon}(t)}=\sum_{i=1}^{k^1_n}{\sum_{j=1}^{k^2_n}{\langle A(t)\textbf{s}^1_{n,i},\textbf{s}^2_{n,j}\rangle_{X^{q,p}_{\bfvarepsilon}(t)}\chi_{E^1_{n,i}\cap E^2_{n,j}}(t)}}\quad\text{ for a.e. }t\in I,
		\end{align*}
		and the functions
		$(t\to\langle A(t)\textbf{s}^1_{n,i},\textbf{s}^2_{n,j}\rangle_{X^{q,p}_{\bfvarepsilon}(t)}):I\to \setR$, $i=1,...,k_n^1$, $j=1,...,k_n^2$, $n\in\setN$, 
		are Lebesgue measurable (cf.~(\hyperlink{C.2}{C.2})), we conclude that $(t\mapsto\langle A(t)\bsu_1(t),\bsu_2(t)\rangle_{X^{q,p}_{\bfvarepsilon}(t)}):I\to \setR$ is Lebesgue measurable. Hence,~we~can~inspect~the~function~${(t\mapsto\langle A(t)\bsu_1(t),\bsu_2(t)\rangle_{X^{q,p}_{\bfvarepsilon}(t)}):I\to \setR}$ for integrability. In doing so, using (\hyperlink{C.3}{C.3}), we obtain
		\begin{align}
		\int_I{\vert\langle A(t)\bsu_1(t),\bsu_2(t)\rangle_{X^{q,p}_{\bfvarepsilon}(t)}\vert\,dt}&\leq \mathscr{B}(\|\bsu_1\|_{\bscal{Y}^\infty(Q_T)})\big(\rho_{q(\cdot,\cdot)}(\bsu_1)+\rho_{p(\cdot,\cdot)}(\bfvarepsilon(\bsu_1))\big)\label{eq:7.1.1}
		\\[-2pt]&\quad+\mathscr{B}(\|\bsu_1\|_{\bscal{Y}^\infty(Q_T)})\big(\|\alpha\|_{L^1(I)}+\rho_{q(\cdot,\cdot)}(\bsu_2)+\rho_{p(\cdot,\cdot)}(\bfvarepsilon(\bsu_2))\big),\notag
		\end{align}
		i.e., $\bscal{A}:\bscal{X}^{q,p}_{\bfvarepsilon}(Q_T)\cap
		\bscal{Y}^\infty(Q_T)\to \bscal{X}^{q,p}_{\bfvarepsilon}(Q_T)^*$ is well-defined.
		
		\textbf{2. Boundedness:} Since $\|\bsu_2\|_{\bscal{X}^{q,p}_{\bfvarepsilon}(Q_T)}\leq 1$ implies $\rho_{q(\cdot,\cdot)}(\bsu_2)+\rho_{p(\cdot,\cdot)}(\bfvarepsilon(\bsu_2))\leq 2$ for every $\bsu_2\in \bscal{X}^{q,p}_{\bfvarepsilon}(Q_T)$, we infer from \eqref{eq:7.1.1} that for every $\bsu_1\in \bscal{X}^{q,p}_{\bfvarepsilon}(Q_T)\cap
		\bscal{Y}^\infty(Q_T)$ it holds
		\begin{align*}
		\|\bscal{A}\bsu_1\|_{\bscal{X}^{q,p}_{\bfvarepsilon}(Q_T)^*}&=\sup_{\substack{\bsu_2\in \bscal{X}^{q,p}_{\bfvarepsilon}(Q_T)\\\|\bsu_2\|_{\bscal{X}^{q,p}_{\bfvarepsilon}(Q_T)}\leq 1}}{\langle \bscal{A}\bsu_1,\bsu_2\rangle_{\bscal{X}^{q,p}_{\bfvarepsilon}(Q_T)}}\\&\leq 	\mathscr{B}(\|\bsu_1\|_{\bscal{Y}^\infty(Q_T)})\big(\|\alpha\|_{L^{1}(I)}+\rho_{q(\cdot,\cdot)}(\bsu_1)+\rho_{p(\cdot,\cdot)}(\bfvarepsilon(\bsu_1))+2\big),
		\end{align*}
		i.e., $\bscal{A}:\bscal{X}^{q,p}_{\bfvarepsilon}(Q_T)\cap
		\bscal{Y}^\infty(Q_T)\to \bscal{X}^{q,p}_{\bfvarepsilon}(Q_T)^*$ is bounded.
		\hfill$\qed$
	\end{proof}

	\begin{prop}\label{7.2}
	Let $q,p\in \mathcal{P}^{\log}(Q_T)$ with $p^->1$ and  $2\leq q\leq\max\{2,p_*-\vep\}$  in $Q_T$ for some $\vep\in (0,(p^-)_*-1]$. Moreover, let ${A(t):X^{q,p}_{\bfvarepsilon}(t)\to X^{q,p}_{\bfvarepsilon}(t)^*}$, $t\in I$,  be a family of operators fulfilling (\hyperlink{C.1}{C.1})--(\hyperlink{C.3}{C.3}) and additionally satisfying:
		\begin{description}[{\textbf{(C.6)}}]
			\item[\textbf{(C.4)}] \hypertarget{C.4}{} $A(t):X^{q,p}_{\bfvarepsilon}(t)\to X^{q,p}_{\bfvarepsilon}(t)^*$ is pseudo-monotone for almost every $t\in I$.
			\item[\textbf{(C.5)}] \hypertarget{C.5}{} There exist a constant $c_0>0$ and
			functions $c_1,c_2\in L^1(I,\setR_{\ge 0})$, such that  for almost every $t\in I$ and every $\bfu\in X^{q,p}_{\bfvarepsilon}(t)$ it holds
			\begin{align*}
			\langle A(t)\bfu,\bfu\rangle_{X^{q,p}_{\bfvarepsilon}(t)}\ge c_0\rho_{p(t,\cdot)}(\bfvarepsilon(\bfu))-c_1(t)\|\bfu\|_Y^2-c_2(t).
			\end{align*}
			\item[\textbf{(C.6)}] \hypertarget{C.6}{} There exist a function $\alpha\in L^1(I,\setR_{\ge 0})$, a non-decreasing function $\mathscr{B}:\setR_{\ge 0}\to \setR_{\ge 0}$ and  a function $c:\left(0,\vep_0\right)\to \setR_{\ge 0}$, where $\vep_0>0$,  such that  for almost every $t\in I$, every $\vep\in \left(0,\vep_0\right)$ and ${\bfu,\bfv\in X^{q,p}_{\bfvarepsilon}(t)}$ it holds
			\begin{align*}
			\vert\langle A(t)\bfu,\bfv\rangle_{X^{q,p}_{\bfvarepsilon}(t)}\vert\leq \mathscr{B}(\|\bfu\|_Y)\big(\alpha(t)+\vep\rho_{p(t,\cdot)}(\bfvarepsilon(\bfu))+c(\vep)\rho_{p(t,\cdot)}(\bfvarepsilon(\bfv))\big).
			\end{align*}
		\end{description}
		Then, the induced operator $\bscal{A}:\bscal{X}^{q,p}_{\bfvarepsilon}(Q_T)\cap
		\bscal{Y}^\infty(Q_T)\to\bscal{X}^{q,p}_{\bfvarepsilon}(Q_T)^*$ satisfies:
		\begin{description}[{(ii)}]
			\item[(i)] $\bscal{A}:\bscal{X}^{q,p}_{\bfvarepsilon}(Q_T)\cap
			\bscal{Y}^\infty(Q_T)\to\bscal{X}^{q,p}_{\bfvarepsilon}(Q_T)^*$ is Bochner pseudo-monotone.
			\item[(ii)]  $\bscal{A}:\bscal{X}^{q,p}_{\bfvarepsilon}(Q_T)\cap
			\bscal{Y}^\infty(Q_T)\to\bscal{X}^{q,p}_{\bfvarepsilon}(Q_T)^*$ is Bochner coercive with respect to all $\bfu_0\in Y$ and $\bscal{J}_{\bfvarepsilon}(\bsf,\bsF)\in \bscal{X}^{q,p}_{\bfvarepsilon}(Q_T)^*$, where $\bsf\in L^{\min\{2,(p^-)'\}}(Q_T)^d$ and $\bsF\in L^{p'(\cdot,\cdot)}(Q_T,\mathbb{M}_{\sym}^{d\times d})$.
		\end{description}
	\end{prop}
	
	\newpage
	The core of the Hirano--Landes approach, which is also an essential ingredient in the proof of Proposition~\ref{7.2} (i), forms the following abstract lemma. In fact, it is formulated in a more abstract manner as actually necessary for our approaches. The reason for this is that the author conjectures that the following lemma and an adapted version of Proposition~\ref{7.2}, might also have impact in FSI  problems in this general form.
	
	\begin{lem}[Hirano--Landes]\label{7.3}
		Let $(X(t),Y(t))$, $t\in I$, be a family of compatible couples\footnote{Two Banach spaces $X$, $Y$ form a compatible couple, if the exists a Hausdorff vector space $Z$, such that ${X,Y\embedding Z}$. Then, the $X\cap Y$ in $Z$ is well-defined and with the norm $\|\cdot\|_{X\cap Y}:=\|\cdot\|_X+\|\cdot\|_Y$~a~Banach~space.}, ${A(t):X(t)\cap Y(t)\to (X(t)\cap Y(t))^*}$, $t\in I$, a family of operators, ${(\bsu_n(t))_{n\in\setN}\subseteq
			X(t)\cap Y(t)}$,~${t\!\in\! I}$, families of sequences, $\bsu(t)\in X(t)\cap Y(t)$, $t\in I$, a family of elements and $\mu\in L^1(I)$. Moreover, we assume that for almost every $t\in I$
		the following time slice conditions are satisfied:
		\begin{description}[{(iii)}]
			\item[(i)] $\bsu_n(t)\weakto
			\bsu(t)\text{ in }Y(t)$ $(n\to \infty)$.
			\item[(ii)] $A(t):X(t)\cap Y(t)\to (X(t)\cap Y(t))^*$ is pseudo-monotone and $X(t)$ reflexive.
			\item[(iii)] $\langle A(t)(\bsu_n(t)),\bsu_n(t)-\bsu(t)\rangle_{X(t)\cap Y(t)}\ge \mathscr{K}(t)(\bsu_n(t))-\mu(t)$ for all $n\in \setN$, where the mapping $\mathscr{K}(t):X(t)\to\setR_{\ge 0}$ is weakly coercive, i.e., from $\|x\|_{X(t)}\to\infty$ it follows that $\mathscr{K}(t)(x)\to \infty$.
		\end{description}
		Suppose that for $(h_n(t))_{n\in\setN}:=(\langle A(t)(\bsu_n(t)),\bsu_n(t)-\bsu(t)\rangle_{X(t)\cap Y(t)})_{n\in\setN}$, $t\in I$, the following integrability conditions hold:
		\begin{description}
			\item[(iv)] $(h_n)_{n\in\setN}\subseteq L^1(I)$ with $\limsup_{n\to\infty}{\int_I{h_n(s)\,ds}}\leq 0$.
		\end{description}
		Then, there exists a cofinal subset $\Lambda\subseteq \setN$ (not depending on $t\in I$), such that for almost~every~${t\!\in\! I}$ and  every $v\in X(t)\cap Y(t)$ there holds
		\begin{align*}
		\langle A(t)(\bsu(t)),\bsu(t)-v\rangle_{X(t)\cap Y(t)}\leq	\liminf_{\Lambda\ni n\to\infty}{\langle A(t)(\bsu_n(t)),\bsu_n(t)-v\rangle_{X(t)\cap Y(t)}}.
		\end{align*}
	\end{lem}
	
	\begin{proof} We define $\mathcal{S}:=
		\{t\in I \mid \textbf{(i)}-\textbf{(iv)}\text{ and }\vert\mu(t)\vert<\infty\text{ holds for }t\}$. Apparently, $I\setminus\mathcal{S}$~is~a~null~set. Our first objective is to verify for every $t\in\mathcal{S}$ that
		\begin{align}
			\liminf_{n\to\infty}
			{h_n(t)}\ge 0.\tag*{$(\ast )_{t}$}
		\end{align}
		To this end, let us fix an arbitrary
		$t\in\mathcal{S}$ and define the set $\Lambda_t:=\{n\in\setN\mid
		h_n(t)< 0\}$.
		We assume without loss of generality that $\Lambda_t$
		is not finite. Otherwise, $(\ast )_{t}$ would already hold true
		for this specific $t\in\mathcal{S}$ and nothing would be left to
		do. But if $\Lambda_t$ is not finite, then 
		\begin{align}
			\limsup_{\Lambda_t\ni n\to\infty}
			{h_n(t)}\leq 0.\label{eq:7.3.1}
		\end{align}
		Using the definition of $\Lambda_t$ in \textbf{(iii)}, we obtain $\sup_{n\in \Lambda_t}{\mathscr{K}(t)(\bsu_n(t))}\leq\vert\mu(t)\vert
		<\infty$,
		i.e., $(\bsu_n(t))_{n\in \Lambda_t}$ is bounded in $X(t)$, because $\mathscr{K}:X(t)\to\setR_{\ge 0}$ is weakly coercive (cf.~\textbf{(iii)}). Therefore, as $X(t)$ is reflexive (cf.~\textbf{(ii)}) and due to \textbf{(i)}, the sequence $(\bsu_n(t))_{n\in \Lambda_t}\subseteq X(t)\cap Y(t)$ is sequentially weakly compact in $X(t)\cap Y(t)$. \textbf{(i)} and the injectivity of the embedding ${X(t)\cap Y(t)\embedding Y(t)}$ further yields that
		$\bsu(t)\in X(t)\cap Y(t)$ is the only weak accumulation point  of ${(\bsu_n(t))_{n\in \Lambda_t}}$ in $ X(t)\cap Y(t)$. As a result, from the standard convergence principle (cf.~\cite[Kap. I, Lem. 5.4]{GGZ74}),~we~conclude~that
		\begin{align}
			\bsu_n(t)\overset{n\to\infty}{\weakto}
			\bsu(t)\quad\text{ in }X(t)\cap Y(t)\quad(n\in \Lambda_t).\label{eq:7.3.3}
		\end{align}
		Due to \eqref{eq:7.3.1} and \eqref{eq:7.3.3}, the pseudo-monotonicity of $A(t):X(t)\cap Y(t)\to (X(t)\cap Y(t))^*$ (cf.~\textbf{(ii)})
		yields $\liminf_{\Lambda_t\ni n\to\infty}
		{h_n(t)}\ge 0$. Since
		$h_n(t)\ge 0$ for every
		$n\in\Lambda_t^c$, we conclude~$(\ast)_t$~for~every~${t\in\mathcal{S}}$. Thanks to
		$h_n(t)\ge -\mu(t)$ for every $t\in\mathcal{S}$ and 
		$n\in\setN$ (cf.~\textbf{(iii)}), a generalized version of Fatou's lemma (cf.~\cite[Theorem 1.18]{Rou05}) is
		applicable. It yields, also using $(\ast )_{t}$ and \textbf{(iv)}, that
		\begin{align}
			\begin{split}
				0\leq
				\int_I{\liminf_{n\to\infty}
					{h_n(s)\,ds}}\leq
				\liminf_{n\to\infty}{\int_I{h_n(s)\,ds}}\leq
				\limsup_{n\rightarrow\infty}{\int_I{h_n(s)\,ds}}\leq  0,
			\end{split}\label{eq:7.3.4}
		\end{align}
		\pagebreak
		i.e., $\lim_{n\to\infty}{\int_I{h_n(s)\,ds}}=0$. Since $(s\mapsto s^-:=\min\{0,s\}):\setR\to \setR_{\leq 0}$ is continuous and non-decreasing, we
		deduce from $(\ast )_{t}$ that ${0\ge\limsup_{n\to\infty}
		{h_n(t)^-}\ge \min\{0,
		\liminf_{n\to\infty}{h_n(t)}\}=0}$,
		i.e., $h_n(t)^-\to 0$ $(n\to\infty)$, for every $t\in \mathcal{S}$. Because of
		$0\ge h_n(t)^-\ge -\mu(t)$ for every $t\in\mathcal{S}$, Vitali's theorem yields
		$h_n^-\to 0$ in $L^1(I)$ $(n\to\infty)$.  From
		the latter, $\vert h_n\vert=h_n-2h_n^-$ in $L^1(I)$ for every $n\in \setN$ and $\lim_{n\to\infty}{\int_I{h_n(s)\,ds}}=0$, we
		conclude that $h_n\to 0$ in
		$L^1(I)$ $(n\to\infty)$. Thus, there exists  a cofinal subset $\Lambda\subseteq\setN$ (not depending on $t\in I$), such
		that for almost every $t\in I$
		\begin{align}
			\lim_{\Lambda\ni n\to\infty}{h_n(t)}= 0.\label{eq:7.3.5}
		\end{align}
		Using \eqref{eq:7.3.5} in \textbf{(iii)}, we obtain ${\limsup_{\Lambda \ni n\to\infty}
		{\mathscr{K}(t)(\bsu_n(t))}\leq\vert\mu(t)\vert<\infty}$~for~almost~every~${t \in I}$, 
		which leads as before for almost every $t\in I$ to 
		\begin{align}
			\bsu_n(t)\overset{n\to\infty}
			{\weakto}\bsu(t)\quad\text{ in }
			X(t)\cap Y(t)\quad(n\in \Lambda).\label{eq:7.3.6}
		\end{align} 
		From \eqref{eq:7.3.5}, \eqref{eq:7.3.6} and \textbf{(ii)}, we finally conclude for almost every $t\in I$ and every  $v\in X(t)\cap  Y(t)$
		\begin{align*}
			\langle A(t)(\bsu(t)),\bsu(t)-
			v\rangle_{X(t)\cap Y(t)}\leq
			\liminf_{\Lambda\ni n\to\infty}
			{\langle A(t)(\bsu_n(t)),\bsu_n(t)
				-v\rangle_{X(t)\cap Y(t)}}.\tag*{$\qed$}
		\end{align*}
	\end{proof}

	\begin{proof} (of Proposition~\ref{7.2}) 	\textbf{ad (i)} Let
		$(\bsu_n)_{n\in\setN}\subseteq
		\bscal{X}^{q,p}_{\bfvarepsilon }(Q_T)\cap
		\bscal{Y}^\infty(Q_T)$ be a sequence satisfying \eqref{eq:6.1.a}--\eqref{eq:6.1.d}. From the boundedness of $\bscal{A}:\bscal{X}_{\bfvarepsilon }^{q,p}(Q_T)\cap
		\bscal{Y}^\infty(Q_T)\to\bscal{X}^{q,p}_{\bfvarepsilon }(Q_T)^*$ (cf.~Proposition~\ref{7.1}) and reflexivity of $\bscal{X}^{q,p}_{\bfvarepsilon }(Q_T)^*$ (cf.~Proposition~\ref{3.4}) we obtain a subsequence
		$(\bsu_n)_{n\in\Lambda}$, with
		${\Lambda\subseteq\setN}$, and a weak limit $\bsu^*\in\bscal{X}^{q,p}_{\bfvarepsilon }(Q_T)^*$, such that
		${\bscal{A}\bsu_n\weakto\bsu^*\text{ in }\bscal{X}^{q,p}_{\bfvarepsilon }(Q_T)^*}$~$(\Lambda\ni n\to \infty)$~and $\langle \bscal{A}\bsu_n,\bsu_n\rangle_{\bscal{X}^{q,p}_{\bfvarepsilon }(Q_T)}\to
		\liminf_{n\to\infty}{\langle\bscal{A}
			\bsu_n,\bsu_n\rangle_{\bscal{X}^{q,p}_{\bfvarepsilon }(Q_T)}}$ $(\Lambda\ni n\to\infty)$, which implies that for every ${\bsv\in \bscal{X}^{q,p}_{\bfvarepsilon }(Q_T)}$ there holds 
		\begin{align}
			\lim_{\Lambda\ni n\to\infty}{
				\langle \bscal{A}\bsu_n,\bsu_n-\bsv\rangle_{\bscal{X}_{\bfvarepsilon }^{q,p}(Q_T)}}\leq
			\liminf_{n\to\infty}{\langle\bscal{A}
				\bsu_n,\bsu_n-\bsv
				\rangle_{\bscal{X}_{\bfvarepsilon }^{q,p}(Q_T)}}.\label{eq:7.2.1}
		\end{align}
		We intend to apply Lemma~\ref{7.3}, with compatible couples $(X(t),Y(t))\!=\!(X^{q,p}_{\bfvarepsilon}(t),Y)$, $t\in I$, and operator family $A(t):X^{q,p}_{\bfvarepsilon}(t)\to X^{q,p}_{\bfvarepsilon}(t)^*$, $t\in I$, whereby we exploit that ${X^{q,p}_{\bfvarepsilon}(t)=X^{q,p}_{\bfvarepsilon}(t)\cap Y}$  with norm equivalence for every ${t\in I}$ due to $q\ge 2$ in $Q_T$. It is not difficult to see that the conditions \textbf{(i)}, \textbf{(ii)} and \textbf{(iv)} of Lemma~\ref{7.3} are already satisfied.~So,~let~us~check~\textbf{(iii)}. For each $\bsv\in \bscal{X}^{q,p}_{\bfvarepsilon }(Q_T)\cap
		\bscal{Y}^\infty(Q_T)$, we define  ${K_{\bsv}:=\sup_{n\in \setN}{\|\bsu_n\|_{\bscal{Y}^\infty(Q_T)}}+\|\bsv\|_{\bscal{Y}^\infty(Q_T)}<\infty}$~(cf.~\eqref{eq:6.1.b}), 
		${\vep_{\bsv}:=\min\{\frac{c_0}{2\mathscr{B}(K_{\bsv})},\frac{\vep_0}{2}\}}$ and $\mu_{\bsv}(t):=(1+c_1(t))K_{\bsv}^2+c_2(t)+\mathscr{B}(K_{\bsv})(\alpha(t)+c(\vep_{\bsv})\rho_{p(t,\cdot)}(\bfvarepsilon(\bsv(t))))\in L^1(I)$, use the conditions (\hyperlink{C.5}{C.5}) and (\hyperlink{C.6}{C.6}), to obtain for almost every $t\in I$ and every $n\in\Lambda$
		\begin{align}
			\langle
			A(t)(\bsu_n(t)),\bsu_n(t)-
			\bsv(t)\rangle_{X^{q,p}_{\bfvarepsilon}(t)}\ge
			\frac{c_0}{2}\rho_{p(t,\cdot)}(\bfvarepsilon(\bsu_n(t)))+\|\bsu_n(t)\|_Y^2
			-\mu_{\bsv}(t).\label{eq:7.2.2}
		\end{align}
		Due to the weak coercivity of the mappings ${\mathscr{K}(t):=\frac{c_0}{2}\rho_{p(t,\cdot)}(\bfvarepsilon(\cdot))+\|\cdot\|_Y^2:X^{q,p}_{\bfvarepsilon}(t)\to \setR_{\ge 0}}$,~${t\!\in\! I}$, which 
		is basically a consequence of the norm equivalence ${\|\cdot\|_{X^{2,p}_{\boldsymbol{\varepsilon}}(t)}\sim \|\cdot\|_{X^{q,p}_{\boldsymbol{\varepsilon}}(t)}}$~for~every~${t\in I}$, that results from  Lemma~\ref{5.11} and $2\leq q\leq \max\{2,p_*-\vep\}$ in $Q_T$, \eqref{eq:7.2.2}~yields~just~condition~\textbf{(iii)}. Hence, Lemma~\ref{7.3} is applicable and provides a cofinal subset $\Lambda_0\subseteq \Lambda$ (not depending on $t\in I$), such that for every  $\bsv\in \bscal{X}^{q,p}_{\bfvarepsilon }(Q_T)\cap
		\bscal{Y}^\infty(Q_T)$ and almost every ${t\in I}$ there holds
		\begin{align}
			\langle A(t)(\bsu(t)),\bsu(t)-\bsv(t)\rangle_{X^{q,p}_{\bfvarepsilon}(t)}\leq\liminf_{\Lambda_0\ni n\to\infty}{\langle A(t)(\bsu_n(t)),\bsu_n(t)-\bsv(t)\rangle_{X^{q,p}_{\bfvarepsilon}(t)}}.\label{eq:7.2.3}
		\end{align}
		We integrate \eqref{eq:7.2.3} with respect to $t\in I$, apply a generalized version of Fatou's lemma (cf.~\cite[Theorem 1.18]{Rou05}), which is allowed since $\langle A(t)(\bsu_n(t)),\bsu_n(t)-\bsv(t)\rangle_{X^{q,p}_{\bfvarepsilon}(t)}\ge -\mu_{\bsv}(t)$ for almost every $t\in I$ and all $n\in \Lambda$ (cf.~\eqref{eq:7.2.2}), and exploit \eqref{eq:7.2.1},~to~arrive~for~every~${\bsv\in \bscal{X}^{q,p}_{\bfvarepsilon }(Q_T)\cap \bscal{Y}^\infty(Q_T)}$~at
		\begin{align}
				\langle \bscal{A}\bsu,\bsu
				-\bsv\rangle_{\bscal{X}^{q,p}_{\bfvarepsilon }(Q_T)}
				&\leq
				\int_I{\liminf_{\Lambda_0\ni n\to\infty}
					{\langle A(s)(\bsu_n(s)),\bsu_n(s)
						-\bsv(s)\rangle_{X^{q,p}_{\bfvarepsilon}(s)}}\,ds}\notag
				\\
				&\leq \liminf_{\Lambda_0\ni n\to\infty}
				{\int_I{\langle A(s)(\bsu_n(s)),\bsu_n(s)
						-\bsv(s)\rangle_{X^{q,p}_{\bfvarepsilon}(s)}\,ds}}\label{eq:7.2.5.1}
				\\
				&=\lim_{\Lambda\ni n\to\infty}
				{\langle \bscal{A}\bsu_n,\bsu_n
					-\bsv\rangle_{\bscal{X}^{q,p}_{\bfvarepsilon }(Q_T)}}
				\leq \liminf_{n\to\infty}{\langle \bscal{A}
					\bsu_n,\bsu_n-\bsv\rangle_{\bscal{X}^{q,p}_{\bfvarepsilon }(Q_T)}}.\notag
		\end{align}
		As $C^\infty_{0}(Q_T)^d\subseteq \bscal{X}^{q,p}_{\bfvarepsilon }(Q_T)\cap
		\bscal{Y}^\infty(Q_T)$ is dense in $\bscal{X}^{q,p}_{\bfvarepsilon }(Q_T)$ (cf.~Proposition~\ref{3.4.1}), \eqref{eq:7.2.5.1} holds for every ${\bsv\in \bscal{X}^{q,p}_{\bfvarepsilon }(Q_T)}$, i.e., $\bscal{A}:\bscal{X}^{q,p}_{\bfvarepsilon }(Q_T)\cap
		\bscal{Y}^\infty(Q_T)\to \bscal{X}^{q,p}_{\bfvarepsilon }(Q_T)^*$ is Bochner pseudo-monotone.\vspace*{-0.75cm}
		
		\newpage
		\textbf{ad (ii)} 
		Let $\bsf\in L^{\nu'}(Q_T)^d$, where $\nu:=\max\{2,p^-\}$, $\bsF\in L^{p'(\cdot,\cdot)}(Q_T,\mathbb{M}_{\sym}^{d\times d})$ and $\bfu_0\in Y$ be arbitrary and assume that
		$\bsu\in\bscal{X}^{q,p}_{\bfvarepsilon }(Q_T)\cap
		\bscal{Y}^\infty(Q_T)$
		satisfies for almost every $t\in I$
		\begin{align}
			\frac{1}{2}\|\bsu(t)\|_Y^2+\int_0^t{\langle
				A(s)(\bsu(s))-\mathcal{J}_{\bfvarepsilon}(\bsf(s),\bsF(s)),\bsu(s)\rangle_{X^{q,p}_{\bfvarepsilon}(s)}\,ds}\leq \frac{1}{2}\|\bfu_0\|_Y^2.\label{eq:7.2.4} 
		\end{align}\\[-5pt]
		Using (\hyperlink{C.5}{C.5}) and \eqref{eq:3.9.1} in \eqref{eq:7.2.4}, we get for almost every $t\in I$
		\begin{align}
			\begin{split}
				\frac{1}{2}\|\bsu(t)\|_Y^2&+c_0\int_{0}^{t}{\rho_{p(s,\cdot)}(\bfvarepsilon(\bsu)(s))\,ds}\leq
				\frac{1}{2}\|\bfu_0\|_Y^2+ \|c_2\|_{L^1(I)}+\int_0^t{\vert
					c_1(s)\vert\|\bsu(s)\|_Y^2\,ds}\\&\quad+\int_0^t{\int_{\Omega}{\vert\bsf(s,y)\vert\vert\bsu(s,y)\vert+\vert\bsF(s,y)\vert\vert\bfvarepsilon(\bsu)(s,y)\vert\,dy}\,ds}.
			\end{split}\label{eq:7.2.5} 
		\end{align}
		Applying the weighted $\vep$--Young inequality with respect to the exponent $\nu$, with the constant $c_\nu(\vep):=(\nu\vep)^{1-\nu'}(\nu')^{-1}$ for every $\vep>0$, on the first summand in the last intergral in \eqref{eq:7.2.5} and pointwise for almost every  $(t,x)^\top\in Q_T$
		with respect to the exponent $p(t,x)$, with the constant $c_{p(t,x)}(\vep):=(p(t,x)\vep)^{1-p'(t,x)}(p'(t,x))^{-1}$ for every $\vep \in (0,(p^-)^{-1})$, on the second summand, also\\[-1pt] using that $c_{p(t,x)}(\vep)\leq c_p(\vep):=(p^-\vep)^{1-(p^-)'}((p^+)')^{-1}$ for every $\vep \in (0,(p^-)^{-1})$ and $(t,x)^\top\in Q_T$,  we deduce for every $t\in \overline{I}$
		\begin{align}
			\begin{split}
				\int_0^t\int_{\Omega}\;\vert\bsf(s,y)\vert\vert\bsu(s,y)\vert&+\vert\bsF(s,y)\vert\vert\bfvarepsilon(\bsu)(s,y)\vert\,dy\,ds\leq c_\nu(\vep)\rho_{\nu'}(\bsf)+ c_p(\vep)\rho_{p'(\cdot,\cdot)}(\bsF)\\&\quad+\vep	\int_0^t{\rho_\nu(\bsu(s))+\rho_{p(s,\cdot)}(\bfvarepsilon(\bsu)(s))\,ds}.
			\end{split}
			\label{eq:7.2.6}
		\end{align}
		Then, using that $\rho_\nu(\bsu(t))\leq \|\bsu(t)\|_Y^2+\rho_{p^-}(\bsu(t))$  for almost every $t\in I$ and  additionally that $\rho_{p^-}(\bsu(t))\leq c_{p^-}\rho_{p^-}(\bfvarepsilon(\bsu)(t))\leq c_{p^-}2^{p^+}(\vert \Omega\vert +\rho_{p(t,\cdot)}(\bfvarepsilon(\bsu)(t)))$  for almost every $t\in I$, due to Poincar\'e's and Korn's inequality with respect to  $p^-\in\left(1,\infty\right)$,  we obtain for every ${t\in \overline{I}}$
		\begin{align}
			\begin{split}
			\int_0^t\int_{\Omega}\;\vert\bsf(s,y)\vert&\vert\bsu(s,y)\vert+\vert\bsF(s,y)\vert\vert\bfvarepsilon(\bsu)(s,y)\vert\,dy\,ds\leq c_\nu(\vep)\rho_{\nu'}(\bsf)+c_p(\vep)\rho_{p'(\cdot,\cdot)}(\bsF)\\&\quad+\vep c_{p^-}2^{p^+}\vert Q_T\vert+\vep	\int_0^t{\|\bsu(s)\|_{Y}^2+\big(1+c_{p^-}2^{p^+}\big)\rho_{p(s,\cdot)}(\bfvarepsilon(\bsu)(s))\,ds}.
		\end{split}\label{eq:7.2.7}
		\end{align}
		Therefore, if we set $M_0:=\frac{1}{2}\|\bfu_0\|_Y^2+\|c_2\|_{L^1(I)}+c_\nu(\vep)\rho_{\nu'}(\bsf)+c_p(\vep)\rho_{p'(\cdot,\cdot)}(\bsF)+\vep c_{p^-}2^{p^+}\vert Q_T\vert\!>\!0$, \linebreak $\bsy:=\frac{1}{2}\|\bsu(\cdot)\|_Y^2\in L^\infty(I)$, $\bsa:=2\vep+2\vert c_1(\cdot)\vert\in L^1(I)$ and $\vep:=\min\big\{\frac{c_0}{2(1+c_{p^-}2^{p^+})},\frac{1}{2p^-}\big\}\,>\,0$, \,then \\[-3pt]we deduce from \eqref{eq:7.2.5}--\eqref{eq:7.2.7} that for almost every $t\in I$ it holds
		\begin{align}
			\begin{split}
				\bsy(t)&+\frac{c_0}{2}\int_{0}^{t}{\rho_{p(s,\cdot)}(\bfvarepsilon(\bsu)(s))\,ds}\leq
				M_0+\int_0^t{\bsa(s)\bsy(s)\,ds}.
			\end{split}\label{eq:7.2.9} 
		\end{align}
		If we apply Grönwall's inequality (cf.~\cite[Lemma II.4.10]{BF13}) on \eqref{eq:7.2.9}, then~we~further~infer~that  ${\|\bsu\|_{\bscal{Y}^\infty(Q_T)}^2\!=\!2\|\bsy\|_{L^\infty(I)}\!\leq\! 2M_0\exp(\|\bsa\|_{L^1(I)})\!=:\!M_1}$, which, looking back to \eqref{eq:7.2.9}, implies that ${\rho_{p(\cdot,\cdot)}(\bfvarepsilon(\bsu))\!\leq\! \frac{2}{c_0}\big(M_0+\|\bsa\|_{L^1(I)}\frac{M_1}{2}\big)\!=:\!M_2}$. In virtue of the parabolic interpolation inequality for $\bscal{X}^{2,p}_{\bfvarepsilon}(Q_T)\cap\bscal{Y}^\infty(Q_T)$ (cf.~Proposition~\ref{5.10}), also using that  ${2\leq q\leq\max\{2,p_*-\vep\}}$ in $Q_T$ and \cite[Cor. 3.34]{DHHR11}, there exists a constant ${c_\vep>0}$, not depending on $ {\bsu\in\bscal{X}^{q,p}_{\bfvarepsilon }(Q_T)\cap
		\bscal{Y}^\infty(Q_T)}$,~such~that
		\begin{align*}
			\|\bsu\|_{L^{q(\cdot,\cdot)}(Q_T)^d}\leq c_\vep\big(\|\bfvarepsilon(\bsu)\|_{L^{p(\cdot,\cdot)}(Q_T)^{d\times d}}+\|\bsu\|_{\bscal{Y}^\infty(Q_T)}\big)\leq c_\vep\Big((M_2+1)^{\frac{1}{p^-}}+M_1^{\frac{1}{2}}\Big)=:M_3.
		\end{align*}
		Altogether,  we proved the estimate $\|\bsu\|_{\bscal{X}^{q,p}_{\bfvarepsilon }(Q_T)\cap\bscal{Y}^\infty(Q_T)}\leq M_3+(M_2+1)^{\frac{1}{p^-}}+M_1^{\frac{1}{2}}=:M$, i.e.,  $\bscal{A}:\bscal{X}^{q,p}_{\bfvarepsilon }(Q_T)\cap
		\bscal{Y}^\infty(Q_T)\to\bscal{X}^{q,p}_{\bfvarepsilon }(Q_T)^*$~is~Bochner~coercive with respect to all $\bfu_0\in Y$ and $\bscal{J}_{\bfvarepsilon}(\bsf,\bsF)\in \bscal{X}^{q,p}_{\bfvarepsilon }(Q_T)^*\!$, where $\bsf\in L^{\min\{2,(p^-)'\}}(Q_T)^d$ and ${\bsF\in L^{p'(\cdot,\cdot)}(Q_T,\mathbb{M}^{d\times d}_{\textup{sym}})}$.\hfill$\qed$
	\end{proof}
	\newpage
	\section{Existence theorem}
	\label{sec:8}
	
	Now we can formulate and prove the main result.\vspace*{-1mm}
	
	\begin{thm}[Main theorem]\label{main}
		let $\Omega\subseteq \setR^d$, $d\ge 2$, be a bounded Lipschitz~domain,~${I:=\left(0,T\right)}$, $T<\infty$, $Q_T:=I\times\Omega$ and $q,p\in \mathcal{P}^{\log}(Q_T)$ with $q^-,p^->1$, such that $X^{q,p}_-\embedding Y$. Moreover, let 
		 $\bscal{A}:\bscal{X}^{q,p}_{\bfvarepsilon}(Q_T)\cap \bscal{Y}^\infty(Q_T)\to \bscal{X}^{q,p}_{\bfvarepsilon}(Q_T)^*$ be bounded, Bochner coercive with respect to ${\bsu^*\in\bscal{X}^{q,p}_{\bfvarepsilon}(Q_T)^*}$ and $\bfu_0\in Y$, and satisfying  the Bochner condition (M). Then, there exists a solution $\bsu\in \bscal{W}^{q,p}_{\bfvarepsilon}(Q_T)$ with a representation $\bsu_c\in \bscal{Y}^0(Q_T)$ of\vspace*{-1mm} 
		\begin{alignat*}{2}
		\frac{\textbf{d}\bsu}{\textbf{dt}}+\bscal{A}\bsu&=\bsu^*&&\quad\text{ in }\bscal{X}^{q,p}_{\bfvarepsilon}(Q_T)^*,\\
		\bsu_c(0)&=\bfu_0&&\quad\text{ in }Y.
		\end{alignat*}
	\end{thm}
	
		By means of Proposition~\ref{7.1} and Proposition~\ref{7.2}, we conclude  the following less abstract and thus potentially more applicable version of Theorem~\ref{main}.\vspace*{-1mm}
	
	\begin{cor}\label{main2}
		Let $\Omega\subseteq \setR^d$, $d\ge 2$, be a bounded Lipschitz domain, $I:=\left(0,T\right)$, $T<\infty$, ${Q_T:=I\times\Omega}$, $q,p\in \mathcal{P}^{\log}(Q_T)$ with $p^->1$ and $2\leq q\leq \max\{2,p_*-\vep\}$~in~$Q_T$~for~${\vep\in (0,(p^-)_*-1]}$.  Moreover, let $A(t):X^{q,p}_{\bfvarepsilon}(t)\to X^{q,p}_{\bfvarepsilon}(t)^*$, $t\in I$, be family of operators satisfying \textup{(\hyperlink{C.1}{C.1})}--\textup{(\hyperlink{C.6}{C.6})}. Then, for  $\bscal{J}_{\bfvarepsilon}(\bsf,\bsF)\in \bscal{X}^{q,p}_{\bfvarepsilon }(Q_T)^*\!$, where $\bsf\in L^{\min\{2,(p^-)'\}}(Q_T)^d$ and ${\bsF\in L^{p'(\cdot,\cdot)}(Q_T,\mathbb{M}^{d\times d}_{\textup{sym}})}$, and $\bfu_0\!\in\! Y$, there exists $\bsu\!\in\! \bscal{W}^{q,p}_{\bfvarepsilon}(Q_T)$ with a representation $\bsu_c\!\in \!\bscal{Y}^0(Q_T)$, such that ${\bsu_c(0)=\bfu_0}$~in~$Y$ and~for~every~${\bfphi\in C^\infty_0(Q_T)^d}$ there holds\vspace*{-1mm}
		\begin{align*}
		\int_I{\bigg\langle \frac{\mathbf{d}\bsu}{\mathbf{dt}}(t),\bfphi(t)\bigg\rangle_{X^{q,p}_{\bfvarepsilon}(t)}\!\! dt}+\int_I{\langle A(t)(\bsu(t)),\bfphi(t)\rangle_{X^{q,p}_{\bfvarepsilon}(t)}\,dt}=\int_I{\langle\bsu^*(t),\bfphi(t)\rangle_{X^{q,p}_{\bfvarepsilon}(t)}\,dt}.
		\end{align*}
	\end{cor}

	\begin{proof}(of Theorem~\ref{main})
		\textbf{0. Reduction of assumptions:} It suffices to treat the case~${\bsu^*=\boldsymbol{0}}$. Otherwise, we consider  $\bscal{A}-\bsu^*:\bscal{X}^{q,p}_{\bfvarepsilon}(Q_T)\cap \bscal{Y}^\infty(Q_T)\to \bscal{X}^{q,p}_{\bfvarepsilon}(Q_T)^*$, which is again~bounded,~Bochner coercive with respect to  $\boldsymbol{0}\in\bscal{X}^{q,p}_{\bfvarepsilon}(Q_T)^*$ and $\bfu_0\in Y$, and  satisfies the~Bochner~condition~(M).
		
		\textbf{1. Galerkin approximation:}
		On the basis of the separability of $X^{q,p}_+$ there exists a sequence $(\bfv_i)_{i\in \setN}\subseteq X^{q,p}_+$,
		which is dense in $X^{q,p}_+$. Due to the density of $X^{q,p}_+$ in $Y$ and the Gram--Schmidt process, we can 
	 assume that $(\bfv_i)_{i\in \setN}$ is dense and orthonormal in $Y$.~We~set~${X_n:=\text{span}\{\bfv_1,...,\bfv_n\}}$, if equipped 
		with $\|\cdot\|_{X^{q,p}_+}$, and $Y_n:=\text{span}\{\bfv_1,...,\bfv_n\}$, if equipped with $(\cdot,\cdot)_Y$. Then,  the triple $(X_n,Y_n,\mathbf{id}_{X_n})$
		is an evolution triple with canonical embedding $e_n:=(\mathbf{id}_{X_n})^*R_n\mathbf{id}_{X_n}:X_n\to X_n^*$, where $R_n:Y_n\to Y_n^*$ denotes the corresponding Riesz isomorphism with respect to $Y_n$ and $(\mathbf{id}_{X_n})^*$ the adjoint operator of $\mathbf{id}_{X_n}:X_n\to Y$.
		Moreover, we introduce the spaces
		\begin{align*}
		\bscal{X}_n:=L^{\max\{q^+,p^+\}}(I,X_n),\;\;
		\bscal{W}_n:=W_{e_n}^{1,\max\{q^+,p^+\},\min\{(q^+)',(p^+)'\}}(I,X_n,X_n^*),\;\;\bscal{Y}_n:=C^0(\overline{I},Y_n).
		\end{align*}
		Then, we have $\bscal{X}_n\embedding \bscal{X}^{q,p}_{\bfvarepsilon}(Q_T)$ and Proposition~\ref{2.2} provides an embedding  $(\cdot)_c:\bscal{W}_n\to \bscal{Y}_n$, which yields for every $\bsu\in \bscal{W}_n$ the existence of a unique representation $\bsu_c\in \bscal{Y}_n$, and a formula of
		integration by parts for $\bscal{W}_n$ (cf.~Proposition~\ref{2.2}~(ii)).
		
		We are seeking
		approximative solutions
		$\bsu_n\in \bscal{W}_n$, $n\in \setN$, which solve the evolution equations
		\begin{align}
		\begin{split}
		\begin{alignedat}{2}
		\frac{d_{e_n}\bsu_n}{dt}+\bscal{A}_n\bsu_n&=\boldsymbol{0}&&\quad\text{ in }\bscal{X}_n^*,\\ 
		(\bsu_n)_c(0)&=\bfu_0^n&&\quad\text{ in }Y_n.
		\end{alignedat}
		\end{split}\label{eq:main.2}
		\end{align}
		where $\bscal{A}_n:=(\mathbf{id}_{\bscal{X}_n})^*\bscal{A}:\bscal{X}_n\cap \bscal{Y}_n\to \bscal{X}_n^*$ and $\bfu_0^n:=\sum_{i=1}^{n}{(\bfu_0,\bfv_i)_Y\bfv_i}\in Y_n$ for every $n\in\setN$.
		
		\textbf{2. Existence of Galerkin solutions:}
		We intend to apply Theorem~\ref{2.9}. Thus, we have to check for every $n\in \setN$, whether $\bscal{A}_n:\bscal{X}_n\cap \bscal{Y}_n\to \bscal{X}_n^*$ is bounded, Bochner coercive with respect to $\boldsymbol{0}\in \bscal{X}_n^*$ and $\bfu_0^n\in Y_n$ in the sense of Definition~\ref{2.4}, and satisfies the Bochner condition (M) in the sense of Definition~\ref{2.3}. In fact, since  $\bscal{A}:\bscal{X}^{q,p}_{\bfvarepsilon}(Q_T)\cap\bscal{Y}^\infty(Q_T)\to \bscal{X}^{q,p}_{\bfvarepsilon}(Q_T)^*$ is bounded, Bochner coercive with respect to  $\boldsymbol{0}\in\bscal{X}^{q,p}_{\bfvarepsilon}(Q_T)^*$ and $\bfu_0\in Y$ in the sense of Definition~\ref{main.3}, satisfies the Bochner condition (M) in the sense of Definition~\ref{main.1}, and $\sup_{n\in \setN}{\|\bfu_0^n\|_Y}\leq \|\bfu_0\|_Y$, 
		 the same holds  in the sense of Definition~\ref{2.3} and Definition~\ref{2.4} for its restrictions ${\bscal{A}_n:\bscal{X}_n\cap \bscal{Y}_n\to \bscal{X}_n^*}$,~${n\in \setN}$.
		Therefore, Theorem~\ref{2.9} yields the existence of 	$\bsu_n\in \bscal{W}_n$, $n\in \setN$, solving \eqref{eq:main.2}.\vspace*{-1.5cm}
		
		\newpage
		Moreover, if we test $\eqref{eq:main.2}_1$ for every $n\in \mathbb{N}$ with $\bsu_n\chi_{\left[0,t\right]}\in\bscal{X}_n$, for arbitrary $t\in \left(0,T\right]$,~apply the formula of integration by parts for
		$\bscal{W}_n$ (cf.~Proposition~\ref{2.2}~(ii)) and exploit $\eqref{eq:main.2}_2$ for every ${n\in \setN}$ as well as $\sup_{n\in \setN}{\|\bfu_0^n\|_Y}\leq \|\bfu_0\|_Y$, we obtain for every $t\in \overline{I}$ and $n\in \setN$ the estimate
		\begin{align}
		\frac{1}{2}\|(\bsu_n)_c(t)\|_Y^2+\int_{0}^{t}{\langle A(s)(\bsu_n(s)),\bsu_n(s)\rangle_{X^{q,p}_{\bfvarepsilon}(s)}\,ds}\leq \frac{1}{2}\|\bfu_0\|_Y^2.\label{eq:main.3}
		\end{align}
		From \eqref{eq:main.3} and  the Bochner coercivity of $\bscal{A}:\bscal{X}^{q,p}_{\bfvarepsilon}(Q_T)\cap\bscal{Y}^\infty(Q_T)\to \bscal{X}^{q,p}_{\bfvarepsilon}(Q_T)^*$ with respect to  $\boldsymbol{0}\in\bscal{X}^{q,p}_{\bfvarepsilon}(Q_T)^*$ and $\bfu_0\in Y$, as well as its boundedness, we derive the existence of a constant ${M>0}$, which does not depend on $n\in\setN$, such that for every ${n\in \setN}$ there holds
		\begin{align}
		\|\bsu_n\|_{\bscal{X}^{q,p}_{\bfvarepsilon}(Q_T)\cap\bscal{Y}^\infty(Q_T)}+\|\bscal{A}\bsu_n\|_{\bscal{X}^{q,p}_{\bfvarepsilon}(Q_T)^*}\leq
		M.\label{eq:main.4} 
		\end{align}
		
		\textbf{3. Passage to the limit:}
		
		\textbf{3.1 Convergence of the Galerkin solutions:}
		From the a-priori estimate \eqref{eq:main.4}, the reflexivity of both $\bscal{X}^{q,p}_{\bfvarepsilon}(Q_T)$ and $\bscal{X}^{q,p}_{\bfvarepsilon}(Q_T)^*$ (cf.~Proposition~\ref{3.4}), and the existing separable pre-dual $L^1(I,Y)$ of $\bscal{Y}^\infty(Q_T)$, we obtain a not relabelled
		subsequence $(\bsu_n)_{n\in \setN}\subseteq
		\bscal{X}^{q,p}_{\bfvarepsilon}(Q_T)\cap\bscal{Y}^\infty(Q_T)$ as well as elements
		$\bsu\in\bscal{X}^{q,p}_{\bfvarepsilon}(Q_T)\cap\bscal{Y}^\infty(Q_T)$ and $\bfxi\in \bscal{X}^{q,p}_{\bfvarepsilon}(Q_T)^*$, such that
		\begin{align}
		\begin{alignedat}{2}
		\bsu_n&\overset{n\to\infty}{\weakto}\bsu&\quad
		&\text{ in }\bscal{X}^{q,p}_{\bfvarepsilon}(Q_T),\\
		\bsu_n&\;\;\overset{\ast}{\rightharpoondown}\;\;\bsu&&\text{
			in }\bscal{Y}^\infty(Q_T)\quad (n\to\infty),\\
		\bscal{A}\bsu_n&\overset{n\to\infty}{\weakto}\bfxi&&\text{
			in }\bscal{X}^{q,p}_{\bfvarepsilon}(Q_T)^*.
		\end{alignedat}\label{eq:main.5}
		\end{align}
		
		\textbf{3.2 Regularity and trace of the weak limit:} Let $\bfv\in X_k$, where $k\in \setN$ is arbitrary, and
		$\varphi\in C^\infty(\overline{I})$ with $\varphi(T)=0$. Testing
		$\eqref{eq:main.2}_1$ for every $n\in \setN$ with $n\ge k$ by
		$\bfv\varphi\in \bscal{X}_k\subseteq \bscal{X}_n$, a
		subsequent application of the formula of integration by parts formula
		for $\bscal{W}_n$ (cf.~Proposition~\ref{2.2}~(ii)) and exploiting $\eqref{eq:main.2}_2$ for every $n\in \setN$, yields for every $n\in \setN$ with $n\ge k$
		\begin{align}
		\langle \bscal{A}\bsu_n,\bfv\varphi\rangle_{\bscal{X}^{q,p}_{\bfvarepsilon}(Q_T)} =\int_I{(\bsu_n(s),\bfv)_Y\varphi^\prime(s)\,ds}
		+(\bfu_0^n,\bfv)_Y\varphi(0).\label{eq:main.6}
		\end{align}
		By passing  for $n\to \infty$ in \eqref{eq:main.6}, using \eqref{eq:main.5} and
		$\bfu_0^n\to\bfu_0$ in
		$Y$ $(n\to \infty)$ in doing so, we obtain
		\begin{align}
		\langle \bfxi,\bfv\varphi\rangle_{\bscal{X}^{q,p}_{\bfvarepsilon}(Q_T)}=\int_I{(\bsu(s),\bfv)_Y\varphi^\prime(s)\,ds}
		+(\bfu_0,\bfv)_Y\varphi(0)\label{eq:main.7} 
		\end{align}
		for every $\bfv\in \bigcup_{k\in\setN}{X_k}$ and
		$\varphi\in C^\infty(\overline{I})$ with $\varphi(T)=0$. In the case
		$\varphi\in C_0^\infty(I)$ in \eqref{eq:main.6}, Proposition~\ref{4.5}
		proves, due to the density of $\bigcup_{k\in\setN}{X_k}$ in $X^{q,p}_+$, that $\bsu\in \bscal{W}^{q,p}_{\bfvarepsilon}(Q_T)$ with a representation $\bsu_c\in \bscal{Y}^0(Q_T)$ (cf.~Proposition~\ref{4.6}~(i)) and 
		\begin{align}\label{eq:main.8}
		\frac{\textbf{d}\bsu}{\textbf{dt}}=-\bfxi\quad\textrm{ in }\bscal{X}^{q,p}_{\bfvarepsilon}(Q_T)^*.
		\end{align}
		Thus, we are able to apply the formula of integration by parts for $\bscal{W}^{q,p}_{\bfvarepsilon}(Q_T)$ (cf.~Proposition~\ref{4.6}) in
		\eqref{eq:main.7} for $\varphi\in C^\infty(\overline{I})$ with
		$\varphi(T)=0$ and $\varphi(0)=1$, which yields, also using \eqref{eq:main.8}, that
		\begin{align}
		(\bsu_c(0)-\bfu_0,\bfv)_Y=0 \label{eq:main.9}
		\end{align}
		for every $\bfv\in
		\bigcup_{k\in\setN}{X_k}$ . Because $\bigcup_{k\in\setN}{X_k}$ is dense in $X^{q,p}_+$ and $X^{q,p}_+$ is dense in $Y$, we conclude from \eqref{eq:main.9} the desired initial condition, i.e., that
		\begin{align}
			\bsu_c(0)=\bfu_0\quad\text{ in }Y.\label{eq:main.10}
		\end{align}
		\newpage
		
		\textbf{3.3 Pointwise weak convergence in $Y$:}
		Next, we show that $(\bsu_n)_c(t)\weakto \bsu_c(t)$ in $Y$ ${(n\to \infty)}$ for every $t\in \overline{I}$. To this end, let us fix an arbitrary
		$t\in\left(0,T\right]$. From the a-priori estimate
		$\sup_{n\in \setN}{\|(\bsu_n)_c(t)\|_Y}\leq M$ (cf.~\eqref{eq:main.4}) we
		obtain the existence of a subsequence
		$((\bsu_n)_c(t))_{n\in\Lambda_t}\subseteq
		Y$, with a cofinal subset $\Lambda_t\subseteq\setN$, which initially
		depends on this fixed $t$, and an element
		$\bfu_{\Lambda_t}\in Y$, such that
		\begin{align}
		(\bsu_n)_c(t)
		\overset{n\to\infty}{\weakto}
		\bfu_{\Lambda_t}\quad\text{ in }Y\quad(n\in
		\Lambda_t).\label{eq:main.11}  
		\end{align}
		For $\bfv\in X_k$, where $k\in\Lambda_t$ is arbitrary, and
		$\varphi\in C^\infty(\overline{I})$ with $\varphi(0)=0$ and
		$\varphi(t)=1$, we test $\eqref{eq:main.2}_1$ for every $n\in \Lambda_t$ with $n\ge k$
		 by
		$\bfv\varphi\chi_{\left[0,t\right]}\in \bscal{X}_k\subseteq
		\bscal{X}_n$ and use the
		formula of integration by parts for $\bscal{W}_n$ (cf.~Proposition~\ref{2.2}~(ii)), to arrive for every $n\in \Lambda_t$ with 
		$n\ge k$  at
		\begin{align}
		\langle \bscal{A}\bsu_n,\bfv\varphi\chi_{\left[0,t\right]}\rangle_{\bscal{X}^{q,p}_{\bfvarepsilon}(Q_T)}=\int_0^t{(\bsu_n(s),
			\bfv)_Y\varphi^\prime(s)\,ds}-((\bsu_n)_c(t),\bfv)_Y.\label{eq:main.12}
		\end{align}
		By passing  for  $\Lambda_t \ni n\to \infty$ in \eqref{eq:main.12}, using \eqref{eq:main.5} and \eqref{eq:main.11}, we obtain for every $\bfv\in \bigcup_{k\in\Lambda_t}{X_k}$
		\begin{align}
		\langle \bfxi,\bfv\varphi\chi_{\left[0,t\right]}\rangle_{\bscal{X}^{q,p}_{\bfvarepsilon}(Q_T)}=\int_0^t{(\bsu(s),
			\bfv)_Y\varphi^\prime(s)\,ds}-(\bfu_{\Lambda_t},\bfv)_Y. \label{eq:main.13}
		\end{align}
		From \eqref{eq:main.13}, \eqref{eq:main.8} and the formula of integration by parts for $\bscal{W}^{q,p}_{\bfvarepsilon}(Q_T)$ (cf.~Proposition~\ref{4.6}~(ii)), we also obtain
		\begin{align}
		(\bsu_c(t)-
		\bfu_{\Lambda_t},\bfv)_Y=0\label{eq:main.14} 
		\end{align}
		for every $\bfv\in\bigcup_{k\in\Lambda_t}{X_k}$. Thanks to
		$X_k\subseteq X_{k+1}$ for every $k\in\setN$, it holds
		${\bigcup_{k\in\Lambda_t}{X_k}=\bigcup_{k\in\setN}{X_k}}$. Therefore,
		$\bigcup_{k\in\Lambda_t}{X_k}$ is dense in $Y$ and
		\eqref{eq:main.14} yields that
		$\bsu_c(t)=\bfu_{\Lambda_t}$
		in $Y$, i.e., we have
		\begin{align}
		(\bsu_n)_c(t)\overset{n\to\infty}
		{\weakto}\bsu_c(t)\quad\text{
			in }Y\quad(n\in\Lambda_t).\label{eq:main.15} 
		\end{align}
		As this argumentation stays valid for each weakly convergent subsequence of
		$((\bsu_n)_c(t))_{n\in\setN}\subseteq Y$,
		$\bsu_c(t)\in Y$ is a weak accumulation point of
		each subsequence of
		$((\bsu_n)_c(t))_{n\in\setN}\subseteq
		Y$. The standard convergence principle \cite[Kap. I, Lem. 5.4]{GGZ74} yields that \eqref{eq:main.15} even holds if
		$\Lambda_t=\setN$, i.e., we have
		\begin{align}
		(\bsu_n)_c(t)\overset{n\to \infty}{\weakto}\bsu_c(t)\quad\text{ in }Y\quad\text{ for all }t\in \overline{I}.\label{eq:main.16.0}
		\end{align}
		Therefore, due to $\bsu_c(t)=\bsu(t)$ in $Y$ for almost every $t\in I$, we conlcude from \eqref{eq:main.16.0} that 
		\begin{align}
			\bsu_n(t)\overset{n\to \infty}{\weakto}\bsu(t)\quad\text{ in }Y\quad\text{ for a.e. }t\in I.\label{eq:main.16}
		\end{align}
		\textbf{3.4 Identification of $\bscal{A}\bsu$ and $\bfxi$:}
		According  to \eqref{eq:main.3} in the case $t=T$, also using \eqref{eq:main.10}, we have for every $n\in \setN$
		\begin{align}
		\langle
		\bscal{A}\bsu_n,\bsu_n\rangle_{\bscal{X}^{q,p}_{\bfvarepsilon}(Q_T)}\leq
		-\frac{1}{2}\|(\bsu_n)_c(T)\|_Y^2+
		\frac{1}{2}\|\bsu_c(0)\|_Y^2.\label{eq:main.17}
		\end{align}
		The limit superior with respect to $n\to\infty$ on both sides in \eqref{eq:main.17}, \eqref{eq:main.16.0} in the case $t=T$, the weak lower
		semi-continuity of $\|\cdot\|_Y$, the formula of integration by parts for $\bscal{W}^{q,p}_{\bfvarepsilon}(Q_T)$ (cf.~Proposition~\ref{4.6}~(ii)) and \eqref{eq:main.8}
		yield
		\begin{align}\begin{split}
		\limsup_{n\to\infty}{ \langle
			\bscal{A}\bsu_n,\bsu_n\rangle_{\bscal{X}^{q,p}_{\bfvarepsilon}(Q_T)}}
		&\leq
		-\frac{1}{2}\|\bsu_c(T)\|_Y^2+\frac{1}{2}\|\bsu_c(0)\|_Y^2\\&=-\left\langle
		\frac{\textbf{d}\bsu}{\textbf{dt}},\bsu\right\rangle_{\bscal{X}^{q,p}_{\bfvarepsilon}(Q_T)}=\langle
		\bfxi,\bsu\rangle_{\bscal{X}^{q,p}_{\bfvarepsilon}(Q_T)}.
		\end{split}\label{eq:main.18}
		\end{align}
		By virtue of the Bochner condition (M) of $\bscal{A}:\bscal{X}^{q,p}_{\bfvarepsilon}(Q_T)\cap \bscal{Y}^\infty(Q_T)\to \bscal{X}^{q,p}_{\bfvarepsilon}(Q_T)^*$, we finally conclude from \eqref{eq:main.5}, \eqref{eq:main.16} and \eqref{eq:main.18} that   $\bscal{A}\bsu=\bfxi$ in $\bscal{X}^{q,p}_{\bfvarepsilon}(Q_T)^*$. This completes the proof of Theorem~\ref{main}.\hfill$\qed$
	\end{proof}

	\section{Application}
	\label{sec:9}
	At long last, we apply the developed theory on the model problem \eqref{eq:model}. In doing so, we always assume that $\Omega\subseteq \setR^d$, $d\ge 2$, is a bounded Lipschitz domain, $I:=\left(0,T\right)$, $T<\infty$, $Q_T:=I\times \Omega$ and $q,p\in \mathcal{P}^{\log}(Q_T)$ with $p^->1$ and $2\leq q\leq \max\{2,p_*-\vep\}$ in $Q_T$ for some $\vep\in (0,(p^-)_*-1]$.
	
	\begin{prop}\label{9.1}
		Let
		${\bfS:Q_T\times\mathbb{M}^{d\times d}_{\sym}\to \mathbb{M}^{d\times d}_{\sym}}$ be a mapping satisfying (\hyperlink{S.1}{S.1})--(\hyperlink{S.4}{S.4}) with  respect to $p$. Then, the family of operators  ${S(t):X^{q,p}_{\bfvarepsilon}(t)\to X^{q,p}_{\bfvarepsilon}(t)^*}$, $t\in I$, for almost every $t\in I$ and every $\bfu,\bfv\in X^{q,p}_{\bfvarepsilon}(t)$ given via
		\begin{align*}
		\langle S(t)\bfu,\bfv\rangle_{X^{q,p}_{\bfvarepsilon}(t)}:=\langle \mathcal{J}_{\bfvarepsilon}(\mathbf{0},\bfS(t,\cdot,\bfvarepsilon(\bfu))),\bfv\rangle_{X^{q,p}_{\bfvarepsilon}(t)}=(\bfS(t,\cdot,\bfvarepsilon(\bfu)),\bfvarepsilon(\bfv))_{L^{p(t,\cdot)}(\Omega)^{d\times d}},
		\end{align*} 
		 satisfies  (\hyperlink{C.1}{C.1})--(\hyperlink{C.6}{C.6}). In particular, the induced operator $\bscal{S}\!:\!\bscal{X}^{q,p}_{\bfvarepsilon }(Q_T)\!\to\! \bscal{X}^{q,p}_{\bfvarepsilon }(Q_T)^*$,~given~via~\eqref{eq:induced}, is bounded, Bochner pseudo-monotone and Bochner coercive with respect to all $\bfu_0\in Y$ and ${\bscal{J}_{\bfvarepsilon}(\bsf,\bsF)\in \bscal{X}^{q,p}_{\bfvarepsilon }(Q_T)^*}$, where $\bsf\in L^{\min\{2,(p^-)'\}}(Q_T)^d$ and $\bsF\in L^{p'(\cdot,\cdot)}(Q_T,\mathbb{M}^{d\times d}_{\sym})$.
	\end{prop}
	
	\begin{proof}
		\textbf{ad (C.1)} Using Fubini's theorem in conjunction with (\hyperlink{S.1}{S.1}) and (\hyperlink{S.2}{S.2}), it is readily seen  that $\bfS(t,\cdot,\cdot)\!:\!\Omega\times \mathbb{M}_{\sym}^{d\times d}\!\to\! \mathbb{M}_{\sym}^{d\times d}$ is for almost every $t\in I$ a Carath\'eodory mapping.
		We~fix~${t\!\in\! I}$,~such~that $\bfS(t,\cdot,\cdot)\!:\!\Omega\times \mathbb{M}_{\sym}^{d\times d}\!\to\! \mathbb{M}_{\sym}^{d\times d}$  is a Carath\'eodory mapping and (\hyperlink{S.2}{S.2}) holds for $t$ with ${\beta(t,\cdot)\!\in\! L^{p'(t,\cdot)}(\Omega)}$. Let us first prove~that~${S(t):X^{q,p}_{\bfvarepsilon}(t)\to X^{q,p}_{\bfvarepsilon}(t)^*}$~is~well-defined. To this end,~we~consider~${\bfu \!\in\! X^{q,p}_{\bfvarepsilon}(t)}$. Since $\bfS(t,\cdot,\cdot):\Omega\times \mathbb{M}_{\sym}^{d\times d}\to \mathbb{M}_{\sym}^{d\times d}$ is a Carath\'eodory~mapping,~${(x\mapsto \bfS(t,x,\bfvarepsilon(\bfu)(x))):\Omega\to \mathbb{M}^{d\times d}_{\sym}}$ is Lebesgue measurable. Therefore, we can inspect it for integrability. In doing so, using (\hyperlink{S.2}{S.2}) and repeatedly that $(a+b)^s\leq 2^s(a^s+b^s)$ for all $a,b\ge 0$ and $s>0$, we obtain~for~almost~every~${x\in \Omega}$
			\begin{align}
				\begin{split}
			\vert\bfS(t,x,\bfvarepsilon(\bfu)(x))\vert^{p'(t,x)}&\leq \big[\alpha( \delta+\vert\bfvarepsilon(\bfu)(x)\vert)^{p(t,x)-1}+\beta(t,x)\big]^{p'(t,x)}
			\\&\leq 2^{p'(t,x)}\big[\alpha^{p'(t,x)}(\delta+\vert\bfvarepsilon(\bfu)(x)\vert)^{p(t,x)}+\beta(t,x)^{p'(t,x)}\big]
			\\&\leq 2^{(p^-)'}\big[\alpha^{(p^-)'}2^{p^+}\big(\delta^{p(t,x)}+\vert\bfvarepsilon(\bfu)(x)\vert^{p(t,x)}\big)+\beta(t,x)^{p'(t,x)}\big],
			\end{split}\label{eq:9.1.1}
		\end{align}
		i.e., $\bfS(t,\cdot,\bfvarepsilon(\bfu))\in L^{p'(t,\cdot)}(\Omega,\mathbb{M}^{d\times d}_{\sym})$. Thus, we conclude that $S(t)\bfu\in X^{q,p}_{\bfvarepsilon}(t)^*$ (cf.~Corollary~\ref{3.8}), i.e., $S(t):X^{q,p}_{\bfvarepsilon}(t)\to X^{q,p}_{\bfvarepsilon}(t)^*$ is well-defined. 
		Next, let us show that $S(t):X^{q,p}_{\bfvarepsilon}(t)\to X^{q,p}_{\bfvarepsilon}(t)^*$  is continuous. To this end, let  $(\bfu_n)_{n\in \setN}\subseteq X^{q,p}_{\bfvarepsilon}(t)$ be a sequence, such that ${\bfu_n\!\to\! \bfu}$~in~$X^{q,p}_{\bfvarepsilon}(t)$~${(n\!\to\! \infty)}$. Then, it holds
		\begin{align}
			\bfvarepsilon(\bfu_n)\overset{n\to \infty}{\to}	\bfvarepsilon(\bfu)\quad\text{ in }L^{p(t,\cdot)}(\Omega,\mathbb{M}^{d\times d}_{\text{sym}}).\label{eq:9.1.4}
		\end{align}
		In particular, there exist a subsequence $(\bfu_n)_{n\in \Lambda}$, with a cofinal subset $\Lambda\subseteq \setN$, such that
		\begin{align}
			\bfvarepsilon(\bfu_n)\overset{n\to \infty}{\to}	\bfvarepsilon(\bfu)\quad\text{ in }\mathbb{M}^{d\times d}_{\text{sym}}\quad(n\in \Lambda)\quad\text{ a.e. in }\Omega.\label{eq:9.1.5}
		\end{align}
		Since $\bfS(t,\cdot,\cdot):\Omega\times \mathbb{M}_{\sym}^{d\times d}\to \mathbb{M}_{\sym}^{d\times d}$ is a Carath\'eodory mapping, \eqref{eq:9.1.5} also leads to 
		\begin{align}
			\bfS(t,\cdot,\bfvarepsilon(\bfu_n))\overset{n\to \infty}{\to}\bfS(t,\cdot,\bfvarepsilon(\bfu))\quad\text{ in }\mathbb{M}^{d\times d}_{\text{sym}}\quad(n\in \Lambda)\quad\text{ a.e. in }\Omega.\label{eq:9.1.6}
		\end{align}
		On the other hand, we have by \eqref{eq:9.1.1} for almost every $x\in\Omega$ and every $n\in \Lambda$
		\begin{align*}
			\vert\bfS(t,x,\bfvarepsilon(\bfu_n)(x))\vert^{p'(t,x)}\leq 2^{(p^-)'}\big[\alpha^{(p^-)'}2^{p^+}\big( \delta^{p(t,x)}+\vert\bfvarepsilon(\bfu_n)(x)\vert^{p(t,x)}\big)+\beta(t,x)^{p'(t,x)}\big],
		\end{align*}
		i.e., $(\bfS(t,\cdot,\bfvarepsilon(\bfu_n)))_{n\in\Lambda}\!\subseteq\! L^{p'(t,\cdot)}(\Omega,\mathbb{M}^{d\times d}_{\text{sym}})$ is $L^{p'(t,\cdot)}(\Omega)$--uniformly integrable (cf.~\cite{Gut09})~due~to~\eqref{eq:9.1.4}. As a result, also using \eqref{eq:9.1.6}, Vitali's convergence theorem~(cf.~\cite[Theorem 3.8]{Gut09})~yields
		\begin{align}
			\bfS(t,\cdot,\bfvarepsilon(\bfu_n))\overset{n\to \infty}{\to}\bfS(t,\cdot,\bfvarepsilon(\bfu))\quad\text{ in }L^{p'(t,\cdot)}(\Omega,\mathbb{M}^{d\times d}_{\text{sym}})\quad(n\in \Lambda).\label{eq:9.1.8}
		\end{align}
		The standard convergence principle (cf.~\cite[Kap. I, Lem. 5.4]{GGZ74}) yields that \eqref{eq:9.1.8}~even~holds~if~${\Lambda\!=\!\setN}$. Because of the continuity of  $\mathcal{J}_{\bfvarepsilon}:L^{q'(t,\cdot)}(\Omega)^d\times L^{p'(t,\cdot)}(\Omega,\mathbb{M}^{d\times d}_{\text{sym}})\to X^{q,p}_{\bfvarepsilon}(t)^*$ (cf.~Corollary~\ref{3.8}), we conclude from \eqref{eq:9.1.8} with $\Lambda=\setN$ that
		$S(t)\bfu_n\!\to\! S(t)\bfu$ in $X^{q,p}_{\bfvarepsilon}(t)^*$~${(n\!\to\! \infty)}$. Altogether, the operator ${S(t):X^{q,p}_{\bfvarepsilon}(t)\!\to\! X^{q,p}_{\bfvarepsilon}(t)^*}$~is~indeed~continuous. Since $t\! \in\! I$ was chosen arbitrarily,  such that $\bfS(t,\cdot,\cdot)\!:\!\Omega\times \mathbb{M}_{\sym}^{d\times d}\!\to\! \mathbb{M}_{\sym}^{d\times d}$  is a Carath\'eodory mapping and (\hyperlink{S.2}{S.2}) holds for~$t$~with~${\beta(t,\cdot)\!\in\! L^{p'(t,\cdot)}(\Omega)}$, we thus conclude that condition (\hyperlink{C.1}{C.1}) is satisfied.\vspace*{-0.5cm}
			
		\newpage
		\textbf{ad (C.2)} Inequality \eqref{eq:9.1.1}  in conjunction with Hölder's inequality (cf.~\eqref{hoelder}), guarantees for every $\bfu,\bfv\in X_+^{q,p}$ that  $\big((t,x)^\top\mapsto \bfS(t,x,\bfvarepsilon(\bfu)(x)):\bfvarepsilon(\bfv)(x)\big)\in L^1(Q_T)$. Thus, Fubini's theorem yields  for every $\bfu,\bfv\in X_+^{q,p}$ the Lebesgue measurability of the mapping $\big(t\mapsto\langle S(t)\bfu,\bfv\rangle_{X^{q,p}_{\bfvarepsilon}(t)}\big)=\big(t\mapsto (\bfS(t,\cdot,\bfvarepsilon(\bfu)),\bfvarepsilon(\bfv))_{L^{p(t,\cdot)}(\Omega)^{d\times d}}\big):I\to\setR$, i.e., condition (\hyperlink{C.2}{C.2}) is satisfied.
		
		\textbf{ad (C.3)\vspace*{-0.1pt}\&\vspace*{-0.1pt}(C.6)} 
		Appyling the  $\vep$--Young inequality pointwise for almost every ${(t,x)^\top\in Q_T}$ with respect to the exponent $p(t,x)$, with the constant 
		$c_{p(t,x)}(\tilde{\vep}):=(p'(t,x)\tilde{\vep})^{1-p(t,x)}p(t,x)^{-1}$ for every $\tilde{\vep}\in \left(0,\tilde{\vep}_0\right)$, where ${\tilde{\vep}_0:=((p^+)')^{-1}}$, also using that $c_{p(t,x)}(\tilde{\vep})\leq c_p(\tilde{\vep}):=((p^+)'\tilde{\vep})^{1-p^+}(p^-)^{-1}$ for every $\tilde{\vep}\in \left(0,\tilde{\vep}_0\right)$, we infer for almost every $(t,x)^\top\in Q_T$, every $\tilde{\vep}\in \left(0,\tilde{\vep}_0\right)$~and~${\bfu,\bfv\in X^{q,p}_{\bfvarepsilon}(t)}$
		\begin{align}
			\begin{split}
				\vert  \bfS(t,x,\bfvarepsilon(\bfu)(x)):\bfvarepsilon(\bfv)(x)\vert &\leq \tilde{\vep}	\vert  \bfS(t,x,\bfvarepsilon(\bfu)(x))\vert^{p'(t,x)} +
				c_{p(t,x)}(\tilde{\vep})\vert \bfvarepsilon(\bfv)(x)\vert^{p(t,x)}\\&\leq \tilde{\vep}	\vert  \bfS(t,x,\bfvarepsilon(\bfu)(x))\vert^{p'(t,x)}+c_p(\tilde{\vep})\vert \bfvarepsilon(\bfv)(x)\vert^{p(t,x)}.
			\end{split}\label{eq:4.15.2}
		\end{align}
		Eventually, we integrate \eqref{eq:4.15.2}  with respect to $x\in \Omega$, and use the estimate \eqref{eq:9.1.1} in doing so, to arrive for almost every $t\in I$, every $\tilde{\vep}\in \left(0,\tilde{\vep}_0\right)$ and $\bfu,\bfv\in X^{q,p}_{\bfvarepsilon}(t)$ at
		\begin{align*}
			\vert \langle S(t)\bfu,\bfv\rangle_{X^{q,p}_{\boldsymbol{\varepsilon}}(t)}\vert\leq \tilde{\vep} 2^{(p^-)'}\!\big[\alpha^{(p^-)'}2^{p^+}\!\big(\rho_{p(t,\cdot)}(\delta)\!+\!\rho_{p(t,\cdot)}(\bfvarepsilon(\bfu))\big)\!+\!\rho_{p'(t,\cdot)}(\beta(t,\cdot))\big]\!+\!c_p(\tilde{\vep})\rho_{p(t,\cdot)}(\bfvarepsilon(\bfv)),
		\end{align*}
		i.e., condition (\hyperlink{C.6}{C.6}), and simultaneously also condition (\hyperlink{C.3}{C.3}).
		
		\textbf{ad (C.4)} From (\hyperlink{S.4}{S.4})  one readily derives that $S(t):X^{q,p}_{\bfvarepsilon}(t)\to X^{q,p}_{\bfvarepsilon}(t)^*$ is for almost every $t\in I$ monotone. Therefore, because $S(t):X^{q,p}_{\bfvarepsilon}(t)\to X^{q,p}_{\bfvarepsilon}(t)^*$ is for almost every $t\in I$ continuous, it is also pseudo-monotone (cf.~\cite[Kap. III, Lem. 2.6]{Ru04}), i.e.,  condition (\hyperlink{C.4}{C.4}) is satisfied.
		
		\textbf{ad (C.5)} Using (\hyperlink{S.3}{S.3}) and $(\delta+a)^{p(t,x)-2}a^2\ge \frac{1}{2}a^{p(t,x)}-\delta^{p(t,x)}$ for all $a\ge 0$ and ${(t,x)^\top\in Q_T}$\footnote{Here, we used for $a,\delta\ge 0$ that if $p\ge 2$, then it holds $(\delta+a)^{p-2}a^2\ge a^p$, and if $1<p<2$, then $(\delta+a)^{p-2}a^2\ge (\delta+a)^{p-2}\big(\frac{1}{2}(\delta +a)^2-\delta^2\big)\ge  \frac{1}{2}(\delta+a)^{p}-\delta^{p}\ge\frac{1}{2}a^{p}-\delta^{p}$, as $(\delta+a)^2\leq 2(\delta^2+a^2)$ and ${(\delta+a)^{p-2}\leq \delta^{p-2}}$.
		}, we obtain for almost every $t\in I$ and every $\bfu\in X^{q,p}_{\bfvarepsilon}(t)$ that
		\begin{align*}
		\langle S(t)\bfu,\bfu\rangle_{X^{q,p}_{\bfvarepsilon}(t)}\ge \frac{c_0}{2}\rho_{p(t,\cdot)}(\bfvarepsilon(\bfu))-c_0\rho_{p(t,\cdot)}(\delta)-\|c_1(t,\cdot)\|_{L^1(\Omega)},
		\end{align*}
		i.e., condition (\hyperlink{C.5}{C.5}) is satisfied.\hfill$\qed$
	\end{proof}

	\begin{prop}\label{9.2}
			Let
		 $\bfb:Q_T\times \setR^d\to\setR^d$ be a mapping satisfying (\hyperlink{B.1}{B.1})--(\hyperlink{B.3}{B.3}) with respect to $r:=\max\{2,p_*\}-\vep$. Then, the family of operators ${B(t):X^{q,p}_{\bfvarepsilon}(t)\to X^{q,p}_{\bfvarepsilon}(t)^*}$,~${t\in I}$, for almost every $t\in I$ and every $\bfu,\bfv\in X^{q,p}_{\bfvarepsilon}(t)$ given via
		\begin{align*}
			\langle B(t)\bfu,\bfv\rangle_{X^{q,p}_{\bfvarepsilon}(t)}:=	\langle \mathcal{J}_{\bfvarepsilon}(\bfb(t,\cdot,\bfu),\mathbf{0}),\bfv\rangle_{X^{q,p}_{\bfvarepsilon}(t)}=(\bfb(t,\cdot,\bfu),\bfv)_{L^{q(t,\cdot)}(\Omega)^d},
		\end{align*}
		satisfies (\hyperlink{C.1}{C.1})--(\hyperlink{C.4}{C.4}) and  (\hyperlink{C.6}{C.6}). 
	\end{prop}
	
	\begin{proof}
		\textbf{ad (C.1)\&(C.2)\&(C.4)} 
		Fubini's theorem in conjunction with (\hyperlink{B.1}{B.1}) and (\hyperlink{B.2}{B.2}) implies that  $\bfb(t,\cdot,\cdot)\!:\!\Omega\times \mathbb{R}^d\!\to\! \mathbb{R}^d$ is for almost every $t\!\in\! I$ a Carath\'eodory mapping. We~fix~${t\!\in\! I}$,~such~that $\bfb(t,\cdot,\cdot)\!:\!\Omega\times \mathbb{R}^d\!\to\! \mathbb{R}^d$ is a Carath\'eodory mapping and  (\hyperlink{B.2}{B.2}) holds for~$t$~with~${\eta(t,\cdot)\!\in\! L^{r'(t,\cdot)}(\Omega)}$. Again, let us first prove that ${B(t)\!:\!X^{q,p}_{\bfvarepsilon}(t)\!\to\! X^{q,p}_{\bfvarepsilon}(t)^*}$ is well-defined. To this end,~let~${\bfu\in X^{q,p}_{\bfvarepsilon}(t)}$.  Because $\bfb(t,\cdot,\cdot):\Omega\times \mathbb{R}^d\to \mathbb{R}^d$ is a Carath\'eodory mapping, ${(x\mapsto\bfb(t,x,\bfu(x))):\Omega\to \setR^d}$~is Lebesgue measurable. Moreover, by means of~the~estimates~${(a+b)^{r'(t,x)}\leq 2^{(r^-)'}(a^{r'(t,x)}+b^{r'(t,x)})}$ and $(a+b)^{r(t,x)}\leq 2^{r^+}(a^{r(t,x)}+b^{r(t,x)})$ for all $a,b\ge 0$ and $(t,x)^\top\in Q_T$, also using~(\hyperlink{B.2}{B.2}),~we~derive
		\begin{align}
		\begin{split}
		\rho_{r'(t,\cdot)}(\bfb(\cdot,\cdot,\bfu))&\leq 2^{(r^-)'}\big[\gamma^{(r^-)'}\rho_{r(t,\cdot)}(1+\vert \bfu\vert)+\rho_{r'(t,\cdot)}(\eta(t,\cdot))\big]\\&\leq 2^{(r^-)'}\big[\gamma^{(r^-)'}2^{r^+}\big(\vert \Omega\vert +\rho_{r(t,\cdot)}(\bfu)\big)+\rho_{r'(t,\cdot)}(\eta(t,\cdot))\big].
		\end{split}\label{eq:syn1.1}
		\end{align}
		Thus, $(\bfu\mapsto \bfb(\cdot,\cdot,\bfu)):X^{q,p}_{\bfvarepsilon}(t)\to L^{q'(t,\cdot)}(\Omega)^d$ is well-defined and bounded, because~there~also~holds $\rho_{q'(t,\cdot)}(\bfb(\cdot,\cdot,\bfu))\leq 2^{(r^-)'}(\vert \Omega\vert +\rho_{r'(t,\cdot)}(\bfb(\cdot,\cdot,\bfu)))$ and $\rho_{r(t,\cdot)}(\bfu)\leq 2^{q^+}(\vert \Omega\vert +\rho_{q(t,\cdot)}(\bfu))$ due to $q\ge r$ in $Q_T$, i.e., $r'\ge q'$ in $Q_T$. Therefore, by virtue of Corollary \ref{3.8},  $B(t):X^{q,p}_{\bfvarepsilon}(t)\to X^{q,p}_{\bfvarepsilon}(t)^*$ is well-defined and bounded. 
		\newpage
	   
	   Next, let us show that $B(t):X^{q,p}_{\bfvarepsilon}(t)\to X^{q,p}_{\bfvarepsilon}(t)^*$  is strongly continuous. To this end, let  $(\bfu_n)_{n\in \setN}\subseteq X^{q,p}_{\bfvarepsilon}(t)$ be a sequence, such that ${\bfu_n\weakto \bfu}$~in~$X^{q,p}_{\bfvarepsilon}(t)$~${(n\to \infty)}$. Then, the compact embedding $X^{q,p}_{\bfvarepsilon}(t)\!\embedding\embedding\! L^{p^*(t,\cdot)-\vep}(\Omega)^d$ (cf.~Proposition~\ref{5.12}) yields 
		$\bfu_n\to \bfu $ in $L^{p^*(t,\cdot)-\vep}(\Omega)^d$~${(n\!\to\! \infty)}$. In particular, there exists a cofinal subset $\Lambda\subseteq 
		\setN$,  such that $\bfu_n\to \bfu $ in $\mathbb{R}^d$ $(\Lambda\ni n\to \infty)$ almost everywhere in $\Omega$, which, since $\bfb(t,\cdot,\cdot):\Omega\times \mathbb{R}^d\to \mathbb{R}^d$ is a Carath\'eodory mapping, implies that
		\begin{align}
			\bfb(t,\cdot,\bfu_n)\overset{n\to \infty}{\to }	\bfb(t,\cdot,\bfu)\quad\text{ in }\mathbb{R}^d\quad\text{ a.e. in }\Omega\quad(\Lambda\ni n\to \infty).\label{eq:syn1.2}
		\end{align}
		On the other hand, since $(\bfu_n)_{n\in \setN}\subseteq X^{q,p}_{\bfvarepsilon}(t)$ is also bounded in $Y$, and thus uniformly integrable in $L^{2-\vep}(\Omega)^d$, it is readily seen that $(\bfu_n)_{n\in \setN}\subseteq X^{q,p}_{\bfvarepsilon}(t)$ is also uniformly integrable in $L^{r(t,\cdot)}(\Omega)^d$, where we used that $\max\{2,p_*(t,\cdot)\}\leq \max\{2,p^*(t,\cdot)\}$ in $\Omega$ for almost every $t\in I$, because of $p_*(t,x)\leq p^*(t,x)$, if $p(t,x)\ge \frac{2d}{d+2}$, and 	$p_*(t,x)\leq 2$, if $p(t,x)\leq \frac{2d}{d+2}$, for every $x\in \Omega$.
		 Proceeding as for  \eqref{eq:syn1.1}, we deduce for every Lebesgue measurable set $K\subseteq \Omega$~and~every~${n\in \setN}$~that
		\begin{align}
			\rho_{r'(t,\cdot)}(\bfb(\cdot,\cdot,\bfu_n))\leq 2^{(r^-)'}\big[\gamma^{(r^-)'}2^{r^+}\big(\vert K\vert +\rho_{r(t,\cdot)}(\bfu_n\chi_K)\big)+\rho_{r'(t,\cdot)}(\eta(t,\cdot)\chi_K)\big],\label{eq:syn1.3}
		\end{align}
		i.e., $(\bfb(\cdot,\cdot,\bfu_n))_{n\in \setN}\!\subseteq\! L^{r'(t,\cdot)}(\Omega)^d$ is $L^{r'(t,\cdot)}(\Omega)$--uniformly integrable. Thus, due~to~\eqref{eq:syn1.2}~and~\eqref{eq:syn1.3}, we conclude by
		means of Vitali's convergence theorem (cf.~\cite[Theorem 3.8]{Gut09}) that 
		\begin{align}
			\bfb(\cdot,\cdot,\bfu_n)\overset{n\to \infty}{\to }	\bfb(\cdot,\cdot,\bfu)\quad\text{ in }L^{r'(t,\cdot)}(\Omega)^d\quad(\Lambda\ni n\to \infty).\label{eq:syn1.4}
		\end{align}
		By means of the standard convergence principle (cf.~\cite[Kap. I, Lem. 5.4]{GGZ74}), we further deduce that \eqref{eq:syn1.4} even holds true if $\Lambda=\setN$. Therefore, since $L^{r'(t,\cdot)}(\Omega)^d\embedding L^{q'(t,\cdot)}(\Omega)^d$ (cf.~\eqref{embed}), because ${q\ge r}$ in $Q_T$, i.e., $r'\ge q'$ in $Q_T$, we conclude from \eqref{eq:syn1.4} with $\Lambda=\setN$ and Corollary \ref{3.8} that
		$B(t)\bfu_n\to B(t)\bfu$ in $X^{q,p}_{\bfvarepsilon}(t)^*$ $(n\to \infty)$, i.e., $B(t):X^{q,p}_{\bfvarepsilon}(t)\to X^{q,p}_{\bfvarepsilon}(t)^*$ is strongly continuous, and thus also pseudo-monotone (cf.~\cite[Kap. III., Lem. 2.6 (ii)]{Ru04}). As a whole, we verified both condition (\hyperlink{C.1}{C.1})~and~condition~(\hyperlink{C.4}{C.4}). For  (\hyperlink{C.2}{C.2}) we argue as in the proof of Proposition~\ref{9.1}.

		\textbf{ad (C.3)\vspace*{-0.1pt}\&\vspace*{-0.1pt}(C.6)} Applying the $\vep$--Young inequality pointwise for almost every ${(t,x)^\top \in Q_T}$ with respect to the exponent $r(t,x)$, with the constant  ${c_{r(t,x)}(\tilde{\vep}):=(r'(t,x)\tilde{\vep})^{1-r(t,x)}r(t,x)^{-1}}$ for every  $\tilde{\vep}\in \left(0,\tilde{\vep}_0\right)$, where $\tilde{\vep}_0:=((r^+)')^{-1}$, also using that $c_{r(t,x)}(\tilde{\vep})\leq c_r(\tilde{\vep}):=((r^+)'\tilde{\vep})^{1-r^+}(r^-)^{-1}$ for every $\tilde{\vep}\in \left(0,\tilde{\vep}_0\right)$, using \eqref{eq:syn1.1} and Lemma~\ref{5.11}, we obtain for almost every $t\in I$, every $\tilde{\vep}\in \left(0,\tilde{\vep}_0\right)$ and $\bfu,\bfv\in X^{q,p}_{\bfvarepsilon}(t)$ the estimate
		\begin{align*}
			\vert \langle B(t)\bfu,&\bfv\rangle_{X^{q,p}_{\bfvarepsilon}(t)}\vert \leq \tilde{\vep}\rho_{r(t,\cdot)}(\bfb(t,\cdot,\bfu))+c_r(\tilde{\vep})\rho_{r(t,\cdot)}(\bfv)
			\\&\leq \tilde{\vep} 2^{(r^-)'}\big[\gamma^{(r^-)'}2^{r^+}\big(\vert\Omega\vert+\rho_{r(t,\cdot)}(\bfu)\big)+\rho_{r(t,\cdot)}(\eta(t,\cdot))\big]+c_r(\tilde{\vep})\rho_{r(t,\cdot)}(\bfv)
			\\&
			\leq \tilde{\vep}2^{(r^-)'}\big[\gamma^{(r^-)'}2^{r^+}\big(\vert\Omega\vert+c_\vep\big[1+\rho_{p(t,\cdot)}(\bfvarepsilon(\bfu))+\|\bfu\|_{Y}^{\gamma_\vep}\big]\big(1+\|\bfu\|_{Y}^{\gamma_\vep}\big)\big)+\rho_{r(t,\cdot)}(\eta(t,\cdot))\big]\\&\quad+c_q(\tilde{\vep})c_\vep\big[1+\rho_{p(t,\cdot)}(\bfvarepsilon(\bfv))+\|\bfv\|_{Y}^{\gamma_\vep}\big]\big(1+\|\bfv\|_{Y}^{\gamma_\vep}\big),
		\end{align*}
		i.e., condition (\hyperlink{C.3}{C.3}), which follows from the second inequality, since  ${\rho_{r(t,\cdot)}(\bfu)\!\leq\! 2^{q^+}\!(\vert\Omega\vert\!+\!\rho_{q(t,\cdot)}(\bfu))}$ for  every $t\in I$ and $\bfu\in X^{q,p}_{\bfvarepsilon}(t)$, and condition (\hyperlink{C.6}{C.6}), which is just the last inequality.\hfill$\qed$
\end{proof}

	\begin{prop}\label{9.3}
		Let $\bfS:Q_T\times \mathbb{M}^{d\times d}_{\sym}\to\mathbb{M}^{d\times d}_{\sym}$ be a mapping satisfying (\hyperlink{S.1}{S.1})--(\hyperlink{S.4}{S.4}) with respect to $p$  and $\bfb:Q_T\times \setR^d\to \setR^d$ a mapping satisfying (\hyperlink{B.1}{B.1})--(\hyperlink{B.3}{B.3}) with respect to $r:=\max\{2,p_*\}-\vep$. Then, the sum $S(t)+B(t):X^{q,p}_{\bfvarepsilon}(t)\to X^{q,p}_{\bfvarepsilon}(t)^*$, $t\in I$, satisfies (\hyperlink{C.1}{C.1})--(\hyperlink{C.6}{C.6}). In particular, the induced operator $\bscal{S}+\bscal{B}:\bscal{X}^{q,p}_{\bfvarepsilon}(Q_T)\cap \bscal{Y}^\infty(Q_T)\to \bscal{X}^{q,p}_{\bfvarepsilon}(Q_T)^*$, given via \eqref{eq:induced}, is bounded, Bochner pseudo-monotone and Bochner coercive with respect to all $\bfu_0\in Y$ and $\bscal{J}_{\bfvarepsilon}(\bsf,\bsF)\in \bscal{X}^{q,p}_{\bfvarepsilon}(Q_T)^*$, where $\bsf\in L^{\min\{2,(p^-)'\}}(Q_T)^d$ and $\bsF\in L^{p'(\cdot,\cdot)}(Q_T,\mathbb{M}^{d\times d}_{\sym})$.
	\end{prop}
	
	\begin{proof}
		Due to Proposition \ref{9.1} and Proposition \ref{9.2}, ${S(t),B(t):X^{q,p}_{\bfvarepsilon}(t)\to X^{q,p}_{\bfvarepsilon}(t)^*}$,~${t\in I}$, satisfy (\hyperlink{C.1}{C.1})--(\hyperlink{C.4}{C.4}) and (\hyperlink{C.6}{C.6}). Thus, also  ${S(t)+B(t):X^{q,p}_{\bfvarepsilon}(t)\to X^{q,p}_{\bfvarepsilon}(t)^*}$,~${t\in I}$, satisfies (\hyperlink{C.1}{C.1})--(\hyperlink{C.4}{C.4}) and (\hyperlink{C.6}{C.6}). Finally, as  ${S(t):X^{q,p}_{\bfvarepsilon}(t)\to X^{q,p}_{\bfvarepsilon}(t)^*}$,~${t\in I}$, satsifies  (\hyperlink{C.5}{C.5}) and ${B(t)\!:\!X^{q,p}_{\bfvarepsilon}(t)\!\to\! X^{q,p}_{\bfvarepsilon}(t)^*}$, ${t\in I}$, due to (\hyperlink{B.3}{B.3}), satisfies $\langle B(t)\bfu,\bfu\rangle_{X^{q,p}_{\bfvarepsilon}(t)}\ge -c_2\|\bfu\|_Y^2-\|c_3(t,\cdot)\|_{L^1(\Omega)}$ for almost every $t\in I$ and every $\bfu\in X^{q,p}_{\bfvarepsilon}(t)$, ${S(t)+B(t):X^{q,p}_{\bfvarepsilon}(t)\to X^{q,p}_{\bfvarepsilon}(t)^*}$,~${t\in I}$, also satisfies (\hyperlink{C.5}{C.5}). The remaining part of the assertion results from Proposition~\ref{7.1} and Proposition~\ref{7.2}.\vspace*{-0.75cm} 
	\end{proof}
	\newpage
	
	\begin{prop}
		Let $\bfS:Q_T\times \mathbb{M}^{d\times d}_{\sym}\to\mathbb{M}^{d\times d}_{\sym}$ be a mapping satisfying (\hyperlink{S.1}{S.1})--(\hyperlink{S.4}{S.4}) with respect to $p$  and $\bfb:Q_T\times \setR^d\to \setR^d$ a mapping satisfying (\hyperlink{B.1}{B.1})--(\hyperlink{B.3}{B.3})~with~respect~to~${r:=\max\{2,p_*\}-\vep}$. Then, for arbitrary $\bfu_0\in Y$ and $\bsu^*:=\bscal{J}_{\bfvarepsilon}(\bsf,\bsF)\in \bscal{X}^{q,p}_{\bfvarepsilon}(Q_T)^*$, where $\bsf\in L^{\min\{2,(p^-)'\}}(Q_T)^d$ and $\bsF\in L^{p'(\cdot,\cdot)}(Q_T,\mathbb{M}^{d\times d}_{\sym})$, there exists a solution $\bsu\in \bscal{W}^{q,p}_{\bfvarepsilon}(Q_T)$~with~a~representation~${\bsu_c\in \bscal{Y}^0(Q_T)}$ of
		\begin{alignat*}{2}
			\frac{\bfd\bsu}{\bfd\bft}+\bscal{S}\bsu+\bscal{B}\bsu&=\bsu^*&&\quad\text{ in }\bscal{X}^{q,p}_{\bfvarepsilon}(Q_T)^*,\\
			\bsu_c(0)&=\bfu_0&&\quad\text{ in }Y, 
		\end{alignat*}
		or equivalently, for every $\bfphi\in C^\infty(\overline{Q_T})^d$ with $\textup{supp}(\bfphi)\subset\subset \left[0,T\right)\times \Omega$ there holds
		\begin{align*}
			-\int_{Q_T}{\bsu(t,x)\cdot\partial_t\bfphi(t,x)\,dtdx}&+\int_{Q_T}{\bfb(t,x,\bsu(t,x))\cdot\bfphi(t,x)+\bfS(t,x,\bfvarepsilon(\bsu)(t,x)):\bfvarepsilon(\bfphi)(t,x)\,dtdx}
			\\&=(\bfu_0,\bfphi(0))_Y+\int_{Q_T}{\bsf(t,x)\cdot\bfphi(t,x)+\bsF(t,x):\bfvarepsilon(\bfphi)(t,x)\,dtdx}.
		\end{align*}
	\end{prop}

	\begin{proof}
		 Follows from Theorem~\ref{main}, or Corollary~\ref{main2}, in conjunction of Propositon~\ref{9.3}.~\hfill$\qed$
	\end{proof}

\bibliographystyle{my-amsplain}    
\bibliography{literatur2} 
\end{document}